\documentclass[smallextended]{article}

\usepackage{graphicx}
 \usepackage{mathptmx}      
\usepackage{latexsym}
\usepackage{amssymb}
\usepackage{amsmath}
\usepackage{amsthm}
\usepackage[colorlinks,citecolor=blue,urlcolor=blue]{hyperref}
\usepackage[square,numbers,colon]{natbib}

\usepackage{multirow}

\usepackage[T1]{fontenc}
\usepackage[utf8]{inputenc}

\newtheorem{thm}{Theorem}
\newtheorem{cor}{Corollary}
\newtheorem{lem}{Lemma}
\newtheorem{rem}{Remark}
\newtheorem{result}{Result}

\begin{document}

\title{Functional Central Limit Theorems for Conditional Poisson sampling Designs\thanks{This work was supported by the grant 2016-ATE-0459 and the grant 2017-ATE-0402 from Università degli Studi di Milano-Bicocca.}
}


\author{Leo Pasquazzi\footnote{Dipartimento di Statistica e Metodi Quantitativi, Università degli Studi di Milano-Bicocca, Edificio U7, Via Bicocca degli Arcimboldi 8,	20126 – Milano}}




\maketitle

\begin{abstract}
This paper provides refined versions of some known functional central limit theorems for conditional Poisson sampling which are more suitable for applications. The theorems presented in this paper are generalizations of some results that have been recently published by \citet*{Bertail_2017}. The asymptotic equicontinuity part of the proofs presented in this paper is based on the same idea as in \citep{Bertail_2017} but some of the missing details are provided. On the way to the functional central limit theorems, this paper provides a detailed discussion of what must be done in order to prove conditional and unconditional weak convergence in bounded function spaces in the context of survey sampling. The results from this discussion can be useful to prove further weak convergence results.
\vspace{0.2cm}


\vspace{0.2cm}

\end{abstract}

\section{Introduction}

\citet*{Bertail_2017} have recently published a paper where they proposed some FLCTs for Poisson sampling designs as well as for conditional Poisson sampling designs (henceforth CPS designs or rejective sampling designs). The author of the present paper has already published a draft manuscript which provides quite substantial generalizations of the results for the Poisson sampling case (see \cite{Pasquazzi_2019}). In fact, \cite{Bertail_2017} considers only empirical processes indexed by function classes which satisfy the uniform entropy condition, while \cite{Pasquazzi_2019} extends these results to arbitrary Donsker classes with uniformly bounded means. The proofs of the more general results given in \cite{Pasquazzi_2019} are based on the symmetrization technique and they differ substantially from those given in \cite{Bertail_2017} which use the Hoeffding inequality (see \cite{Hoeffding_1963}). Unfortunately, the symmetrization trick cannot be applied in the conditional Poisson sampling case which prevents to generalize the weak convergence results for the conditional Poisson sampling case given in \cite{Bertail_2017} by using the symmetrization technique as in \cite{Pasquazzi_2019}. Anyway, the results given in \cite{Bertail_2017} are somewhat unsatisfactory because the assumptions about the sequence of conditional Poisson sampling designs are unnecessarily restrictive. In fact, perhaps in order to simplify the proofs, in \cite{Bertail_2017} it is assumed that the first order sample inclusion probabilities of the underlying (approximately canonical) Poisson sampling designs are realizations of i.i.d. random variables which are bounded away from zero. As a consequence, the assumptions of the theorems presented in \cite{Bertail_2017} imply that the sequence of sample sizes of the rejective sampling designs must be random and moreover the theorems cannot be applied to cases where there is dependence among the first order sample inclusion probabilities, or to cases where the sample inclusion probabilities are proportional to some size variable which can take on values arbitrarily close to zero. The results given in the present paper overcome these shortcomings. 

This work is organized as follows. Section \ref{Section_notation_definitions} introduces the probabilistic framework within which the FCLTs will be derived. The probabilistic model and all other definitions given in Section \ref{Section_notation_definitions} are identical to those given in Section 2 of \cite{Pasquazzi_2019}. Section \ref{general_weak_convergence_theory} provides some general definitions and theorems which are very useful for showing conditional weak convergence results in the context of survey sampling. The definitions and theorems provided in this section are conditional analogues of the definitions and theorems given in Chapter 1.5 on pages 34 - 41 in \citep{vdVW}. They are completely general and can be used to prove other weak convergence results in the context of survey sampling as well. Section \ref{rejective_sampling_review} reviews the relevant conditional Poisson sampling theory which is due to H{\'a}jek (see \cite{Hajek_1964}). In Section \ref{main_results} the FCLTs for the conditional Poisson sampling case will be derived. Section \ref{Hajek_empirical_process_theory_POISSON} provides extensions for the H{\'a}jek empirical process and Section \ref{Simulation_results} concludes this work with a simulation study.


\section{Notation and Definitions}\label{Section_notation_definitions}

Let $Y_{1}$, $Y_{2}$, \dots, $Y_{N}$ denote the values taken on by a study variable $Y$ on the $N$ units of a finite population and let $X_{1}$, $X_{2}$, \dots, $X_{N}$ denote corresponding values of an auxiliary variable $X$. In this paper it will be assumed that the $N$ ordered pairs $(Y_{i}, X_{i})$ corresponding to a given finite population of interest are the first $N$ realizations of an infinite sequence of i.i.d. random variables which take on values in the cartesian product of two measurable spaces which will be denoted by $(\mathcal{Y},\mathcal{A})$ and $(\mathcal{X},\mathcal{B})$, respectively. Moreover, as usual in finite population sampling theory, it will be assumed that the values taken on by the auxiliary variable $X$ are known in advance for all the $N$ population units, while the values taken on by the study variable $Y$ are only known for the population units that have been selected into a random sample. The corresponding vector of sample inclusion indicator functions will be denoted by $\mathbf{S}_{N}:=(S_{1,N}, S_{2,N}, \dots, S_{N,N})$ and it will be assumed that the vectors $\mathbf{S}_{N}$ and $\mathbf{Y}_{N}:=(Y_{1}, Y_{2}, \dots, Y_{N})$ are \textit{conditionally} independent \textit{given} $\mathbf{X}_{N}:=(X_{1}, X_{2}, \dots, X_{N})$. With reference to the sample design, probability and expectation will be denoted by $P_{d}$ e $E_{d}$, respectively. With this notation, the vector of first order sample inclusion probabilities will be given by
\begin{equation*}
\begin{split}
\underline{\mathbf{\pi}}_{N}&:=(\pi_{1,N}, \pi_{2,N}, \dots, \pi_{N,N})\\
&:=(E_{d}S_{1,N}, E_{d}S_{2,N}, \dots, E_{d}S_{N,N})\\
&=(P_{d}\{S_{1,N}=1\}, P_{d}\{S_{2,N}=1\},\dots, P_{d}\{S_{N,N}=1\}),
\end{split}
\end{equation*}
and from the conditional independence assumption it follows that $\underline{\mathbf{\pi}}_{N}$ must be a deterministic function of $\mathbf{X}_{N}$.

\smallskip

Now, with reference to the measurable space $(\mathcal{Y},\mathcal{A})$, consider the random empirical measure given by
\begin{equation}\label{HT_empirical_process}
\mathbb{G}_{N}':=\frac{1}{\sqrt{N}}\sum_{i=1}^{N}\left(\frac{S_{i,N}}{\pi_{i,N}}-1\right)\delta_{Y_{i}}.
\end{equation}
For a given $f:\mathcal{Y}\mapsto\mathbb{R}$, the integral of $f$ with respect to $\mathbb{G}_{N}'$ can be written as
\[\mathbb{G}_{N}'f:=\frac{1}{\sqrt{N}}\sum_{i=1}^{N}\left(\frac{S_{i,N}}{\pi_{i,N}}-1\right)f(Y_{i})\]
so that, for any given class $\mathcal{F}$ of functions $f:\mathcal{Y}\mapsto\mathbb{R}$, the random empirical measure $\mathbb{G}_{N}'$, as a real-valued function of $f\in\mathcal{F}$, can be interpreted as a stochastic process indexed by the set $\mathcal{F}$. For obvious reasons $\mathbb{G}_{N}'$ will be called \textbf{Horvitz-Thompson empirical process} (henceforth \hypertarget{HTEP}{\textbf{HTEP}}). Depending on the values taken on by the study variable $Y$ and on the class of functions $\mathcal{F}$, a sample path of $\mathbb{G}_{N}'$ could be either bounded or not. In the former case it will be an element of $l^{\infty}(\mathcal{F})$, the space of all bounded and real-valued functions with domain given by the class of functions $\mathcal{F}$. In what follows $l^{\infty}(\mathcal{F})$ will be considered as a metric space with distance function induced by the norm $\lVert z\rVert_{\mathcal{F}}:=\sup_{f\in\mathcal{F}}|z(f)|$.

\smallskip

As already mentioned in the introduction, the present paper provides FCLTs for conditional Poisson sampling designs (or rejective sampling designs). To be precise, the present paper investigates conditions under which
\[\mathbb{G}_{N}'\rightsquigarrow \mathbb{G}'\text{ in }l^{\infty}(\mathcal{F}),\]
where $\mathbb{G}'$ is a Borel measurable and tight (in $l^{\infty}(\mathcal{F})$) Gaussian process. Both \textit{unconditional} and \textit{conditional} (on the realized values of $X$ and $Y$) weak convergence will be considered. Recall that unconditional weak convergence is defined as
\[E^{*}h(\mathbb{G}_{N}')\rightarrow Eh(\mathbb{G}')\quad\text{ for all }h\in C_{b}(l^{\infty}(\mathcal{F})),\]
where $C_{b}(l^{\infty}(\mathcal{F}))$ is the class of all real-valued and bounded functions on $l^{\infty}(\mathcal{F})$. If the realizations of $\mathbb{G}'$ lie in a separable subset of $l^{\infty}(\mathcal{F})$ almost surely, this is equivalent to
\[\sup_{h\in BL_{1}(l^{\infty}(\mathcal{F}))}\left|E^{*}h(\mathbb{G}_{N}')-Eh(\mathbb{G}')\right|\rightarrow 0,\]
where $BL_{1}(l^{\infty}(\mathcal{F}))$ is the set of all functions $h:l^{\infty}(\mathcal{F})\mapsto[0,1]$ such that $|h(z_{1}-h(z_{2})|\leq \lVert z_{1}-z_{2}\rVert_{\mathcal{F}}$ for every $z_{1},z_{2}\in l^{\infty}(\mathcal{F})$ (see Chapter 1.12 in \cite{vdVW}). Based on this observation, \cite{vdVW} provides two definitions of conditional weak convergence: \textbf{conditional weak convergence in outer probability} (henceforth \hypertarget{opCWC}{\textbf{opCWC}}), which in the context of this paper translates to the condition
\[\sup_{h\in BL_{1}(l^{\infty}(\mathcal{F}))}\left|E_{d}h(\mathbb{G}_{N}')-Eh(\mathbb{G}')\right|\overset{P^{*}}{\longrightarrow} 0\]
(see page 181 in \cite{vdVW}), and \textbf{outer almost sure conditional weak convergence} (henceforth \hypertarget{oasCWC}{\textbf{oasCWC}}), which in the context of this paper translates to the condition
\[\sup_{h\in BL_{1}(l^{\infty}(\mathcal{F}))}\left|E_{d}h(\mathbb{G}_{N}')-Eh(\mathbb{G}')\right|\overset{as*}{\rightarrow} 0.\]
As expected, \hyperlink{oasCWC}{oasCWC} implies \hyperlink{opCWC}{opCWC} (see Lemma 1.9.2 on page 53 in \cite{vdVW}). However,  it seems that \hyperlink{oasCWC}{oasCWC} is not strong enough to imply asymptotic measurability (cfr. Theorem 2.9.6 on page 182 in \cite{vdVW} and the comments thereafter) which is a necessary condition for unconditional weak convergence (see Lemma 1.3.8 on page 21 in \cite{vdVW}).

\smallskip

Since the very definition of weak convergence relies on the concept of outer expectation, some assumptions about the underlying probability space will be necessary for what follows. Throughout this paper it will be assumed that the latter is a product space of the form
\begin{equation}\label{probability_space}
\prod_{i=1}^{\infty} (\Omega_{y,x}, \mathcal{A}_{y,x}, P_{y,x})\times (\Omega_{d},\mathcal{A}_{d},P_{d})
\end{equation}
and that the elements of the random sequence $\{(Y_{i},X_{i})\}_{i=1}^{\infty}$ are the coordinate projections on the first infinite coordinates of the sample points $\omega\in\Omega_{y,x}^{\infty}\times\Omega_{d}$. On the other hand, the sample inclusion indicators $S_{i,N}$ are allowed to depend on all the coordinates. As suggested by the notation, it will be assumed that for each value of $N$ the corresponding sample inclusion indicator functions $S_{1,N}$, $S_{2,N}$, \dots, $S_{N,N}$ are the elements of one row of a triangular array of random variables. This assumption is needed in order to make sure that for each value of $N$ the sample design can be readapted according to all the $N$ (known) values taken on by the auxiliary variable $X$ as the population size increases. To make sure that the conditional independence assumption holds, it will be assumed that for each value of $N$ the corresponding vector $\mathbf{S}_{N}$ is defined as a function of the random vector $\mathbf{X}_{N}$ and of random variables $D_{1}$, $D_{2}$, \dots which are functions of the last coordinate of the sample points $\omega\in\Omega_{y,x}^{\infty}\times\Omega_{d}$ only, i.e. of the coordinate that takes on values in the set $\Omega_{d}$ (instead of a random sequence $\{D_{i}\}_{i=1}^{\infty}$ one could also consider a stochastic process $\{D_{t}: t\in T\}$ with an arbitrary index set $T$ but this will not be of interest in the present paper). For example, in the case of a Poisson sampling design with a given vector of first order sample inclusion probabilities $\underline{\mathbf{\pi}}_{N}$ (which could be a function of $\mathbf{X}_{N}$) one could define $\{D_{i}\}_{i=1}^{\infty}$ as a sequence of i.i.d. uniform-$[0,1]$ random variables and define for each value of $N$ the corresponding row of sample inclusion indicators by
\[S_{i,N}:=\begin{cases}
1&\text{if }D_{i}\leq \pi_{i,N}\\
0 & \text{ otherwise}
\end{cases}
\quad\quad i=1,2,\dots, N\]
Of course, the above probability space does not only work for Poisson sampling designs, but it can accommodate any non-informative sampling design. In fact, it is not difficult to show that \textit{for any non-informative sampling design the vector of sample inclusion indicators $\mathbf{S}_{N}$ can be defined as a function of $\mathbf{X}_{N}$ and of a single uniform-$[0,1]$ random variable $D$ that depends on the last coordinate of the sample points $\omega\in\Omega_{y,x}^{\infty}\times\Omega_{d}$ only}. To this aim let
\begin{equation}\label{sample_selection_probability}
\mathfrak{p}_{N}(\mathbf{s}_{N}):=\mathfrak{p}_{N}(\mathbf{s}_{N}; \mathbf{X}_{N})
\end{equation}
denote the probability to select a given sample $\mathbf{s}_{N}\in\{0,1\}^{N}$. Note that the definition of the function $\mathfrak{p}_{N}$ specifies a desired sampling design. Since the values taken on by the auxiliary variable $X$ are assumed to be already known before the sample is drawn, the sample selection probabilities $\mathfrak{p}_{N}(\mathbf{s}_{N})$ are allowed to depend on $\mathbf{X}_{N}$. Now, let $\mathbf{s}_{N}^{(1)}$, $\mathbf{s}_{N}^{(2)}$, \dots, $\mathbf{s}_{N}^{(2^{N})}$ denote the $2^{N}$ elements of $\{0,1\}^{N}$ arranged in some fixed order (for example, according to the order determined by the binary expansion corresponding to the finite sequence of zeros and ones in $\mathbf{s}_{N}$), and put $\mathfrak{p}_{N}^{(i)}:=\mathfrak{p}_{N}(\mathbf{s}_{N}^{(i)})$, $i=1,2,\dots, 2^{N}$. Then, define the vector of sample inclusion indicators $\mathbf{S}_{N}$ by
\[\mathbf{S}_{N}:=\begin{cases}
\mathbf{s}_{N}^{(1)}\quad\text{ if }D\leq \mathfrak{p}_{N}^{(1)}\\
\mathbf{s}_{N}^{(i)}\quad\text{ if }\sum_{j=1}^{i-1}\mathfrak{p}_{N}^{(j)}<D\leq\sum_{j=1}^{i}\mathfrak{p}_{N}^{(j)}\text{ for } i=2,3,\dots, 2^{N},
\end{cases}\]
and note that for every $\mathbf{s}_{N}\in\{0,1\}^{N}$ this vector satisfies $P_{d}\{\mathbf{S}_{N}=\mathbf{s}_{N}\}=\mathfrak{p}_{N}(\mathbf{s}_{N})$ as desired. This concludes the proof of the above assertion written in italics. 

Next, observe that in the above construction the sample selection probabilities $P_{d}\{\mathbf{S}_{N}=\mathbf{s}_{N}\}$ are functions of $\mathbf{X}_{N}$. If for a given $\mathbf{s}_{N}\in\{0,1\}^{N}$ the corresponding sample selection probability $P_{d}\{\mathbf{S}_{N}=\mathbf{s}_{N}\}$ is a measurable function of $\mathbf{X}_{N}\in\mathcal{X}^{N}$ (this depends on the sampling design), then, with reference to the probability space of this paper, $P_{d}\{\mathbf{S}_{N}=\mathbf{s}_{N}\}$ can be interpreted as a conditional probability in the proper sense. Otherwise, $P_{d}\{\mathbf{S}_{N}=\mathbf{s}_{N}\}$ will just be a non measurable (random) function of $\mathbf{X}_{N}$. More generally, the expectation with respect to the uniform random variable $D$ with $\mathbf{Y}_{N}$ and $\mathbf{X}_{N}$ kept fixed, which can be interpreted as design expectation and will therefore be denoted by $E_{d}$, can be applied to \textit{any} function $g$ of $\mathbf{S}_{N}$, $\mathbf{Y}_{N}$ and $\mathbf{X}_{N}$. In fact, the expectation
\[E_{d}g(\mathbf{S}_{N}, \mathbf{Y}_{N},\mathbf{X}_{N})\]
is given by 
\[\sum_{\mathbf{s}_{N}\in\{0,1\}^{N}}g(\mathbf{s}_{N}, \mathbf{Y}_{N},\mathbf{X}_{N})P_{d}\{\mathbf{S}_{N}=\mathbf{s}_{N}\},\]
and $E_{d}g(\mathbf{S}_{N}, \mathbf{Y}_{N},\mathbf{X}_{N})$ is thus a function of $\mathbf{Y}_{N}$ and $\mathbf{X}_{N}$. If for every fixed $\mathbf{s}_{N}\in\{0,1\}^{N}$ the corresponding function $g(\mathbf{s}_{N},\cdot)$ is a measurable function of $\mathbf{Y}_{N}$ and $\mathbf{X}_{N}$ and the function $P_{d}\{\mathbf{S}_{N}=\mathbf{s}_{N}\}$ is a measurable function of $\mathbf{X}_{N}$, then, with respect  to the probability space of this paper, $E_{d}g(\mathbf{S}_{N}, \mathbf{Y}_{N},\mathbf{X}_{N})$ can be interpreted as a conditional expectation in the proper sense (and in this case it will obviously be a measurable function of $\mathbf{Y}_{N}$ and $\mathbf{X}_{N}$), while otherwise it could either be a measurable or a non measurable function of $\mathbf{Y}_{N}$ and $\mathbf{X}_{N}$.

Throughout this paper it will be assumed that all the vectors of sample inclusion indicators $\mathbf{S}_{N}$ are defined as described in the above construction (the one which involves a single uniform-$[0,1]$ random variable $D$). Of course, in this way the random vectors $\mathbf{S}_{N}$ will be dependent for different values of $N$, but for the purposes of this paper this dependence structure is irrelevant. Moreover, in what follows only \textit{measurable sample designs} will be considered, i.e. sample designs such that for every fixed $\mathbf{s}_{N}\in\{0,1\}^{N}$ the corresponding sample selection probability in (\ref{sample_selection_probability}) is a measurable function of $\mathbf{X}_{N}$. Note that this is a very mild restriction that should be satisfied in virtually every practical setting. However, it entails three important consequences which will be relevant for the proofs presented in this paper. They are: (i) the vectors of sample inclusion indicators $\mathbf{S}_{N}$ are measurable functions of $\mathbf{X}_{N}$ and of the uniform-$[0,1]$ random variable $D$, (ii) for every $\mathbf{s}_{N}\in\{0,1\}^{N}$ the corresponding probability $P_{d}\{\mathbf{S}_{N}=\mathbf{s}_{N}\}$ is a conditional probability in the proper sense, and (iii) for $g$ a measurable function of $\mathbf{S}_{N}$, $\mathbf{Y}_{N}$ and $\mathbf{X}_{N}$ the corresponding expectation $E_{d}g(\mathbf{S}_{N}, \mathbf{Y}_{N}, \mathbf{X}_{N})$ is a conditional expectation in the proper sense.

\section{Weak convergence in bounded function spaces in the context of survey sampling}\label{general_weak_convergence_theory}

This section contains a detailed discussion of general methods for proving unconditional and conditional weak convergence in bounded function spaces in the context of survey sampling. Throughout this section it will be assumed that $\mathcal{F}$ is an arbitrary set and that $\{\mathbb{H}_{N}'\}_{N=1}^{\infty}$ is a sequence of mappings from the probability space (\ref{probability_space}) into $l^{\infty}(\mathcal{F})$ where each $\mathbb{H}_{N}'$ depends on the sample points $\omega\in\Omega_{y,x}^{\infty}\times\Omega_{d}$ only through $\mathbf{Y}_{N}$, $\mathbf{X}_{N}$ and $\mathbf{S}_{N}$. Moreover, it will be assumed that for every $N=1,2,\dots$ and for every $f\in\mathcal{F}$ the corresponding coordinate projection $\mathbb{H}_{N}'f$ is measurable. The scope of this section is to provide necessary and sufficient  conditions for \hyperlink{opCWC}{opCWC} (\hyperlink{oasCWC}{oasCWC}) in $l^{\infty}(\mathcal{F})$, i.e. for
\begin{equation}\label{CWC_H}
\sup_{h\in BL_{1}(l^{\infty}(\mathcal{F}))}\left|E_{d}h(\mathbb{H}_{N}')-Eh(\mathbb{H}')\right|\overset{P*(as*)}{\rightarrow} 0
\end{equation}
with $\mathbb{H}'$ a Borel measurable and tight mapping from some probability space into $l^{\infty}(\mathcal{F})$. In order to avoid repetitions, in what follows the symbols $\overset{P(as)}{\rightarrow}$ and $\overset{P*(as*)}{\rightarrow}$ will be used in order to express two versions of a convergence condition. This notation will often appear in the assumptions and in the conclusions of lemmas, theorems and corollaries. In these cases it is understood that the "probability convergence versions" of the assumptions imply the "probability convergence versions" of the conclusions and that the "almost sure convergence versions" of the assumptions imply the "almost sure convergence versions" of the conclusions.

What needs to be done in order to prove \hyperlink{opCWC}{opCWC} (\hyperlink{oasCWC}{oasCWC}) seems to be clear from the general (unconditional) weak convergence theory laid out in \cite{vdVW}. In fact, according to Theorem 1.5.4 on page 35 in \cite{vdVW}, if the sequence $\{\mathbb{H}_{N}'\}_{N=1}^{\infty}$ is asymptotically tight and all the finite dimensional marginals $(\mathbb{H}_{N}'f_{1}, \mathbb{H}_{N}'f_{2}, \dots, \mathbb{H}_{N}'f_{r})$ converge weakly (in $\mathbb{R}^{r}$) to the corresponding marginals of some stochastic process $\{\mathbb{H}'f:f\in\mathcal{F}\}$, then there exists a version of $\mathbb{H}'$ which is a Borel measurable and tight mapping from some probability space into $l^{\infty}(\mathcal{F})$ such that
\begin{equation}\label{WC_H_easy_notation}
\mathbb{H}_{N}'\rightsquigarrow\mathbb{H}'\text{ in }l^{\infty}(\mathcal{F}).
\end{equation}
Since the realizations of $\mathbb{H}'$ lie in a separable subset of $l^{\infty}(\mathcal{F})$ almost surely (this follows from tightness), it follows that condition (\ref{WC_H_easy_notation}) is equivalent to 
\begin{equation}\label{WC_H}
\sup_{h\in BL_{1}(l^{\infty}(\mathcal{F}))}\left|Eh(\mathbb{H}_{N}')-Eh(\mathbb{H}')\right|\rightarrow 0
\end{equation}
(see the comments at the top of page 73 in \cite{vdVW}). On the other hand, Theorem 1.5.4 on page 35 in \cite{vdVW} says also that if $\mathbb{H}'$ is a Borel measurable and tight mapping from some probability space into $l^{\infty}(\mathcal{F})$ and if condition (\ref{WC_H_easy_notation}) or equivalently condition (\ref{WC_H}) holds, then it must be necessarily true that the sequence $\{\mathbb{H}_{N}'\}_{N=1}^{\infty}$ is asymptotically tight and that its finite-dimensional marginals converge weakly to the corresponding marginals of $\mathbb{H}'$.

Since the only difference between condition (\ref{WC_H}) and condition (\ref{CWC_H}) is the fact that the unconditional expectation $Eh(\mathbb{H}_{N}')$ is replaced by the sample design expectation $E_{d}h(\mathbb{H}_{N}')$ (which is not necessarily a conditional expectation in the proper sense), one would expect that \hyperlink{opCWC}{opCWC} (\hyperlink{oasCWC}{oasCWC}) is equivalent to the joint occurrence of some form of \textit{conditional asymptotic tightness} and of some form of \textit{conditional weak convergence of the finite-dimensional marginals}. In this section it will be shown that this is indeed true. The first step towards this goal is to provide a clear definition of what "conditional asymptotic tightness" and "conditional weak convergence of the finite-dimensional marginals" mean. To this aim recall that according to Definition 1.3.7 on pages 20-21 in \cite{vdVW} the sequence $\{\mathbb{H}_{N}'\}_{N=1}^{\infty}$ is asymptotically tight in the usual unconditional sense if for every $\eta>0$ there exists a compact set $K\subset l^{\infty}(\mathcal{F})$ such that
\[\liminf_{N\rightarrow\infty}P_{*}\{\mathbb{H}_{N}'\in K^{\delta}\}\geq 1-\eta\quad\text{ for every }\delta>0,\]
where $K^{\delta}:=\{z\in l^{\infty}(\mathcal{F}):\lVert z-z'\rVert_{\mathcal{F}}<\delta\text{ for some }z'\in K\}$. Of course this condition is satisfied if and only if for every $\eta>0$ there exists a compact set $K\subset l^{\infty}(\mathcal{F})$ such that for every $\delta>0$ there exists a sequence of real numbers $A_{N}\rightarrow 0$ for which
\[P_{*}\{\mathbb{H}_{N}'\in K^{\delta}\}+A_{N}\geq 1-\eta\quad\text{ for every }N=1,2,\dots.\]
Based on this observation we can define two versions of \textbf{conditional asymptotic tightness} (henceforth \hypertarget{CAT}{\textbf{CAT}}): a \textit{probability version} by requiring that for every $\eta>0$ there exists a compact set $K\subset l^{\infty}(\mathcal{F})$ such that for every $\delta>0$ there exists a sequence of random variables $\widetilde{A}_{N}\overset{P}{\rightarrow}0$ for which
\begin{equation}\label{conditional_asy_tightness}
P_{d}\{\mathbb{H}_{N}'\in K^{\delta}\}+\widetilde{A}_{N}\geq 1-\eta\quad\text{ for every }N=1,2,\dots,
\end{equation}
and an \textit{almost sure version} of \hyperlink{CAT}{CAT} by requiring $\widetilde{A}_{N}\overset{as}{\rightarrow}0$ instead of $\widetilde{A}_{N}\overset{P}{\rightarrow}0$.

The following theorem provides two characterizations of \hyperlink{CAT}{CAT} which are analogous to the characterizations of (unconditional) asymptotic tightness given in Theorem 1.5.6 and Theorem 1.5.7 on pages 36-37 in \cite{vdVW}.

\begin{thm}\label{Theorem_CAT_characterizations}
The following three conditions are equivalent: 
\begin{itemize}
\item[(i) ] the sequence $\{\mathbb{H}_{N}'\}_{N=1}^{\infty}$ is \hyperlink{CAT}{CAT};
\item[(ii) ] the marginals of the sequence $\{\mathbb{H}_{N}'\}_{N=1}^{\infty}$ are \hyperlink{CAT}{CAT} in the sense that for every $f\in\mathcal{F}$ and $\eta>0$ there exists a constant $M>0$ such that
\begin{equation}\label{marginal_CAT}
\begin{split}
&P_{d}\left\{|\mathbb{H}_{N}'f|\leq M\right\}+\widetilde{B}_{N}\geq1-\eta\quad\text{ for every }N=1,2,\dots
\end{split}
\end{equation}
for some sequence of random variables $\widetilde{B}_{N}$ which goes to zero in probability (almost surely), and there exists a semimetric $\rho$ for which $\mathcal{F}$ is totally bounded and for which the sequence $\{\mathbb{H}_{N}'\}_{N=1}^{\infty}$ is \textbf{conditionally asymptotically $\rho$-equicontinuous} (henceforth \hypertarget{conditional_AEC}{\textbf{conditionally AEC}} w.r.t. $\rho$) in the sense that for every $\epsilon,\eta>0$ there exists a $\delta>0$ such that
\begin{equation}\label{AEC_general_definition}
\begin{split}
&P_{d}\left\{\sup_{f,g\in\mathcal{F}:\rho(f,g)<\delta}|\mathbb{H}_{N}'f-\mathbb{H}_{N}'g|>\epsilon\right\}<\eta+\widetilde{C}_{N}\\
&\text{ for every }N=1,2,\dots
\end{split}
\end{equation}
for some sequence of random variables $\widetilde{C}_{N}$ which goes to zero in probability (almost surely);
\item[(iii) ] the marginals of the sequence $\{\mathbb{H}_{N}'\}_{N=1}^{\infty}$ are \hyperlink{CAT}{CAT} and the following conditional probability version of finite approximation holds: for every $\epsilon,\eta>0$ there exists a finite partition $\mathcal{F}=\cup_{i=1}^{k}\mathcal{F}_{i}$ such that 
\begin{equation}\label{FA_general_definition}
\begin{split}
&P_{d}\left\{\sup_{1\leq i\leq k} \sup_{f,g\in\mathcal{F}_{i}}|\mathbb{H}_{N}'f-\mathbb{H}_{N}'g|>\epsilon\right\}<\eta+\widetilde{D}_{N}\\
&\text{ for every }N=1,2,\dots
\end{split}
\end{equation}
for some sequence of random variables $\widetilde{D}_{N}$ which goes to zero in probability (almost surely).
\end{itemize}
\end{thm}

\begin{proof}
(i)$\Rightarrow$(ii). If condition (\ref{conditional_asy_tightness}) holds, then it follows that
\[P_{d}\{|\mathbb{H}_{N}'f|\leq M\}+\widetilde{A}_{N}\geq 1-\eta\quad\text{ for every }N=1,2,\dots\]
with $M:=\sup\{|z(f)|:z\in K^{\delta}\}<\infty$ and with the same sequence $\{\widetilde{A}_{N}\}_{N=1}^{\infty}$ as in (\ref{conditional_asy_tightness}). This shows that the marginals of the sequence $\{\mathbb{H}_{N}'\}_{N=1}^{\infty}$ are \hyperlink{CAT}{CAT}. Next, consider a sequence $\{K_{m}\}_{m=1}^{\infty}$ of compact subsets of $l^{\infty}(\mathcal{F})$ such that, for every fixed $m=1,2,\dots$, condition (\ref{conditional_asy_tightness}) holds with $\eta=1/m$ and $K=K_{m}$. Note that the sequence $\widetilde{A}_{N}$ in (\ref{conditional_asy_tightness}) depends on $\eta$, $K$ and $\delta>0$ and hence we shall write $\widetilde{A}_{N}(\eta, K, \delta)$ instead of $\widetilde{A}_{N}$. By part (b) of Theorem 6.2 on page 88 in \cite{Kosorok}, to every $K_{m}$ there corresponds a semimetric $\rho_{m}$ which makes $\mathcal{F}$ totally bounded and for which $K_{m}$ is a subset of $UC(\mathcal{F},\rho_{m})$, i.e. of the class of all real-valued functions on $\mathcal{F}$ which are uniformly $\rho_{m}$-continuous. Based on the sequence of semimetrics $\{\rho_{m}\}_{m=1}^{\infty}$, define a new semimetric by
\[\rho(f,g):=\sum_{m=1}^{\infty}2^{-m}[\rho_{m}(f,g)\wedge 1],\quad f,g\in\mathcal{F}.\]
Then it follows that $K_{m}\subset UC(\mathcal{F},\rho)$ for every $m=1,2,\dots$ and moreover it is not difficult to show that $\mathcal{F}$ is totally bounded w.r.t. $\rho$ too. Now, choose $\epsilon>0$ arbitrarily and note that $\mathbb{H}_{N}'\in K_{m}$ implies 
\[\sup_{f,g\in\mathcal{F}:\rho(f,g)<\delta}|\mathbb{H}_{N}'f-\mathbb{H}_{N}'g|\leq\epsilon/3\]
for some small enough $\delta>0$, and hence that $\mathbb{H}_{N}'\in K_{m}^{\epsilon/3}$ implies 
\[\sup_{f,g\in\mathcal{F}:\rho(f,g)<\delta}|\mathbb{H}_{N}'f-\mathbb{H}_{N}'g|\leq\epsilon\]
for the same set of values of $\delta$. For small enough values of $\delta>0$ it follows therefore that condition (\ref{AEC_general_definition}) is satisfied with $\eta=1/m$ and $\widetilde{C}_{N}:=\widetilde{A}_{N}(1/m, K_{m}, \epsilon/3)$.

(ii)$\Rightarrow$(iii). Fix $\delta>0$ and choose a finite collection of open balls of $\rho$-radius $\delta$ which covers $\mathcal{F}$, disjointify and create a finite partition $\mathcal{F}=\cup_{i=1}^{k}\mathcal{F}_{i}$ such that each $\mathcal{F}_{i}$ is a subset of an open ball of $\rho$-radius $\delta$. Then (\ref{AEC_general_definition}) implies (\ref{FA_general_definition}) with the sequence $\{\widetilde{D}_{N}\}_{N=1}^{\infty}$ equal to the sequence $\{\widetilde{C}_{N}\}_{N=1}^{\infty}$ from the definition of \hyperlink{conditional_AEC}{conditional AEC} corresponding to $\epsilon$, $\eta$ and $\delta$.

(iii)$\Rightarrow$(i). This part of the proof is essentially the same as the proof of Theorem 1.5.6 on page 36 in \cite{vdVW}. First, it will be shown that (iii) implies that $\lVert \mathbb{H}_{N}'\rVert_{\mathcal{F}}$ is \hyperlink{CAT}{CAT}, i.e. that for every $\eta>0$ there exists a constant $M$ such that
\begin{equation}\label{maximum_CAT}
\begin{split}
&P_{d}\left\{\lVert\mathbb{H}_{N}'\rVert_{\mathcal{F}}\leq M\right\}+\widetilde{E}_{N}\geq1-\eta\quad\text{ for every }N=1,2,\dots
\end{split}
\end{equation}
for some sequence of random variables $\widetilde{E}_{N}$ which goes to zero in probability (almost surely). To this aim, choose $\epsilon,\eta>0$ arbitrarily and let $\mathcal{F}=\cup_{i=1}^{k}\mathcal{F}_{i}$ be a corresponding partition for which (\ref{FA_general_definition}) holds. Then choose one index $f_{i}$ from each partition set $\mathcal{F}_{i}$ and note that
\[|\mathbb{H}_{N}'f_{i}|\leq M-\epsilon,\quad i=1,2,\dots, k,\quad\text{ and }\quad\sup_{1\leq i\leq k} \sup_{f,g\in\mathcal{F}_{i}}|\mathbb{H}_{N}'f-\mathbb{H}_{N}'g|\leq\epsilon\]
implies $\lVert\mathbb{H}_{N}'\rVert_{\mathcal{F}}\leq M$. Now, from the  assumptions in (iii) it follows that there exists sequences of random variables $\{\widetilde{B}_{N}\}_{N=1}^{\infty}$ and $\{\widetilde{D}_{N}\}_{N=1}^{\infty}$ which go to zero in probability (almost surely) such that, for every $N=1,2,\dots$,
\[P_{d}\left\{|\mathbb{H}_{N}'f_{i}|\leq M-\epsilon\right\}+\widetilde{B}_{N}\geq1-\eta/(k+1),\quad i=1,2,\dots,k,\]
and
\[P_{d}\left\{\sup_{1\leq i\leq k} \sup_{f,g\in\mathcal{F}_{i}}|\mathbb{H}_{N}'f-\mathbb{H}_{N}'g|\leq\epsilon\right\}+\widetilde{D}_{N}\geq1-\eta/(k+1),\]
and from this it follows that condition (\ref{maximum_CAT}) holds with $\widetilde{E}_{N}:=k \widetilde{B}_{N}+\widetilde{D}_{N}$.

Next, consider an arbitrary $\eta>0$ and put $\epsilon=\epsilon_{m}$ for an arbitrary sequence $\epsilon_{m}\downarrow 0$. Let $\{\mathcal{F}=\cup_{i=1}^{k_{m}}\mathcal{F}_{m,i}\}_{m=1}^{\infty}$ be a corresponding sequence of partitions such that 
\begin{equation}\label{FA_general_definition_proof}
\begin{split}
&P_{d}\left\{\sup_{1\leq i\leq k_{m}} \sup_{f,g\in\mathcal{F}_{m,i}}|\mathbb{H}_{N}'f-\mathbb{H}_{N}'g|>\epsilon_{m}\right\}<2^{-m}\eta+\widetilde{D}_{m,N}\\
&\text{ for every }N=1,2,\dots
\end{split}
\end{equation}
for some sequences of non negative random variables $\widetilde{D}_{m,N}$ which go to zero in probability (almost surely) when $m=1,2,\dots$ is kept fixed and $N$ goes to infinity. The conditional probability version of finite approximation ensures that for every fixed $m=1,2,\dots$ there exist such a partition $\mathcal{F}=\cup_{i=1}^{k_{m}}\mathcal{F}_{m,i}$ and such a sequence of random variables $\{\widetilde{D}_{m,N}\}_{N=1}^{\infty}$. Now, denote by $z_{m,1}, z_{m,2},\dots, z_{m,p_{m}}\in l^{\infty}(\mathcal{F})$ the functions which are constant on each partition set $\mathcal{F}_{m,i}$ and which can take on only the values $\pm\epsilon_{m}, \pm 2\epsilon_{m},\dots, \pm \lfloor M/\epsilon_{m}\rfloor$. Note for each $m$ there are only finitely many such functions, i.e. that $p_{m}<\infty$ for every $m=1,2,\dots$. Let $K_{m}$ denote the union of the $p_{m}$ closed balls of radius $\epsilon_{m}$ around each function $z_{m,i}$. Then, 
\[\lVert \mathbb{H}_{N}'\rVert_{\mathcal{F}}\leq M\quad\text{ and }\quad \sup_{1\leq i\leq k_{m}} \sup_{f,g\in\mathcal{F}_{m,i}}|\mathbb{H}_{N}'f-\mathbb{H}_{N}'g|\leq\epsilon_{m}\]
implies that $\mathbb{H}_{N}'\in K_{m}$. Consider the set $K:=\cap_{m=1}^{\infty}K_{m}$ and note that $K$ is closed and totally bounded and hence compact (since $l^{\infty}(\mathcal{F})$ is complete). Moreover, it can be shown that for every $\delta>0$ there is a finite $m$ such that $\cap_{i=1}^{m}K_{i}\subset K^{\delta}$ (the proof of this claim is given in the proof of Theorem 1.5.6 on pages 36 and 37 in \cite{vdVW}). Using this fact along with condition (\ref{maximum_CAT}) and condition (\ref{FA_general_definition_proof}) yields
\[P_{d}\{\mathbb{H}_{N}'\in K^{\delta}\}\geq P_{d}\{\mathbb{H}_{N}'\in \cap_{i=1}^{m}K_{i}\}\geq 1-2\eta-\sum_{i=1}^{m}\widetilde{D}_{i,N}-\widetilde{E}_{N}\]
which completes the proof.
\end{proof}

\begin{rem}\label{remark_conditional_AEC}
Consider the case where $\mathcal{F}$ is a class of measurable functions $f:\mathcal{Y}\mapsto\mathbb{R}$ and where $\{\mathbb{H}_{N}'\}_{N=1}^{\infty}$ is a sequence of \hyperlink{HTEP}{HTEP}s as defined in (\ref{HT_empirical_process}). It is not difficult to show that in this case $\{\mathbb{H}_{N}'\}_{N=1}^{\infty}$ is \hyperlink{conditional_AEC}{conditionally AEC} w.r.t. to a given semimetric $\rho$ (see (ii) in the statement of the previous theorem) if and only if
\begin{equation}\label{def_conditional_AEC}
P_{d}\left\{\lVert\mathbb{H}_{N}'\rVert_{\mathcal{F}_{\delta_{N}}}>\epsilon\right\}\overset{P*(as*)}{\rightarrow}0\quad\text{ for every }\epsilon>0\text{ and for every }\delta_{N}\downarrow 0,
\end{equation}
where
\[\mathcal{F}_{\delta}:=\{f-g:f,g\in\mathcal{F}\wedge\rho(f,g)<\delta\},\quad\delta>0.\]
\end{rem}

The next theorem shows that \hyperlink{CAT}{CAT} is a necessary condition for conditional weak convergence.

\begin{thm}
Let $\mathbb{H}'$ be a Borel measurable and tight mapping from some probability space into $l^{\infty}(\mathcal{F})$ and assume that condition (\ref{CWC_H}) holds. Then it follows that the sequence $\{\mathbb{H}_{N}'\}_{N=1}^{\infty}$ satisfies the probability (almost sure) version of \hyperlink{CAT}{CAT}.
\end{thm}

\begin{proof}
Let $B$ be an arbitrary Borel subset of $l^{\infty}(\mathcal{F})$ and let
\[d(z,B):=\inf_{z'\in B}\lVert z-z'\rVert_{\mathcal{F}}\]
be the distance between $z\in l^{\infty}(\mathcal{F})$ and the set $B$. Note that for $\delta> 0$ the function
\[h_{B,\delta}(z):=(1-\delta^{-1}d(z,B))^{+},\quad z\in l^{\infty}(\mathcal{F}),\]
has Lipschitz constant $\delta^{-1}$ and that $I(z\in B)\leq h_{B,\delta}(z)\leq I(z\in B^{\delta})$, where $B^{\delta}$ is the open $\delta$-enlargement of the set $B$. Thus, $\delta h_{B,\delta}\in BL_{1}(l^{\infty}(\mathcal{F}))$ and therefore it follows that
\begin{equation*}
\begin{split}
P_{d}\{\mathbb{H}_{N}'\in B^{\delta}\}&=E_{d}I(\mathbb{H}_{N}'\in B^{\delta})\\
&\geq E_{d}h_{B,\delta}(\mathbb{H}_{N}')\\
&\geq P\{\mathbb{H}'\in B\} -\frac{1}{\delta}\sup_{h\in BL_{1}(l^{\infty}(\mathcal{F}))}\left|E_{d}h(\mathbb{H}_{N}')-Eh(\mathbb{H}')\right|.
\end{split}
\end{equation*}
Since \hyperlink{opCWC}{opCWC} (\hyperlink{oasCWC}{oasCWC}) implies that the supremum goes to zero in outer probability (outer almost surely), this shows that for every $\delta>0$ there exists a sequence of random variables $\widetilde{A}_{N}\overset{P(as)}{\rightarrow} 0$ such that 
\[P_{d}\{\mathbb{H}_{N}'\in B^{\delta}\}+\widetilde{A}_{N}\geq P\{\mathbb{H}'\in B\}\]
for every Borel subset $B$ of $l^{\infty}(\mathcal{F})$ and for every $N=1,2,\dots$. Since $\mathbb{H}'$ is tight by assumption, the conclusion of the theorem follows from this.
\end{proof}

Now, consider "conditional weak convergence of the finite-dimensional marginals". Perhaps the most obvious way to define this concept is to require pointwise convergence in probability (pointwise almost sure converence) of the sample design characteristic function for every sequence of finite-dimensional vectors $\{\mathbb{H}_{N}'\mathbf{f}\}_{N=1}^{\infty}:=\{(\mathbb{H}_{N}'f_{1}, \mathbb{H}_{N}'f_{2},\dots, \mathbb{H}_{N}'f_{r})^{\intercal}\}_{N=1}^{\infty}$ with $\mathbf{f}:=(f_{1}, f_{2}, \dots, f_{r})^{\intercal}\in\mathcal{F}^{r}$ and $r=1,2,\dots$. Since we are assuming that the components of the vectors $\mathbb{H}_{N}'\mathbf{f}$ are measurable, \textbf{conditional weak convergence of the finite-dimensional marginals} (henceforth \hypertarget{CWCM}{\textbf{CWCM}}) can therefore be defined as 
\begin{equation}\label{CWCM}
E_{d}\exp\left(i\mathbf{t}^{\intercal}\mathbb{H}_{N}'\mathbf{f}\right)\overset{P(as)}{\rightarrow}E\exp\left(i\mathbf{t}^{\intercal}\mathbb{H}'\mathbf{f}\right)
\end{equation}
for every $\mathbf{t}\in\mathbb{R}^{r}$, $\mathbf{f}\in\mathcal{F}^{r}$ and for every $r=1,2,\dots$, where $\mathbb{H}'\mathbf{f}:=(\mathbb{H}'f_{1}, \mathbb{H}'f_{2},\dots, \mathbb{H}'f_{r})^{\intercal}$ is the finite dimensional vector of random variables corresponding to some $\mathcal{F}$-indexed stochastic process $\mathbb{H}':=\{\mathbb{H}'f:f\in \mathcal{F}\}$.

It is not difficult to prove that \hyperlink{CWCM}{CWCM} is a necessary condition for \hyperlink{opCWC}{opCWC} (\hyperlink{oasCWC}{oasCWC}). The next theorem says that \hyperlink{CWCM}{CWCM} is already equivalent to \hyperlink{opCWC}{opCWC} (\hyperlink{oasCWC}{oasCWC}) if the index set $\mathcal{F}$ is finite. In its statement $z\upharpoonright\mathcal{G}$ indicates the restriction of some function $z\in l^{\infty}(\mathcal{F})$ to a subset $\mathcal{G}$ of $\mathcal{F}$.

\begin{thm}\label{teorema_CWCM}
Let $\mathcal{G}$ be a finite subset of $\mathcal{F}$. Under the assumptions made at the beginning of this section 
\[\sup_{h\in BL_{1}(l^{\infty}(\mathcal{G}))}\left|E_{d}h(\mathbb{H}_{N}'\upharpoonright\mathcal{G})-Eh(\mathbb{H}'\upharpoonright\mathcal{G})\right|\]
is measurable, and \hyperlink{CWCM}{CWCM} is equivalent to
\[\sup_{h\in BL_{1}(l^{\infty}(\mathcal{G}))}\left|E_{d}h(\mathbb{H}_{N}'\upharpoonright\mathcal{G})-Eh(\mathbb{H}'\upharpoonright\mathcal{G})\right|\overset{P(as)}{\rightarrow} 0\quad\text{ for every finite }\mathcal{G}\subset\mathcal{F}.\]
\end{thm}

\begin{proof}
The proof is the same as the proof of Corollary 3.1 in \cite{Pasquazzi_2019}.
\end{proof}

Up to now it has already been shown that \hyperlink{CAT}{CAT} and \hyperlink{CWCM}{CWCM} are necessary conditions for \hyperlink{opCWC}{opCWC} (\hyperlink{oasCWC}{oasCWC}) and that \hyperlink{CWCM}{CWCM} is equivalent to \hyperlink{opCWC}{opCWC} (\hyperlink{oasCWC}{oasCWC}) when $\mathcal{F}$ is finite. The next theorem says that \hyperlink{CAT}{CAT} and \hyperlink{CWCM}{CWCM} together are sufficient conditions for \hyperlink{opCWC}{opCWC} (\hyperlink{oasCWC}{oasCWC}) regardless of the cardinality of $\mathcal{F}$.

\begin{thm}\label{sufficient_conditions_for_CWC}
Assume that $\{\mathbb{H}_{N}'\}_{N=1}^{\infty}$ satisfies \hyperlink{CAT}{CAT} and \hyperlink{CWCM}{CWCM} for some stochastic process $\{\mathbb{H}'f:f\in\mathcal{F}\}$. Then there exists a version of the stochastic process $\mathbb{H}'$ which is a Borel measurable and tight mapping from some probability space into $l^{\infty}(\mathcal{F})$ such that \hyperlink{opCWC}{opCWC} (\hyperlink{oasCWC}{oasCWC}) as defined in (\ref{CWC_H}) holds.
\end{thm}

\begin{proof}
Consider the stochastic process $\{\mathbb{H}'f:f\in\mathcal{F}\}$ from \hyperlink{CWCM}{CWCM}. The first step of the proof is to show that \hyperlink{CAT}{CAT} and \hyperlink{CWCM}{CWCM} imply that there exists a version of $\mathbb{H}'$ with the following properties:
\begin{itemize}
\item[a) ] $\mathbb{H}'$ is a Borel measurable and tight mapping into $l^{\infty}(\mathcal{F})$;
\item[b) ] the sample paths $f\mapsto\mathbb{H}'f$ are uniformly $\rho$-continuous with probability $1$.
\end{itemize}
To this aim note that by the characterization of \hyperlink{CAT}{CAT} given in (ii) of Theorem \ref{Theorem_CAT_characterizations} there must exists a semimetric $\rho$ for which $\mathcal{F}$ is totally bounded and for which $\{\mathbb{H}_{N}'\}_{N=1}^{\infty}$ is \hyperlink{conditional_AEC}{conditionally AEC}. Now, consider a countable and dense (w.r.t. $\rho$) subset $\mathcal{G}$ of $\mathcal{F}$ and note that the random variable
\[\sup_{f,g\in\mathcal{G}:\rho(f,g)<\delta}|\mathbb{H}_{N}'f-\mathbb{H}_{N}'g|\]
is measurable. Conclude that 
\[P_{d}\left\{\sup_{f,g\in\mathcal{G}:\rho(f,g)<\delta}|\mathbb{H}_{N}'f-\mathbb{H}_{N}'g|>\epsilon\right\}\]
is a conditional probability in the proper sense and use this fact along with (\ref{AEC_general_definition}) in order to show that for every $\epsilon,\eta>0$ there exists a $\delta>0$ such that
\[\limsup_{N\rightarrow\infty}P\left\{\sup_{f,g\in\mathcal{G}:\rho(f,g)<\delta}|\mathbb{H}_{N}'f-\mathbb{H}_{N}'g|>\epsilon\right\}<\eta,\]
i.e. that the sequence of mappings $\{\mathbb{H}_{N}'\upharpoonright\mathcal{G}\}_{N=1}^{\infty}$ is asymptotically uniformly $\rho$-equicontinuous in probability as defined on page 37 in \cite{vdVW}. By Theorem 1.5.7 on the same page in \cite{vdVW} it follows that $\{\mathbb{H}_{N}'\upharpoonright\mathcal{G}\}_{N=1}^{\infty}$ is asymptotically tight, and since \hyperlink{CWCM}{CWCM} implies unconditional convergence of the marginal distributions of $\{\mathbb{H}_{N}'\upharpoonright\mathcal{G}\}_{N=1}^{\infty}$, it follows by Theorem 1.5.4 on page 35 in \cite{vdVW} that there exists a Borel measurable and tight mapping from some probability space into $l^{\infty}(\mathcal{G})$, call it $\widetilde{\mathbb{H}}'$, such that $\mathbb{H}_{N}'\rightsquigarrow\widetilde{\mathbb{H}}'$ in $l^{\infty}(\mathcal{G})$. Moreover, by Addendum 1.5.8 on page 37 in \cite{vdVW} it follows that the sample paths of $\widetilde{\mathbb{H}}'$ are uniformly $\rho$-continuous with probability $1$. Let $c:l^{\infty}(\mathcal{G})\mapsto l^{\infty}(\mathcal{F})$ be the mapping which carries the uniformly $\rho$-continuous functions in $l^{\infty}(\mathcal{G})$ to their uniformly $\rho$-continuous extensions in $l^{\infty}(\mathcal{F})$, and which transforms all other functions in $l^{\infty}(\mathcal{G})$ into the zero function in $l^{\infty}(\mathcal{F})$. Then the mapping $c$ is certainly measurable and $\widehat{\mathbb{H}}':=c(\widetilde{\mathbb{H}}')$ is is a continuous function of $\widetilde{\mathbb{H}}'$ with probability $1$. It follows that $\widehat{\mathbb{H}}'$ is a Borel measurable and tight mapping into $l^{\infty}(\mathcal{F})$ whose sample paths are uniformly $\rho$-continuous with probability $1$. In order to prove a) and b) it suffices now to show that the finite dimensional distributions of $\widehat{\mathbb{H}}'$ are the same as those of the limit process $\{\mathbb{H}'f:f\in\mathcal{F}\}$ from \hyperlink{CWCM}{CWCM}, i.e. that every finite dimensional vector $\mathbf{f}:=(f_{1},f_{2},\dots, f_{r})^{\intercal}\in\mathcal{F}^{r}$ satisfies the condition
\begin{equation}\label{marginali_uguali}
Eg(\widehat{\mathbb{H}}'\mathbf{f})=Eg(\mathbb{H}'\mathbf{f})\quad\text{ for every }g\in C_{b}(\mathbb{R}^{r}),
\end{equation}
where $C_{b}(\mathbb{R}^{r})$ is the set of all continuous and bounded functions $g:\mathbb{R}^{r}\mapsto\mathbb{R}$. If all the components of $\mathbf{f}$ are elements of $\mathcal{G}$, then (\ref{marginali_uguali}) follows directly from \hyperlink{CWCM}{CWCM} and the definition of $\widehat{\mathbb{H}}'$. Otherwise, if some or all of the components of $\mathbf{f}$ are elements of $\mathcal{F}$ but not of $\mathcal{G}$, then there exists a sequence $\{\mathbf{f}_{\nu}\}_{\nu=1}^{\infty}:=\{(f_{\nu,1}, f_{\nu,2}, \dots, f_{\nu,r})^{\intercal}\}_{\nu=1}^{\infty}$ in $\mathcal{G}^{r}$ such that $\rho_{r}(\mathbf{f}_{\nu},\mathbf{f}):=\max_{1\leq i\leq r}\rho(f_{i},f_{\nu,i})\rightarrow 0$. Since the sample paths of $\widehat{\mathbb{H}}'$ are uniformly $\rho$-continuous with probability $1$, it follows that $\widehat{\mathbb{H}}'\mathbf{f}_{\nu}\overset{as}{\rightarrow}\widehat{\mathbb{H}}'\mathbf{f}$ and hence that $\widehat{\mathbb{H}}'\mathbf{f}_{\nu}\rightsquigarrow\widehat{\mathbb{H}}'\mathbf{f}$ which is the same as $\mathbb{H}'\mathbf{f}_{\nu}\rightsquigarrow\widehat{\mathbb{H}}'\mathbf{f}$. Now, consider $\mathcal{G}^{\dagger}:=\mathcal{G}\cup\{f_{1}, f_{2}, \dots, f_{r}\}$ and define $\widehat{\mathbb{H}}^{\dagger}$ in terms of $\mathcal{G}^{\dagger}$ in the same way as $\widehat{\mathbb{H}}'$ has been defined above in term of $\mathcal{G}$. Then, as before, it follows by \hyperlink{CWCM}{CWCM} that (\ref{marginali_uguali}) holds with $\widehat{\mathbb{H}}^{\dagger}\mathbf{f}$ in place of $\widehat{\mathbb{H}}'\mathbf{f}$, and since the sample paths of $\widehat{\mathbb{H}}^{\dagger}$ are uniformly $\rho$-continuous with probability $1$, it follows that $\widehat{\mathbb{H}}^{\dagger}\mathbf{f}_{\nu}\rightsquigarrow\widehat{\mathbb{H}}^{\dagger}\mathbf{f}$. Since for every $\nu=1,2,\dots$ the three random vectors $\widehat{\mathbb{H}}^{\dagger}\mathbf{f}_{\nu}$, $\widehat{\mathbb{H}}'\mathbf{f}_{\nu}$ and $\mathbb{H}'\mathbf{f}_{\nu}$ have all the same distribution and since the distributions of $\widehat{\mathbb{H}}^{\dagger}\mathbf{f}$ and $\mathbb{H}\mathbf{f}$ are the same as well, this implies
\[Eg(\mathbb{H}'\mathbf{f})=Eg\left(\widehat{\mathbb{H}}^{\dagger}\mathbf{f}\right)=Eg\left(\widehat{\mathbb{H}}'\mathbf{f}\right)\quad\text{ for every }g\in C_{b}(\mathbb{R}^{r}),\]
and hence that condition (\ref{marginali_uguali}) holds also when some or all of the components of $\mathbf{f}$ are in $\mathcal{F}$ but not in $\mathcal{G}$. This shows that the marginal distributions of $\widehat{\mathbb{H}}'$ are the same as those of the stochastic process $\{\mathbb{H}'f:f\in\mathcal{F}\}$ and hence that there exists a version of the latter process which satisfies a) and b).

Next, it will be shown that if $\mathbb{H}'$ is a version of $\{\mathbb{H}'f:f\in\mathcal{F}\}$ which satisfies a) and b), then \hyperlink{opCWC}{opCWC} (\hyperlink{oasCWC}{oasCWC}) as defined in (\ref{CWC_H}) holds (this part of the proof is essentially the same as the proof of Theorem 2.9.6 on page 182 in \cite{vdVW}). To this aim define for each $\delta>0$ a corresponding set $\mathcal{G}_{\delta}$ which contains the centers of a collection of open balls of $\rho$-radius $\delta$ which cover $\mathcal{F}$. Since $\mathcal{F}$ is totally bounded w.r.t. $\rho$, each $\mathcal{G}_{\delta}$ can be chosen to be finite. Then define for each fixed $\delta>0$ a mapping $\Pi_{\delta}:\mathcal{F}\mapsto\mathcal{G}_{\delta}$ which maps $f\in\mathcal{F}$ to the element $g\in\mathcal{G}_{\delta}$ which is closest to $f$. If there are more than one $g\in\mathcal{G}_{\delta}$ which minimize $\rho(f,g)$, $\Pi_{\delta}(f)$ can be defined to be any such $g$. Since the sample paths of $\mathbb{H}'$ are uniformly $\rho$-continuous with probability $1$, it follows that $\lVert\mathbb{H}'\circ \Pi_{\delta}-\mathbb{H}'\rVert_{\mathcal{F}}\overset{as}{\rightarrow}0$ if $\delta\rightarrow 0$ and hence that
\[\sup_{h\in BL_{1}(l^{\infty}(\mathcal{F}))}|Eh(\mathbb{H}'\circ \Pi_{\delta})-Eh(\mathbb{H}')|\overset{as}{\rightarrow}0\quad\text{ if }\delta\rightarrow 0.\]
Next, it will be shown that
\begin{equation}\label{marginal_convergence}
\sup_{h\in BL_{1}(l^{\infty}(\mathcal{F}))}\left|E_{d}h(\mathbb{H}_{N}'\circ\Pi_{\delta})-Eh(\mathbb{H}'\circ\Pi_{\delta})\right|\overset{P*(as*)}{\rightarrow}0\quad\text{ for every }\delta>0.
\end{equation}
To this aim, define for each $\delta>0$ the mapping $A_{\delta}:l^{\infty}(\mathcal{G}_{\delta})\mapsto l^{\infty}(\mathcal{F})$ by $A_{\delta}(z):=z\circ\Pi_{\delta}$ and note that $A_{\delta}$ transforms a function $z\in l^{\infty}(\mathcal{G}_{\delta})$ into a function $z'\in l^{\infty}(\mathcal{F})$ by extending the domain from $\mathcal{G}_{\delta}$ to $\mathcal{F}$: for $f\in\mathcal{G}_{\delta}$ the new function $z'$ remains the same (in fact, $z'(f):=z(\Pi_{\delta}(f))=z(f)$), and the new function $z'$ is constant on each level set of $\Pi_{\delta}$ (since $\mathcal{G}_{\delta}$ is finite there is only a finite number of such level sets and the range of the new function $z'$ must therefore be finite as well). Then, for $h:l^{\infty}(\mathcal{F})\mapsto\mathbb{R}$ and $H$ an arbitrary $\mathcal{F}$-indexed stochastic process it follows that $h(H\circ\Pi_{\delta})=h\circ A_{\delta}(H\restriction\mathcal{G}_{\delta})$. Moreover, if $h\in BL_{1}(l^{\infty}(\mathcal{F}))$, then
\[|h\circ A_{\delta}(z_{1})-h\circ A_{\delta}(z_{2})|\leq\lVert A_{\delta}(z_{1})-A_{\delta}(z_{2})\rVert_{\mathcal{F}}=\lVert z_{1}\circ\Pi_{\delta}-z_{2}\circ\Pi_{\delta}\rVert_{\mathcal{F}}=\lVert z_{1}-z_{2}\rVert_{\mathcal{G}_{\delta}},\]
and the composition $h\circ A_{\delta}$ is therefore a member of $BL_{1}(l^{\infty}(\mathcal{G}_{\delta}))$, i.e. of the set of all functions $g:l^{\infty}(\mathcal{G}_{\delta})\mapsto[0,1]$ such that $|g(z_{1})-g(z_{2})|\leq \lVert z_{1}-z_{2}\rVert_{\mathcal{G}_{\delta}}$ for every $z_{1},z_{2}\in l^{\infty}(\mathcal{G}_{\delta})$. It follows that the supremum on the left side in (\ref{marginal_convergence}) is bounded by
\[\sup_{g\in BL_{1}(l^{\infty}(\mathcal{G}_{\delta}))}\left|E_{d}g(\mathbb{H}_{N}'\restriction\mathcal{G}_{\delta})-Eg(\mathbb{H}'\restriction\mathcal{G}_{\delta})\right|,\]
which, by Theorem \ref{teorema_CWCM}, is measurable and goes to zero in probability (almost surely). This proves (\ref{marginal_convergence}). 

Finally, in order to complete the proof it remains to show that for every $\epsilon>0$ there exists a $\delta>0$ such that
\[\sup_{h\in BL_{1}(l^{\infty}(\mathcal{F}))}|E_{d}h(\mathbb{H}_{N}')-E_{d}h(\mathbb{H}_{N}'\circ \Pi_{\delta})|\leq\epsilon+\widetilde{C}_{N}\quad\text{ for every }N=1,2,\dots\]
for some sequence of random variables $\{\widetilde{C}_{N}\}_{N=1}^{\infty}$ which goes to zero in probability (almost surely). To this aim note that the left side in the last display is bounded by
\[\epsilon/2+P_{d}\left\{\sup_{f,g\in\mathcal{G}:\rho(f,g)<\delta}|\mathbb{H}_{N}'f-\mathbb{H}_{N}'g|>\epsilon/2\right\}\]
and that by (ii) of Theorem \ref{Theorem_CAT_characterizations} there must exists a $\delta>0$ such that
\[P_{d}\left\{\sup_{f,g\in\mathcal{G}:\rho(f,g)<\delta}|\mathbb{H}_{N}'f-\mathbb{H}_{N}'g|>\epsilon/2\right\}\leq \epsilon/2 + \widetilde{C}_{N}\]
for some sequence of non negative random variables $\widetilde{C}_{N}\overset{P(as)}{\rightarrow}0$. The proof of the theorem is now complete.
\end{proof}

\begin{cor}\label{cor_uniformly_continuous_limit_process}
Assume that $\{\mathbb{H}_{N}'\}_{N=1}^{\infty}$ satisfies \hyperlink{CWCM}{CWCM} for some stochastic process $\{\mathbb{H}'f:f\in\mathcal{F}\}$, and assume that there exists a semimetric $\rho$ for which $\mathcal{F}$ is totally bounded and for which $\{\mathbb{H}_{N}'\}_{N=1}^{\infty}$ is \hyperlink{conditional_AEC}{conditionally AEC}. Then, it follows that
\begin{itemize}
\item[(i) ] there exists a version of $\mathbb{H}'$ which is a Borel measurable and tight mapping from some probability space into $l^{\infty}(\mathcal{F})$ such that \hyperlink{opCWC}{opCWC} (\hyperlink{oasCWC}{oasCWC}) as defined in (\ref{CWC_H}) holds;
\item[(ii) ] the sample paths of $\mathbb{H}'$ are uniformly $\rho$-continuous with probability $1$.
\end{itemize}
\end{cor}

\begin{proof}
It is easy to see that \hyperlink{CWCM}{CWCM} implies that the marginals of the sequence $\{\mathbb{H}_{N}'\}_{N=1}^{\infty}$ are \hyperlink{CAT}{CAT} as required by (ii) in the statement of Theorem \ref{Theorem_CAT_characterizations}. Thus, it follows from Theorem \ref{Theorem_CAT_characterizations} that $\{\mathbb{H}_{N}'\}_{N=1}^{\infty}$ is \hyperlink{CAT}{CAT} and the above proof of Theorem \ref{sufficient_conditions_for_CWC} yields the two conclusions of the corollary. 
\end{proof}

\begin{rem}\label{rem_measurable_suprema}
As already pointed out in Section \ref{Section_notation_definitions}, conditional weak convergence is apparently not strong enough to imply asymptotic measurability of $\{\mathbb{H}_{N}'\}_{N=1}^{\infty}$ which is a necessary condition for unconditional weak convergence. However, if for every $\mathcal{G}\subset\mathcal{F}$ the corresponding supremum
\[\sup_{f,g\in\mathcal{G}}|\mathbb{H}_{N}'f-\mathbb{H}_{N}'g|\]
is measurable, then it will be certainly true that conditional weak convergence is stronger than unconditional weak convergence. In fact, if the suprema in the above display are measurable, then it follows that the probabilities on the left sides in (\ref{AEC_general_definition}) and in (\ref{FA_general_definition}) are conditional probabilities in the proper sense and hence that \hyperlink{CAT}{CAT} implies asymptotic tightness in the usual unconditional sense. Since conditional weak convergence implies \hyperlink{CWCM}{CWCM} and since \hyperlink{CWCM}{CWCM} is certainly stronger than unconditional convergence of the marginal distributions, it follows by Theorem 1.5.4 on page 35 in \citep{vdVW} that conditional weak convergence implies unconditional weak convergence in this case.
\end{rem}

\begin{thm}\label{teorema_generale_joint_weak_convergence}
Let $\{\mathbb{H}_{N}\}_{N=1}^{\infty}$ be a sequence of mappings from the probability space (\ref{probability_space}) into $l^{\infty}(\mathcal{F})$ and assume that each $\mathbb{H}_{N}$ depends on the sample points $\omega\in\Omega_{y,x}^{\infty}\times\Omega_{d}$ only through $\mathbf{Y}_{N}$ and $\mathbf{X}_{N}$. Moreover, assume that $\mathbb{H}_{N}\rightsquigarrow\mathbb{H}$ in $l^{\infty}(\mathcal{F})$ with $\mathbb{H}$ a Borel measurable and tight mapping into $l^{\infty}(\mathcal{F})$, and assume that \hyperlink{opCWC}{opCWC} as defined in (\ref{CWC_H}) holds. Then it follows that
\[(\mathbb{H}_{N},\mathbb{H}_{N}')\rightsquigarrow(\mathbb{H},\mathbb{H}')\quad\text{ in }l^{\infty}(\mathcal{F})\times l^{\infty}(\mathcal{F})\]
with $\mathbb{H}$ and $\mathbb{H}'$ independent.
\end{thm}

\begin{proof}
The proof is the same as the proof of Corollary 3.2 in \citep{Pasquazzi_2019}.
\end{proof}

\section{Conditional Poisson sampling (or rejective sampling)}\label{rejective_sampling_review}

This section reviews some theoretical results about rejective sampling. The basic theory for this sampling design was developed by \cite{Hajek_1964}. Some of the results contained in his paper will be needed in the next section. The relevant ones will be singled out in what follows. 

Recall that a rejective sampling design is a conditional Poisson sampling design where the final sample is rejected unless its size equals a given natural number $n$. Equivalently, rejective sampling can also be defined as random sampling with replacement of a fixed number $n$ of units according to specified selection probabilities where the final sample is rejected and the sampling procedure is repeated until a sample of $n$ different units is obtained. Rejective sampling is of great interest to researchers and practitioners because it provides largest possible entropy subject to the constraints of a fixed sample size and given first order sample inclusion probabilities (see Theorem 3.4 in H{\'a}jek, 1981). By the definition of rejective sampling as conditional Poisson sampling it follows that the pdf of the vector of sample inclusion indicators is given by
\begin{equation}\label{sample_selection_probs_CPS}
\mathfrak{p}_{N}^{R}(\mathbf{s}_{N};\mathbf{p}_{N}):=\begin{cases}
\frac{\prod_{i=1}^{N}p_{i,N}^{s_{i}}(1-p_{i,N})^{1-s_{i}}}{\sum_{\mathbf{s}_{N}'\in\Omega_{N,n}}\prod_{i=1}^{N}p_{i,N}^{s_{i}'}(1-p_{i,N})^{1-s_{i}'}} & \text{ if }\mathbf{s}_{N}\in\Omega_{N,n},\\
0 & \text{otherwise},
\end{cases}
\end{equation}
where $\mathbf{p}_{N}:=(p_{1,N}, p_{2,N},\dots, p_{N,N})\in(0,1)^{N}$ is the vector of first order sample inclusion probabilities of the underlying Poisson sampling design, and where $\Omega_{N,n}:=\{\mathbf{s}_{N}\in\{0,1\}^{N}: \sum_{i=1}^{N}s_{i}=n\}$ is the set of all possible realizations of the vector of sample inclusion indicators $\mathbf{S}_{N}$ that give rise to samples of size $n$. Of course the definition rejective sampling and hence of $\mathfrak{p}_{N}^{R}(\mathbf{s}_{N};\mathbf{p}_{N})$ can also be extended to the case where some or even all of the $p_{i,N}$'s are $0$ or $1$. However, if $p_{i,N}=0$ for some $i=1,2,\dots, N$, then the corresponding population unit $i$ will be excluded from the sample with probability $1$ and unbiased estimation of population characteristics is hence be impossible. Since the Horvitz-Thompson estimator is not well-defined in this case, we shall henceforth consider only rejective sampling designs such that $p_{i,N}>0$ for every $i=1,2,\dots, N$. On the other hand, if $p_{i,N}=1$ for some $i=1,2,\dots, N$, then the corresponding population unit $i$ will be included in the sample with probability $1$ and rejective sampling with sample size $n$ from a population of size $N$ will be equivalent to rejective sampling with sample size $n-1$ from a population of size $N-1$. In this case the pdf of the vector of sample inclusion indicators will still be defined as in  (\ref{sample_selection_probs_CPS}) provided that $n$ is at least as large as the number of population units for which $p_{i,N}=1$. For smaller values of $n$ a corresponding rejective sampling design does obviously not exist.

Note that in general there are infinitely many underlying Poisson sampling designs which give rise to the same rejective sampling design. In fact, if the vector of first order inclusion probabilities corresponding to the underlying Poisson sampling design is changed to $\mathbf{p}_{N}':=(p_{1,N}', p_{2,N}', \dots, p_{N,N}')$ in such way that for some fixed constant $c>0$
\[\frac{p_{i,N}}{1-p_{i,N}}=\frac{cp_{i,N}'}{1-p_{i,N}'}\quad\text{ for every }i=1,2,\dots,N,\]
then $\mathfrak{p}_{N}^{R}(\mathbf{s}_{N};\mathbf{p}_{N})=\mathfrak{p}_{N}^{R}(\mathbf{s}_{N};\mathbf{p}_{N}')$ for every sample $\mathbf{s}_{N}\in\{0,1\}^{N}$. The underlying Poisson sampling design is called \textit{canonical} if its first order sample inclusion probabilities are chosen so that $\sum_{i=1}^{N}p_{i,N}=n$, i.e. so that the expected sample size of the underlying Poisson sampling design equals the fixed sample size of the rejective sampling plan. Of course, the first order sample inclusion probabilities corresponding to a rejective sampling design are in general different from those corresponding to any of the underlying Poisson sampling designs. However, H{\'a}yek (1964) showed that, in some asymptotic sense, they are uniformly close to those corresponding to the underlying \textit{canonical} Poisson sampling design (see Theorem 5.1 in H{\'a}yek's paper). In fact, H{\'a}yek proved the following result:

\begin{result}\label{Hajek_result_1}
Let $\mathbf{\underline{\pi}}_{N}:=(\pi_{1,N},\pi_{2,N}, \dots,\pi_{N,N})\in(0,1)^{N}$ be the vector of first order sample inclusion probabilities for a rejective sampling design and let $\mathbf{p}_{N}:=(p_{1,N}, p_{2,N}, \dots, p_{N,N})$ be the vector of first order sample inclusion probabilities of the corresponding canonical Poisson sampling design. Then it follows that
\begin{itemize}
\item[(i) ]
\[d_{N}:=\sum_{i=1}^{N}p_{i,N}(1-p_{i,N})\rightarrow \infty\quad\text{ if and only if }\quad d_{0,N}:=\sum_{i=1}^{N}\pi_{i,N}(1-\pi_{i,N})\rightarrow \infty\]
(note that $d_{N}$ is the variance of the random sample size corresponding to the canonical Poisson sampling design);
\item[(ii) ]
\[\max_{1\leq i\leq N}\left|\frac{\pi_{i,N}}{p_{i,N}}-1\right|\rightarrow 0\quad\text{ if } d_{N}\rightarrow \infty;\]
\item[(iii) ] 
\[\max_{1\leq i\leq N}\left|\frac{1-\pi_{i,N}}{1-p_{i,N}}-1\right|\rightarrow 0\quad\text{ if } d_{N}\rightarrow \infty.\]
\end{itemize}
\end{result}

Actually, H{\'a}yek did not use the double subscript notation to indicate the $p_{i,N}$'s and the $\pi_{i,N}$'s. However, it is easily checked that all the proofs given in his paper are actually meant for sequences of rejective sampling designs and corresponding canonical Poisson sampling designs where all the first order sample inclusion probabilities can be redefined as the population size $N$ increases. With respect to Result \ref{Hajek_result_1} it is worth noting that for every properly scaled vector of first order sample inclusion probabilities $\mathbf{\underline{\pi}}_{N}\in(0,1)^{N}$ there exists a corresponding rejective sampling design. In other words, for every $\mathbf{\underline{\pi}}_{N}\in(0,1)^{N}$ such that $\sum_{i=1}^{N}\pi_{i,N}=n$ for some $n=0, 1,2,\dots, N$ there exists a corresponding vector $\mathbf{p}_{N}\in(0,1)^{N}$ such that 
\[\pi_{i,N}=\sum_{\mathbf{s}_{N}\in\Omega_{N,n}:s_{i}=1}\mathfrak{p}_{N}^{R}(\mathbf{s}_{N};\mathbf{p}_{N})\quad\text{ for every }i=1,2,\dots, N.\]
This result has been shown by \cite{Dupacova_1979} but it can also be viewed as a consequence of a well-known theorem about exponential families (see Theorem 5 on page 67 in \cite{Tille_2006}). Fast algorithms to recover the vector $\mathbf{p}_{N}$ corresponding to a given vector $\mathbf{\underline{\pi}}_{N}$ and viceversa can be found in \cite{Chen_1994}.

\smallskip


Having shown Result \ref{Hajek_result_1} and some similar approximation results (also for second order sample inclusion probabilities), \cite{Hajek_1964} moves on to show asymptotic normality for the sequence of Horvitz-Thompson estimators. To this aim, he introduces a new sampling design, call it $P_{0}$, which approximates the canonical Poisson sampling design associated to the rejective sampling design of interest. The sampling design $P_{0}$ can be implemented in three steps according to the following procedure. At the first step a sample of size $n$ is selected using the rejective sampling design of interest. As before, let $\mathbf{p}_{N}$ denote the vector of first order sample inclusion probabilities of the corresponding canonical Poisson sampling design. Then, independently from the outcome of the first step, a new random experiment is performed to ascertain the sample size $K$ of a Poisson sampling design with sample inclusion probabilities given by $\mathbf{p}_{N}$. If $K=n$, then the final sample according to $P_{0}$ will be the rejective sample obtained at the first step. However, if $K>n$, rejective sampling is used again to select a sample of size $K-n$ from the population units that were not included in the first rejective sample, and this new sample is added to the sample obtained at the first step in order to obtain the final sample according to $P_{0}$. The first order sample inclusion probabilities of the canonical Poisson sampling design that underlies the second rejective sampling plan are proportional to the $p_{i,N}$-values of the population units that are not included in the first rejective sample. On the other hand, if $K<n$, rejective sampling is used again to select a sample of size $n-K$ from the population units that were already included in the rejective sample obtained at the first step, and the final sample according to $P_{0}$ is obtained by removing the units that are included in the second rejective sample from those which were already included in the first one. In this case, the first order sample inclusion probabilities of the underlying canonical Poisson sampling design will be proportional to the values of $1-p_{i,N}$ corresponding to the population units that were already included in the first rejective sample. 

Now, to give a formal statement of the sense in which the sample design $P_{0}$ provides an approximation to the canonical Poisson sampling design which underlies the rejective sample plan of interest, it will be convenient to introduce some notation. So, let $\mathbf{S}_{N}^{R}:=(S_{1,N}^{R}, S_{2,N}^{R}, \dots, S_{N,N}^{R})$ be the vector of sample inclusion indicators that describe the outcome of the rejective sampling design that is used at the first step of the experiment (i.e. the rejective sampling design of interest), and let $\mathbf{S}_{N}^{P_{0}}:=(S_{1,N}^{P_{0}}, S_{2,N}^{P_{0}}, \dots, S_{N,N}^{P_{0}})$ denote the vector of sample inclusion indicators that identify the final sample according to $P_{0}$. Moreover, denote their joint pdf by $\mathfrak{p}_{N}^{R,P_{0}}(\cdot,\cdot;\mathbf{p}_{N}):\{0,1\}^{N}\times\{0,1\}^{N}\mapsto[0,1]$, and the marginal pdfs corresponding to $\mathbf{S}_{N}^{R}$ and $\mathbf{S}_{N}^{P_{0}}$ by $\mathfrak{p}_{N}^{R}(\cdot;\mathbf{p}_{N})$ and $\mathfrak{p}_{N}^{P_{0}}(\cdot;\mathbf{p}_{N})$, respectively. Finally, let $\mathfrak{p}_{N}^{P}(\cdot;\mathbf{p}_{N})$ be the pdf of the vector of sample inclusion indicators corresponding to the Poisson sampling design with first order inclusion probabilities given by the components of $\mathbf{p}_{N}$ (i.e. the canonical Poisson sampling design which underlies the rejective sampling design of interest). With this notation, the approximation result contained in Lemma 4.3 of H{\'a}jek's paper can now be stated as follows: 

\begin{result}\label{Hajek_result_2}
The total variation distance 
\[d_{T}(\mathfrak{p}_{N}^{P_{0}}(\cdot;\mathbf{p}_{N}), \mathfrak{p}_{N}^{P}(\cdot;\mathbf{p}_{N})):=\sum_{\mathbf{s}_{N}\in\{0,1\}^{N}}|\mathfrak{p}_{N}^{P_{0}}(\mathbf{s}_{N};\mathbf{p}_{N})-\mathfrak{p}_{N}^{P}(\mathbf{s}_{N};\mathbf{p}_{N})|\]
converges to zero as $d_{N}\rightarrow\infty$.
\end{result}

Based on Result \ref{Hajek_result_1} and Result \ref{Hajek_result_2}, \citet{Hajek_1964} proved asymptotic normality of the Horvitz-Thompson estimators corresponding to a sequence of rejective sampling designs as follows. First, he considered the sequence of underlying canonical Poisson sampling designs and showed asymptotic normality for a corresponding sequence of auxiliary statistics $\{T_{N}\}_{N=1}^{\infty}$. As pointed out by \cite{Conti_2014} and by \cite{Bertail_2017}, the auxiliary statistic $T_{N}$ can be viewed as the residual of the projection of the Horvitz-Thompson estimator corresponding to the canonical Poisson sampling design on the random sample size from the latter design. Thus, if the goal is to estimate $\overline{f}:=\sum_{i=1}^{N}f(Y_{i})/N$ for some given real-valued function $f$, then the corresponding auxiliary statistic $T_{N}$ can be written as
\begin{equation*}
\begin{split}
T_{N}(f;\mathbf{S}_{N}^{P})&:=\frac{1}{N}\sum_{i=1}^{N}\frac{S_{i,N}^{P}}{p_{i}}f(Y_{i})-\frac{R_{N}(f)}{N}\sum_{i=1}^{N}(S_{i,N}^{P}-p_{i})\\
&:=Y_{N}(f;\mathbf{S}_{N}^{P})-\frac{R(f)}{N}\sum_{i=1}^{N}(S_{i,N}^{P}-p_{i})
\end{split}
\end{equation*}
where $\mathbf{S}_{N}^{P}:=(S_{1,N}^{P}, S_{2,N}^{P}, \dots, S_{N,N}^{P})$ is the vector of sample inclusion indicators for the canonical Poisson sampling design, and where
\begin{equation*}\label{centering_constant}
\begin{split}
R_{N}(f)&:=\frac{Cov\left(Y_{N}(f;\mathbf{S}_{N}^{P}), \sum_{i=1}^{N}S_{i,N}^{P}\right)}{Var\left(\sum_{i=1}^{N}S_{i,N}^{P}\right)}\\
&=\begin{cases}
\frac{1}{d_{N}}\sum_{i=1}^{N}f(Y_{i})(1-p_{i,N}) &\text{if }d_{N}>0\\
0 & \text{otherwise.}
\end{cases}
\end{split}
\end{equation*}
Note that the design expectation of $T_{N}(f;\mathbf{S}_{N}^{P})$ coincides with $\overline{f}$, and that the design variance of $T_{N}(f;\mathbf{S}_{N}^{P})$ must be smaller than that of $Y_{N}(f;\mathbf{S}_{N}^{P})$ unless $R_{N}(f)=0$. However, $T_{N}(f;\mathbf{S}_{N}^{P})$ is not an estimator because it depends on the unknown value of $R_{N}(f)$. At this point, having proved asymptotic normality for the sequence $\{T_{N}(f;\mathbf{S}_{N}^{P})\}_{N=1}^{\infty}$, H\'{a}jek deduces asymptotic normality for the corresponding sequence $\{T_{N}(f;\mathbf{S}_{N}^{P_{0}})\}_{N=1}^{\infty}$ by using Result \ref{Hajek_result_2}. From this he gets asymptotic normality for the sequence $\{Y_{N}(f;\mathbf{S}_{N}^{R})\}_{N=1}^{\infty}$ by proving the following result (note that $Y_{N}(f;\mathbf{S}_{N}^{R})=T_{N}(f;\mathbf{S}_{N}^{R})$ because $\sum_{i=1}^{N}S_{i,N}^{R}=n=\sum_{i=1}^{N}p_{i,N}$):

\begin{result}\label{Hajek_result_3}
Let 
\[B_{N}^{2}(f):=Var(T_{N}(f;\mathbf{S}_{N}^{P}))=\frac{1}{N}\sum_{i=1}^{N}\frac{1-p_{i,N}}{p_{i,N}}\left[f(y_{i})-R_{N}(f)p_{i,N}\right]^{2}.\]
Then, 
\[B_{N}^{-2}(f) E|T_{N}(f;\mathbf{S}_{N}^{P_{0}})-Y_{N}(f;\mathbf{S}_{N}^{R})|^{2}\rightarrow 0\quad\text{ if }d_{N}\rightarrow\infty,\]
where the expectation refers to $\mathfrak{p}_{N}^{R,P_{0}}(\cdot,\cdot;\mathbf{p}_{N})$.
\end{result}

Actually, in his paper \cite{Hajek_1964}, H\'{a}jek did non single out Result \ref{Hajek_result_3} in a dedicated lemma or theorem, but he proved it in the course of the proof of his Theorem 7.1 which establishes asymptotic normality for the sequence $\{Y_{N}(f;\mathbf{S}_{N}^{R})\}_{N=1}^{\infty}$. Note that the statement of the latter theorem considers actually only the case where the $p_{i,N}$'s are proportional to some size variable, but that the proof of Result \ref{Hajek_result_3} given in H\'{a}jek's paper goes through for any sequence of vectors $\{\mathbf{p}_{N}\}_{N=1}^{\infty}$ such that $d_{N}\rightarrow\infty$. From Result \ref{Hajek_result_3} one can finally deduce asymptotic normality for the Horvitz-Thompson estimators corresponding to a sequence of rejective sampling designs by using Result \ref{Hajek_result_1}.

\section{Weak convergence theorems for CPS designs}\label{main_results}

In this section the functional central limit theorems for the rejective sampling case given in \cite{Bertail_2017} will be proven again in somewhat greater generality. Since this requires some assumptions which involve the marginal distribution of the $Y_{i}$ component in $(Y_{i},X_{i})$, it will be convenient to denote the latter distribution by $P_{y}$. As usual in the empirical process literature, the symbol $P_{y}$ will also be used to indicate an operator on the function class $\mathcal{F}$ or on related function classes. For example, $P_{y}$ will also be used to indicate the real-valued function $f\in\mathcal{F}\mapsto\int f(y)dP_{y}(y)$.

The first result of this section settles a measurability issue. It will be used in the rest of this paper without explicitly mentioning it.

\begin{lem}
For $\mathbf{s}_{N}\in\{0,1\}^{N}$ let $\mathfrak{p}_{N}^{R}(\mathbf{s}_{N};\mathbf{X}_{N})$ be the sample selection probabilities for a CPS sampling design, and let $\mathfrak{p}_{N}^{P}(\mathbf{s}_{N};\mathbf{X}_{N})$ be the sample selection probabilities for the corresponding canonical Poisson sampling design. Moreover, let $\underline{\mathbf{\pi}}_{N}$ be the vector of first order sample inclusion probabilities for the CPS design and let $\mathbf{p}_{N}$ be the vector of first order sample inclusion probabilities for the canonical Poisson sampling design. The following statements are equivalent:

\begin{itemize}
\item[(i) ] The CPS design is measurable, i.e. for every $\mathbf{s}_{N}\in\{0,1\}^{N}$ the corresponding function $\mathbf{X}_{N}\mapsto\mathfrak{p}_{N}^{R}(\mathbf{s}_{N};\mathbf{X}_{N})$ is a measurable function from $(\mathcal{X},\mathcal{A})^{N}$ into $\mathbb{R}$;
\item[(ii) ] the vector $\underline{\mathbf{\pi}}_{N}$ is measurable, i.e. the function $\mathbf{X}_{N}\mapsto\underline{\mathbf{\pi}}_{N}$ is a measurable function from $(\mathcal{X},\mathcal{A})^{N}$ into $\mathbb{R}^{N}$;
\item[(iii) ] the vector $\mathbf{p}_{N}$ is measurable, i.e. the function $\mathbf{X}_{N}\mapsto\mathbf{p}_{N}$ is a measurable function from $(\mathcal{X},\mathcal{A})^{N}$ into $\mathbb{R}^{N}$; 
\item[(iv) ] The canonical Poisson sampling design is measurable, i.e. for every $\mathbf{s}_{N}\in\{0,1\}^{N}$ the corresponding function $\mathbf{X}_{N}\mapsto\mathfrak{p}_{N}^{P}(\mathbf{s}_{N};\mathbf{X}_{N})$ is a measurable function from $(\mathcal{X},\mathcal{A})^{N}$ into $\mathbb{R}$;
\end{itemize}
\end{lem}

\begin{proof}
The proofs of the implications (i)$\Rightarrow$(ii), (iii)$\Rightarrow$(iv) and (iv)$\Rightarrow$(i) are easy and the implication (ii)$\Rightarrow$(iii) follows from Theorem 5 on page 67 in  \cite{Tille_2006} which is a special case of a well-known result about exponential families (see for example \cite{Brown_1986})
\end{proof}

The next lemma establishes conditional convergence of the marginal distributions of the sequence of \hyperlink{HTEP}{HTEP}s. 

\begin{lem}[\hyperlink{CWCM}{CWCM}]\label{lem_marginal_convergence_Poisson_process_C}
Let $\{\mathbf{S}_{N}^{R}\}_{N=1}^{\infty}$ be the sequence of vectors of sample inclusion indicators corresponding to a sequence of measurable CPS designs and let $\{\mathbf{p}_{N}\}_{N=1}^{\infty}$ be the sequence of vectors of first order sample inclusion probabilities of the corresponding sequence of canonical Poisson sampling designs. Let $\mathcal{F}$ be a class of measurable functions $f:\mathcal{Y}\mapsto\mathbb{R}$ and let $\{\mathbb{G}_{N}'\}_{N=1}^{\infty}:=\{\{\mathbb{G}_{N}'f: f\in\mathcal{F}\}\}_{N=1}^{\infty}$ be the sequence \hyperlink{HTEP}{HTEP}s corresponding to $\{\mathbf{S}_{N}^{R}\}_{N=1}^{\infty}$ and $\mathcal{F}$. 

Assume that:
\begin{itemize}
\item[\hypertarget{A0}{\textbf{A0}}) ] the sequence $\{\mathbf{p}_{N}\}_{N=1}^{\infty}$ is such that
\[d_{N}:=\sum_{i=1}^{N}p_{i,N}(1-p_{i,N})\overset{P(as)}{\rightarrow}\infty,\]
\item[\hypertarget{A1}{\textbf{A1}}) ] there exists a function $\Sigma':\mathcal{F}\times\mathcal{F}\mapsto\mathbb{R}$ such that
\[\frac{1}{N}\sum_{i=1}^{N}\frac{1-p_{i,N}}{p_{i,N}}[f(Y_{i})-R_{N}(f)p_{i,N}][g(Y_{i})-R_{N}(g)p_{i,N}]\overset{P(as)}{\rightarrow}\Sigma'(f,g)\]
for every $f,g\in\mathcal{F}$,
\item[\hypertarget{A2}{\textbf{A2}}) ] for every finite-dimensional vector $\mathbf{f}:=(f_{1}, f_{2},\dots, f_{r})^{\intercal}\in\mathcal{F}^{r}$ and for every $\epsilon>0$
\[\frac{1}{N}\sum_{i=1}^{N}\frac{\lVert\mathbf{f}(Y_{i})-R_{N}(\mathbf{f})p_{i,N}\rVert^{2}}{p_{i,N}} I(\lVert\mathbf{f}(Y_{i})-R_{N}(\mathbf{f})p_{i,N}\rVert>p_{i,N}\sqrt{N}\epsilon)\overset{P(as)}{\rightarrow}0,\]
where $\lVert\cdot\rVert$ is the euclidean norm on $\mathbb{R}^{r}$, $\mathbf{f}(Y_{i}):=(f_{1}(Y_{i}), f_{2}(Y_{i}),\dots, f_{r}(Y_{i}))^{\intercal}$ and $R_{N}(\mathbf{f}):=(R_{N}(f_{1}), R_{N}(f_{2}), \dots, R_{N}(f_{r}))^{\intercal}$.
\end{itemize}

Then it follows that the function $\Sigma'$ is a positive semidefinite covariance function, and for every finite-dimensional vector $\mathbf{f}\in\mathcal{F}^{r}$ and for every $\mathbf{t}\in\mathbb{R}^{r}$
\begin{equation*}\label{conditional_convergence_chf_CCC}
E_{d}\exp\left(i\mathbf{t}^{\intercal}\mathbb{G}_{N}'\mathbf{f}\right)\overset{P(as)}{\rightarrow} \exp\left(-\frac{1}{2}\mathbf{t}^{\intercal}\Sigma'(\mathbf{f})\mathbf{t}\right),
\end{equation*}
where $\mathbb{G}_{N}'\mathbf{f}:=(\mathbb{G}_{N}'f_{1}, \mathbb{G}_{N}'f_{2}, \dots, \mathbb{G}_{N}'f_{r})^{\intercal}$, and where $\Sigma'(\mathbf{f})$ is the covariance matrix whose elements are given by $\Sigma'_{(ij)}(\mathbf{f}):=\Sigma'(f_{i},f_{j})$, $i,j=1,2,\dots,r$. 
\end{lem}

\begin{proof}
Assume WLOG that $\{\mathbf{S}_{N}^{R}\}_{N=1}^{\infty}$ and the two sequences $\{\mathbf{S}_{N}^{P}\}_{N=1}^{\infty}$ and $\{\mathbf{S}_{N}^{P_{0}}\}_{N=1}^{\infty}$ of the previous subsection are defined in such way that the sequence of pdfs corresponding to $\{\mathbf{S}_{N}^{P}\}_{N=1}^{\infty}$ is given by $\{\mathfrak{p}_{N}^{P}(\cdot;\mathbf{p}_{N})\}_{N=1}^{\infty}$, and such that the sequence of joint pdfs corresponding to $\{(\mathbf{S}_{N}^{R}, \mathbf{S}_{N}^{P_{0}})\}_{N=1}^{\infty}$ is given by $\{\mathfrak{p}_{N}^{R,P_{0}}(\cdot,\cdot;\mathbf{p}_{N})\}_{N=1}^{\infty}$. This can be done in many ways by defining each vector $\mathbf{S}_{N}^{P}$, $\mathbf{S}_{N}^{P_{0}}$ and $\mathbf{S}_{N}^{R}$ as a measurable function of $\mathbf{X}_{N}$ and a single uniform-$[0,1]$ random variable $D$ as described in Section \ref{Section_notation_definitions}. In what follows the sequence of joint pdfs corresponding to $\{(\mathbf{S}_{N}^{P}, \mathbf{S}_{N}^{R}, \mathbf{S}_{N}^{P_{0}})\}_{N=1}^{\infty}$ will not be relevant. 

Now, consider first the sequence of stochastic processes $\{\mathbb{T}_{N}^{P}\}_{N=1}^{\infty}:=\{\{\mathbb{T}_{N}^{P}f: f\in\mathcal{F}\}\}_{N=1}^{\infty}$ with $\mathbb{T}_{N}^{P}f$ defined as
\begin{equation}\label{processo_T_N}
\begin{split}
\mathbb{T}_{N}^{P}f&:=\sqrt{N}\left(T_{N}(f;\mathbf{S}_{N}^{P})-\frac{1}{N}\sum_{i=1}^{N}f(Y_{i})\right)\\
&=\frac{1}{\sqrt{N}}\sum_{i=1}^{N}\left(\frac{S_{i,N}^{P}}{p_{i,N}}-1\right)[f(Y_{i})-R_{N}(f)p_{i,N}]\\
&=\frac{1}{\sqrt{N}}\sum_{i=1}^{N}\left(\frac{S_{i,N}^{P}}{p_{i,N}}-1\right)f(Y_{i})-\frac{R_{N}(f)}{\sqrt{N}}\sum_{i=1}^{N}(S_{i,N}^{P}-p_{i,N})\\
&:=\mathbb{Y}_{N}^{P}f-\mathbb{R}_{N}^{P}f.
\end{split}
\end{equation}
Note that $E_{d}\mathbb{T}_{N}^{P}f=0$ for every $f\in\mathcal{F}$, and that the left side of the display in condition \hyperlink{A1}{A1} is the sequence of covariances $\Sigma_{N}'(f,g):=E_{d}\mathbb{T}_{N}^{P}f\mathbb{T}_{N}^{P}g$. Now, for $\mathbf{f}\in\mathcal{F}^{r}$ consider the triangular array of rowwise conditionally independent random vectors
\[Z_{i,N}\mathbf{f}:=\left(\frac{S_{i,N}^{P}}{p_{i,N}}-1\right)[\mathbf{f}(Y_{i})-R_{N}(\mathbf{f})p_{i,N}],\]
\[i=1,2,\dots, N,\quad N=1,2,\dots.\]
Observe that the random vector $\mathbb{T}_{N}^{P}\mathbf{f}:=(\mathbb{T}_{N}^{P}f_{1}, \mathbb{T}_{N}^{P}f_{2}, \dots, \mathbb{T}_{N}^{P}f_{r})^{\intercal}$ can be written as
\[\mathbb{T}_{N}^{P}\mathbf{f}=\frac{1}{\sqrt{N}}\sum_{i=1}^{N}Z_{i,N}\mathbf{f}.\]
Using the fact that $\Sigma_{N}'(f,g):=E_{d}\mathbb{T}_{N}^{P}f\mathbb{T}_{N}^{P}g\overset{P(as)}{\rightarrow}\Sigma'(f,g)$ along with condition \hyperlink{A2}{A2} it is not difficult to show that the Lindeberg condition
\[\frac{1}{N\mathbf{t}^{\intercal}\Sigma_{N}'(\mathbf{f})\mathbf{t}}\sum_{i=1}^{N}E_{d}\left(\mathbf{t}^{\intercal}Z_{i,N}\mathbf{f}\right)^{2}I\left(|\mathbf{t}^{\intercal}Z_{i,N}\mathbf{f}|>\epsilon\sqrt{\mathbf{t}^{\intercal}\Sigma_{N}'(\mathbf{f})\mathbf{t}}\right)\overset{P(as)}{\rightarrow}0,\quad\epsilon>0,\]
must be satisfied whenever $\mathbf{f}\in\mathcal{F}$ and $\mathbf{t}\in\mathbb{R}^{r}$ are such that $\mathbf{t}^{\intercal}\Sigma'(\mathbf{f})\mathbf{t}>0$. Therefore it follows that
\[E_{d}\exp\left(i\mathbf{t}^{\intercal}\mathbb{T}_{N}^{P}\mathbf{f})\right)\overset{P(as)}{\rightarrow} \exp\left(-\frac{1}{2}\mathbf{t}^{\intercal}\Sigma'(\mathbf{f})\mathbf{t}\right).\]

Next, consider the sequence of stochastic processes $\{\mathbb{T}_{N}^{P_{0}}\}_{N=1}^{\infty}:=\{\{\mathbb{T}_{N}^{P_{0}}f: f\in\mathcal{F}\}\}_{N=1}^{\infty}$ with $\mathbb{T}_{N}^{P_{0}}f$ defined in the same way as $\mathbb{T}_{N}^{P}f$ but with $\mathbf{S}_{N}^{P_{0}}$ in place of $\mathbf{S}_{N}^{P}$. Use assumption \hyperlink{A0}{A0} along with Result \ref{Hajek_result_2} in Section \ref{rejective_sampling_review} to show that
\[\left|E_{d}\exp\left(i\mathbf{t}^{\intercal}\mathbb{T}_{N}^{P}\mathbf{f})\right)-E_{d}\exp\left(i\mathbf{t}^{\intercal}\mathbb{T}_{N}^{P_{0}}\mathbf{f})\right)\right|\overset{P(as)}{\rightarrow}0, \quad \mathbf{f}\in\mathcal{F}^{r}, \mathbf{t}\in\mathbb{R}^{r}.\]
Note that this does not require to know the joint distributions of the vectors $\mathbf{S}_{N}^{P_{0}}$ and $\mathbf{S}_{N}^{P}$. 

Third, consider the sequence of stochastic processes $\{\mathbb{Y}_{N}^{R}\}_{N=1}^{\infty}:=\{\{\mathbb{Y}_{N}^{R}f: f\in\mathcal{F}\}\}_{N=1}^{\infty}$ with $\mathbb{Y}_{N}^{R}f$ defined in the same way as $\mathbb{Y}_{N}^{P}f$ but with $\mathbf{S}_{N}^{R}$ in place of $\mathbf{S}_{N}^{P}$. Use Result \ref{Hajek_result_3} in Section \ref{rejective_sampling_review} to conclude that
\[\left|E_{d}\exp\left(i\mathbf{t}^{\intercal}\mathbb{T}_{N}^{P_{0}}\mathbf{f})\right)-E_{d}\exp\left(i\mathbf{t}^{\intercal}\mathbb{Y}_{N}^{R}\mathbf{f})\right)\right|\overset{P(as)}{\rightarrow}0, \quad \mathbf{f}\in\mathcal{F}^{r}, \mathbf{t}\in\mathbb{R}^{r}\]
as well. 

Finally, note that the definition of $\mathbb{Y}_{N}^{R}$ coincides with the one of $\mathbb{G}_{N}'$ except for the fact that the former contains the first order sample inclusion probabilities corresponding to $\mathbf{S}_{N}^{P}$ in place of those corresponding to $\mathbf{S}_{N}^{R}$, i.e. $\mathbb{Y}_{N}^{R}$ contains $p_{i,N}$ in place of $\pi_{i,N}:=E_{d}S_{i,N}^{R}$. However, this problem can be easily fixed by using Result \ref{Hajek_result_1}.
\end{proof}

\begin{rem}\label{remark_condizione_A2_stella}
If each vector $\mathbf{p}_{N}$ and every $f\in\mathcal{F}$ are measurable (in their respective senses), then condition \hyperlink{A2}{A2} will be certainly satisfied if $P_{y}f^{2}<\infty$ for every $f\in\mathcal{F}$ and 
\begin{itemize}
\item[\hypertarget{A2_stella}{\textbf{A2$^{\mathbf{*}}$}}) ] there exists a constant $L>0$ such that $\min_{1\leq i\leq N}p_{i,N}>L$ with probability tending to $1$ (eventually almost surely).
\end{itemize}
\end{rem}

%
%

Now, Lemma \ref{lem_marginal_convergence_Poisson_process_C} provides sufficient conditions for convergence of the finite-dimensional marginal distributions of the sequence of \hyperlink{HTEP}{HTEP}s, but in order to establish (conditional) weak convergence in $l^{\infty}(\mathcal{F})$ for infinite function classes $\mathcal{F}$ it must still be shown that sequence of \hyperlink{HTEP}{HTEP}s is (conditionally) asymptotically tight in $l^{\infty}(\mathcal{F})$ (for unconditional weak convergence this follows from Theorem 1.5.4 on page 35 in \cite{vdVW}, while for conditional weak convergence this follows from Theorem \ref{sufficient_conditions_for_CWC} in Section \ref  {general_weak_convergence_theory}). By Theorem 1.5.7 on page 37 in \cite{vdVW} (Theorem \ref{Theorem_CAT_characterizations} in Section \ref  {general_weak_convergence_theory}) this can be done by showing that there exists a semimetric $\rho$ for which $\mathcal{F}$ is totally bounded and for which the \hyperlink{HTEP}{HTEP} sequence is (conditionally) asymptotically equicontinuous (henceforth AEC). In this paper the choice of the semimetric $\rho$ will depend on the definition of the first order sample inclusion probabilities. In the next subsection it will be seen that if the first order sample inclusion probabilities are bounded away from zero, it is convenient to consider the $L_{2}(P_{y})$-semimetric
\[\rho(f,g):=\sqrt{P_{y}(f-g)^{2}},\quad f,g\in\mathcal{F}.\]
The subsequent subsection will then treat the case where the first order sample inclusion probabilities are proportional to some size variable which might take on arbitrarily small values. For that case another semimetric will be used.

\subsection{CPS designs with a positive lower bound on the first order sample inclusion probabilities}

The next lemma provides sufficient conditions which make sure that $\mathcal{F}$ is totally bounded w.r.t. the $L_{2}(P_{y})$ semimetric $\rho$ and that the sequence of \hyperlink{HTEP}{HTEP}s \hyperlink{conditional_AEC}{conditionally AEC} w.r.t. $\rho$.

\begin{lem}[Total boundedness and \hyperlink{conditional_AEC}{conditional AEC}]\label{AEC_case_C}
Let $\{\mathbf{S}_{N}^{R}\}_{N=1}^{\infty}$ and $\{\mathbf{p}_{N}\}_{N=1}^{\infty}$ be defined as in Lemma \ref{lem_marginal_convergence_Poisson_process_C}, let $\mathcal{F}$ be a class of measurable functions $f:\mathcal{Y}\mapsto\mathbb{R}$ and let $\{\mathbb{G}_{N}'\}_{N=1}^{\infty}:=\{\{\mathbb{G}_{N}'f: f\in\mathcal{F}\}\}_{N=1}^{\infty}$ be the sequence \hyperlink{HTEP}{HTEP}s corresponding to $\{\mathbf{S}_{N}^{R}\}_{N=1}^{\infty}$ and $\mathcal{F}$. Assume that condition \hyperlink{A2_stella}{A2$^{*}$} holds and that assumptions
\begin{itemize}
\item[\hypertarget{GC}{\textbf{GC}}) ] $(\mathcal{F}_{\infty})^{2}:=\{(f-g)^{2}:f,g\in\mathcal{F}\}$ is an outer almost sure $P_{y}$-Glivenko-Cantelli class
\item[\hypertarget{F1}{\textbf{F1}}) ] $\mathcal{F}$ has an envelope function $F$ such that $P_{y}^{*}F^{2}<\infty$ and such that the uniform entropy condition
\begin{equation}\label{uniform_entropy}
\int_{0}^{\infty}\sup_{Q_{y}}\sqrt{\log N(\epsilon\lVert F\rVert_{L_{2}(Q_{y})}, \mathcal{F}, L_{2}(Q_{y}))}d\epsilon<\infty
\end{equation}
holds. In the last display the supremum is taken over all finitely discrete probability measures $Q_{y}$ on $\mathcal{Y}$ such that $\lVert F\rVert_{L_{2}(Q_{y})}:=\int F^{2}dQ_{y}>0$;
\end{itemize}
Then it follows that
\begin{itemize}
\item[(i) ] $\mathcal{F}$ is totally bounded w.r.t. $\rho$;
\item[(ii) ]  
\[E_{d}\lVert \mathbb{G}_{N}'\rVert_{\mathcal{F}_{\delta_{N}}}\overset{P*(as*)}{\rightarrow}0\quad\text{ for every }\delta_{N}\downarrow 0,\]
where
\[\mathcal{F}_{\delta}:=\{f-g: f,g \in\mathcal{F}\wedge \rho(f,g)<\delta\},\quad\delta>0.\]
\end{itemize}
\end{lem}

\begin{proof}
Part (i) of the conclusion follows from condition \hyperlink{F1}{F1} (see Problem 2.5.1 on page 133 in \citep{vdVW}.

The proof of part (ii) of the conclusion is essentially the same as the proof of \hyperlink{conditional_AEC}{conditional AEC} for the rejective sampling case given in \cite{Bertail_2017} (see pages 12-13 in the supplement to that paper) but it corrects a little mistake in the final part of that proof. First it will be shown that for arbitrary $\delta_{N}\downarrow0$ the corresponding stochastic processes $\{\mathbb{G}_{N}'f:f\in\mathcal{F}_{\delta_{N}}\}$ are, with probability tending to $1$ (or eventually almost surely), conditionally subgaussian w.r.t. the empirical semimetric
\[\rho_{N}(f,g):=\sqrt{\mathbb{P}_{y,N}(f-g)^{2}}:=\sqrt{\frac{1}{N}\sum_{i=1}^{N}[f(Y_{i})-g(Y_{i})]^{2}},\quad f,g\in\mathcal{F}_{\delta_{N}},\]
i.e. it will be shown that there exists a constant $C>0$ (which does not depend on the sample points $\omega\in\Omega_{y,x}^{\infty}\times\Omega_{d}$ and neither on $N$) such that, with probability tending to $1$ (or eventually almost surely),
\begin{equation}\label{subgaussian_inequality}
P_{d}\left\{|\mathbb{G}_{N}'f-\mathbb{G}_{N}'g|>x\right\}\leq 2 e^{-Cx^{2}/\rho_{N}^{2}(f,g)}\quad\text{for every }f,g\in\mathcal{F}_{\delta_{N}}.
\end{equation}
To this aim, write
\begin{equation*}
\begin{split}
\mathbb{G}_{N}'f-\mathbb{G}_{N}'g&=\frac{1}{\sqrt{N}}\sum_{i=1}^{N}\frac{S_{i,N}^{R}-\pi_{i,N}}{\pi_{i,N}}[f(Y_{i})-g(Y_{i})]:=\sum_{i=1}^{N}Z_{i,N}(f,g),
\end{split}
\end{equation*}
so that, for every $x,\lambda>0$,
\begin{equation}\label{Chernoff_1}
\begin{split}
P_{d}\left\{\mathbb{G}_{N}'f-\mathbb{G}_{N}'g>x\right\}&=P_{d}\left\{e^{\lambda \sum_{i=1}^{N}Z_{i,N}(f,g)}>e^{\lambda x}\right\}\\
&\leq e^{-\lambda x} E_{d}e^{\lambda \sum_{i=1}^{N}Z_{i,N}(f,g)}
\end{split}
\end{equation}
by Markov's inequality. Now, note that by Theorem 2.8 in \cite{Joag-Dev_1983} the components of $\mathbf{S}_{N}^{R}$ are negatively associated, and hence it follows that 
\begin{equation}\label{Chernoff_2}
E_{d}e^{\lambda \sum_{i=1}^{N}Z_{i,N}(f,g)}=E_{d}\prod_{i=1}^{N}e^{\lambda Z_{i,N}(f,g)}\leq \prod_{i=1}^{N}E_{d}e^{\lambda Z_{i,N}(f,g)}.
\end{equation}
Since $E_{d}Z_{i,N}=0$ for every $i=1,2,\dots, N$ and for every $N=1,2,\dots$, and since
\[-|f(Y_{i})-g(Y_{i})|\leq \sqrt{N}Z_{i,N}\leq\frac{1-\pi_{i,N}}{\pi_{i,N}}|f(Y_{i})-g(Y_{i})|,\]
it follows from Hoeffding's lemma (see \cite{Hoeffding_1963}) that
\begin{equation*}
\begin{split}
E_{d}e^{\lambda Z_{i,N}(f,g)}\leq \exp\left\{\frac{\lambda^{2}}{8N\pi_{i,N}^{2}}[f(Y_{i})-g(Y_{i})]^{2}\right\}.
\end{split}
\end{equation*}
By assumption \hyperlink{A2_stella}{A2$^{*}$} and Result \ref{Hajek_result_1} from H{\'a}jek's paper the right side does not exceed 
\[\exp\left\{\frac{\lambda^{2}}{8N L^{2}}[f(Y_{i})-g(Y_{i})]^{2}\right\}\]
with probability tending to one (or eventually almost surely). In combination with (\ref{Chernoff_1}) and (\ref{Chernoff_2}) this shows that
\[P_{d}\left\{\mathbb{G}_{N}'f-\mathbb{G}_{N}'g>x\right\}\leq \exp\left\{-\lambda x+\frac{\lambda^{2}}{8 L^{2}}\rho_{N}^{2}(f,g)\right\}\]
with probability tending to $1$ (or eventually almost surely). Combining this inequality with the same inequality for $\mathbb{G}_{N}'g-\mathbb{G}_{N}'f$ shows that
\[P_{d}\left\{|\mathbb{G}_{N}'f-\mathbb{G}_{N}'g|>x\right\}\leq 2\exp\left\{-\lambda x+\frac{\lambda^{2}}{8 L^{2}}\rho_{N}^{2}(f,g)\right\}\]
with probability tending to $1$ (or eventually almost surely). Finally, optimizing the right side w.r.t. $\lambda>0$ yields the subgaussian inequality in (\ref{subgaussian_inequality}) with $C=2L^{2}$.

\smallskip

Now, note that 
\[N(\epsilon, \mathcal{F}_{\delta_{N}}, L_{2}(\mathbb{P}_{y,N}))\leq N(\epsilon, \mathcal{F}_{\infty}, L_{2}(\mathbb{P}_{y,N}))\leq N^{2}(\epsilon/2, \mathcal{F}, L_{2}(\mathbb{P}_{y,N})),\]
and that the uniform entropy condition in assumption \hyperlink{F1}{F1} implies that for every $\epsilon>0$ the square of the corresponding entropy number on the far right must be finite. From this it follows that $\mathcal{F}_{\delta_{N}}$ contains a countable subset $\mathcal{G}_{\delta_{N}}(\mathbf{Y}_{N})$ (note that this subset depends on $\mathbf{Y}_{N}$) such that
\[\left\lVert\mathbb{G}_{N}'\right\rVert_{\mathcal{F}_{\delta_{N}}}=\left\lVert\mathbb{G}_{N}'\right\rVert_{\mathcal{G}_{\delta_{N}}(\mathbf{Y}_{N})}.\]
As a consequence, the stochastic process $\{\mathbb{G}_{N}'f:f\in\mathcal{F}_{\delta_{N}}\}$ is separable in the sense required for an application of Corollary 2.2.8 on page 101 in \cite{vdVW} with respect to the sample design distribution of the process. Since it has already be shown that the sub-Gaussian inequality (\ref{subgaussian_inequality}) holds with probability tending to $1$ (or eventually almost surely), it follows by the second part of the conclusion of the just cited corollary that there exists a constant $K>0$ (which does not depend on the sample points $\omega\in\Omega_{y,x}^{\infty}\times\Omega_{d}$ and neither on $N$) such that
\begin{equation}\label{disugualglianza_corretta}
E_{d} \left\lVert\mathbb{G}_{N}'\right\rVert_{\mathcal{F}_{\delta_{N}}}\leq K\int_{0}^{\infty}\sqrt{\log D(\epsilon,\mathcal{F}_{\delta_{N}}, L_{2}(\mathbb{P}_{y,N}))}d\epsilon
\end{equation}
with inner probability tending to one (or eventually inner almost surely), where $D(\epsilon,\mathcal{F}_{\delta_{N}},L_{2}(\mathbb{P}_{y,N}))$ denotes the packing number, i.e. the cardinality of the largest subset $\mathcal{H}$ of $\mathcal{F}_{\delta_{N}}$ such that $\rho_{N}(f,g)>\epsilon$ for every $f,g\in\mathcal{H}$. Since 
\[D(\epsilon,\mathcal{F}_{\delta_{N}}, L_{2}(\mathbb{P}_{y,N}))\leq N(\epsilon/2,\mathcal{F}_{\delta_{N}},L_{2}(\mathbb{P}_{y,N})),\]
it follows that the right side in (\ref{disugualglianza_corretta}) is bounded by a constant multiple of
\[\int_{0}^{\infty}\sqrt{\log N(\epsilon,\mathcal{F}_{\delta_{N}},L_{2}(\mathbb{P}_{y,N}))}d\epsilon.\]
The proof can now be completed by using assumptions \hyperlink{GC}{GC} and \hyperlink{F1}{F1} in order to show that the latter integral goes to zero outer almost surely. This can be done as in the proof of Theorem 2.5.2 on page 127 in \cite{vdVW} (see the lines following display (2.5.3) on page 128 in \cite{vdVW}; see also Remark \ref{rem_ass_M} below in order to see that assumption \hyperlink{GC}{GC} can be replaced with a measurability condition).
\end{proof}

\begin{rem}\label{rem_ass_M}
In the proof of Theorem 2.5.2 on page 127 in \cite{vdVW} it is shown that condition \hyperlink{F1}{F1} together with condition
\begin{itemize}
\item[\hypertarget{M1}{\textbf{M1}}) ] $(\mathcal{F}_{\infty})^{2}:=\{(f-g)^{2}:f,g\in\mathcal{F}\}$ is a $P_{y}$-measurable class of functions (see Definition 2.3.3 on page 110 in \cite{vdVW}), i.e. the function
\[\mathbf{Y}_{N}\mapsto \sup_{f\in(\mathcal{F}_{\infty})^{2}} \left|\sum_{i=1}^{N}e_{i} f(Y_{i})\right|\]
is measurable on the completion of $(\mathcal{Y}^{N}, \mathcal{A}^{N}, P_{y}^{N})$ for every $N$ and for every $(e_{1}, e_{2},\dots, e_{N})\in\mathbb{R}^{N}$
\end{itemize}
imply condition \hyperlink{GC}{GC}. Moreover, in the proof of Theorem 2.5.2 on page 127 in \cite{vdVW} assumption \hyperlink{M1}{M1} is used only for this purpose. Thus, the proof of Theorem 2.5.2 on page 127 in \cite{vdVW} does actually show that assumptions \hyperlink{F1}{F1}, \hyperlink{GC}{GC} and
\begin{itemize}
\item[\hypertarget{M2}{\textbf{M2}}) ] for every $\delta>0$ the corresponding function class $\mathcal{F}_{\delta}$ is a $P_{y}$-measurable class of functions, i.e. the function
\[\mathbf{Y}_{N}\mapsto \sup_{f\in\mathcal{F}_{\delta}} \left|\sum_{i=1}^{N}e_{i} f(Y_{i})\right|\]
is measurable on the completion of $(\mathcal{Y}^{N}, \mathcal{A}^{N}, P_{y}^{N})$ for every $N$ and for every $(e_{1}, e_{2},\dots, e_{N})\in\mathbb{R}^{N}$
\end{itemize}
imply that $\mathcal{F}$ is a $P_{y}$-Donsker class. 
\end{rem}

\begin{rem}\label{condition_PM}
It is not difficult to show that condition 
\begin{itemize}
\item[\hypertarget{PM}{\textbf{PM}}) ] $\mathcal{F}$ is a \textit{pointwise measurable} class of functions, i.e. $\mathcal{F}$ contains a countable subset $\mathcal{G}$ such that for every $f\in\mathcal{F}$ there exists a sequence $\{g_{m}\}_{m=1}^{\infty}$ of functions $g_{m}\in\mathcal{G}$ such that $f$ is the pointwise limit of $\{g_{m}\}_{m=1}^{\infty}$ (see Example 2.3.4 on page 110 in \cite{vdVW})
\end{itemize}
implies condition \hyperlink{M1}{M1} as well as condition \hyperlink{M2}{M2}.
\end{rem}

\begin{rem}
The FCLT for the rejective sampling case given in \cite{Bertail_2017} (Theorem 3.2 on page 105 of that paper) imposes neither assumption \hyperlink{M1}{M1} nor assumption \hyperlink{GC}{GC}. However, there is a mistake in the proof of \hyperlink{conditional_AEC}{conditional AEC} given in \cite{Bertail_2017}. In fact, inequality (S3) on page 13 in the supplement to \cite{Bertail_2017} is false in general. According to the first inequality in the conclusion of Corollary 2.2.8 on page 101 in \cite{vdVW}, which was used by the authors of \cite{Bertail_2017} in order to obtain inequality (S3), the left hand side of inequality (S3) should actually be
\[E_{d}\sup_{f,g\in\mathcal{F}_{\delta}:\rho_{N}(f,g)<\delta}\left|\mathbb{G}_{N}'f-\mathbb{G}_{N}'f\right|\]
rather than
\[E_{d}\sup_{f,g\in\mathcal{F}:\rho_{N}(f,g)<\delta}\left|\mathbb{G}_{N}'f-\mathbb{G}_{N}'f\right|\]
with the function class $\mathcal{F}$ in place of $\mathcal{F}_{\delta}$. As a consequence, the proof of \hyperlink{conditional_AEC}{conditional AEC} given in \cite{Bertail_2017} shows actually that
\[E_{d}\sup_{f,g\in\mathcal{F}_{\delta_{N}}:\rho_{N}(f,g)<\delta_{N}}\left|\mathbb{G}_{N}'f-\mathbb{G}_{N}'f\right|\overset{as*}{\rightarrow}0\quad\text{for every }\delta_{N}\downarrow0\]
which does not imply \hyperlink{conditional_AEC}{conditional AEC}. In order to obtain \hyperlink{conditional_AEC}{conditional AEC}, the authors of \cite{Bertail_2017} should have used the second inequality in the conclusion of Corollary 2.2.8 on page 101 in \cite{vdVW} rather than the first one. In this way they would have obtained inequality (\ref{disugualglianza_corretta}) instead of inequality (S3), and in order to prove that under condition \hyperlink{F1}{F1} the right side of (\ref{disugualglianza_corretta}) goes to zero outer almost surely some additional assumption seems to be necessary (cf. the proof of Theorem 2.5.2 on page 127 in \cite{vdVW}).
\end{rem}

As already pointed out in \cite{Pasquazzi_2019} (cf. Remark \ref{rem_measurable_suprema}), \textit{conditional} AEC w.r.t. to a given semimetric $\rho$ and not even the conclusion of Lemma \ref{AEC_case_C} (which is certainly stronger than \hyperlink{conditional_AEC}{conditional AEC}) \textit{seem} to be strong enough to imply \textit{unconditional} AEC, which for the \hyperlink{HTEP}{HTEP} sequence $\{\{\mathbb{G}_{N}'f:f\in\mathcal{F}\}\}_{N=1}^{\infty}$ can be defined as 
\begin{equation}\label{def_unconditional_AEC}
P^{*}\left\{\lVert\mathbb{G}_{N}'\rVert_{\mathcal{F}_{\delta_{N}}}>\epsilon\right\}\rightarrow 0\quad\text{ for every }\epsilon>0\text{ and for every }\delta_{N}\downarrow 0,
\end{equation}
where $P$ is the product probability measure $P_{y,x}^{\infty}\times P_{d}$ (cf. the equivalent definition of \hyperlink{conditional_AEC}{conditional AEC} given in Remark \ref{remark_conditional_AEC}). The problem is that for uncountable function classes $\mathcal{F}$ the random functions $\lVert\mathbb{G}_{N}'\rVert_{\mathcal{F}_{\delta}}$, $\delta>0$, might be non measurable and that $\lVert\mathbb{G}_{N}'\rVert_{\mathcal{F}_{\delta_{N}}}$ might therefore be strictly smaller than $\lVert\mathbb{G}_{N}'\rVert_{\mathcal{F}_{\delta_{N}}}^{*}$ with positive inner probability. As a consequence, the $P_{d}$-probabilities on the left side in (\ref{def_conditional_AEC}) might not be conditional probabilities in the proper sense and condition (\ref{def_unconditional_AEC}) might therefore fail even though condition (\ref{def_conditional_AEC}) is satisfied (note that this is consistent with the conjecture that \hyperlink{oasCWC}{oasCWC} does not imply unconditional weak convergence; see Remark \ref{rem_measurable_suprema}). To be safe, in order to deduce \textit{unconditional} AEC from \textit{conditional} AEC condition \hyperlink{PM}{PM} will be used in this paper. In fact, condition \hyperlink{PM}{PM} makes sure that the random functions $\lVert\mathbb{G}_{N}'\rVert_{\mathcal{F}_{\delta}}$, $\delta>0$, are all measurable and that the $P_{d}$-probabilities on the left side in (\ref{def_conditional_AEC}) are therefore conditional probabilities in the proper sense.

Now, combining the sufficient conditions for \hyperlink{CWCM}{CWCM} with those for total boundedness and \hyperlink{conditional_AEC}{conditional AEC} yields the following weak convergence results:

\begin{thm}[conditional weak convergence]\label{opCWC_case_A}
Let $\{\mathbf{S}_{N}^{R}\}_{N=1}^{\infty}$, $\{\mathbf{p}_{N}\}_{N=1}^{\infty}$, $\mathcal{F}$ and $\{\mathbb{G}_{N}'\}_{N=1}^{\infty}:=\{\{\mathbb{G}_{N}'f: f\in\mathcal{F}\}\}_{N=1}^{\infty}$ be defined as in Lemma \ref{lem_marginal_convergence_Poisson_process_C}. Assume that conditions \hyperlink{A0}{A0}, \hyperlink{A1}{A1}, \hyperlink{A2_stella}{A2$^{*}$}, \hyperlink{GC}{GC} and \hyperlink{F1}{F1} are satisfied. Then it follows that
\begin{itemize}
\item[(i) ] there exists a zero-mean Gaussian process $\{\mathbb{G}'f:f\in\mathcal{F}\}$ with covariance function given by $\Sigma'$ which is a Borel measurable and tight mapping from some probability space into $l^{\infty}(\mathcal{F})$ such that
\[\sup_{h\in BL_{1}(l^{\infty}(\mathcal{F}))}\left|E_{d}h(\mathbb{G}_{N}')-Eh(\mathbb{G}')\right|\overset{P*(as*)}{\rightarrow}0,\]
\item[(ii) ] the sample paths $f\mapsto\mathbb{G}'f$ are uniformly continuous w.r.t. the $L_{2}(P_{y})$ semimetric $\rho(f,g):=[P_{y}(f-g)^{2}]^{1/2}$ with probability $1$.
\end{itemize}
\end{thm}

\begin{proof}
Assumptions \hyperlink{A0}{A0}, \hyperlink{A1}{A1}, \hyperlink{A2_stella}{A2$^{*}$} make sure that \hyperlink{CWCM}{CWCM} holds for some zero-mean Gaussian limit process with covariance function given by $\Sigma'$ (see Lemma \ref{lem_marginal_convergence_Poisson_process_C} and Remark \ref{remark_condizione_A2_stella}), while assumptions \hyperlink{A2_stella}{A2$^{*}$}, \hyperlink{GC}{GC} and \hyperlink{F1}{F1} imply that $\mathcal{F}$ is totally bounded w.r.t. $\rho$ and that $\{\mathbb{G}_{N}'\}_{N=1}^{\infty}$ is \hyperlink{conditional_AEC}{conditionally AEC} w.r.t. $\rho$ (see Lemma \ref{AEC_case_C}). Both conclusions of the theorem follow now from Corollary \ref{cor_uniformly_continuous_limit_process}.
\end{proof}

\begin{thm}[Unconditional weak convergence]\label{unconditional_convergence}
Let $\{\mathbf{S}_{N}^{R}\}_{N=1}^{\infty}$, $\{\mathbf{p}_{N}\}_{N=1}^{\infty}$, $\mathcal{F}$ and $\{\mathbb{G}_{N}'\}_{N=1}^{\infty}:=\{\{\mathbb{G}_{N}'f: f\in\mathcal{F}\}\}_{N=1}^{\infty}$ be defined as in Lemma \ref{lem_marginal_convergence_Poisson_process_C}. Assume that conditions \hyperlink{A0}{A0}, \hyperlink{A1}{A1}, \hyperlink{A2_stella}{A2$^{*}$}, \hyperlink{F1}{F1} and \hyperlink{PM}{PM} are satisfied. Then it follows that 
\begin{itemize}
\item[(i) ] there exists zero-mean Gaussian process $\{\mathbb{G}'f:f\in\mathcal{F}\}$ with covariance function given by $\Sigma'$ which is a Borel measurable and tight mapping from some probability space into $l^{\infty}(\mathcal{F})$ such that
\[\mathbb{G}_{N}'\rightsquigarrow\mathbb{G}'\quad\text{ in }l^{\infty}(\mathcal{F});\]
\item[(ii) ] the sample paths $f\mapsto\mathbb{G}'f$ are uniformly continuous w.r.t. the $L_{2}(P_{y})$ semimetric $\rho(f,g):=[P_{y}(f-g)^{2}]^{1/2}$ with probability $1$.
\end{itemize}
\end{thm}

\begin{proof}
Remark \ref{rem_ass_M} and Remark \ref{condition_PM} show that assumption \hyperlink{F1}{F1} along with assumption \hyperlink{PM}{PM} imply assumption \hyperlink{GC}{GC}. The conditions of the present theorem are therefore stronger than the conditions of Theorem \ref{opCWC_case_A}, and the conclusion of the present theorem follows therefore from Remark \ref{rem_measurable_suprema} (note that condition \hyperlink{PM}{PM} implies measurability of the suprema in Remark \ref{rem_measurable_suprema}).
\end{proof}

The following corollary establishes joint weak convergence for the sequence of \hyperlink{HTEP}{HTEP}s and the classical sequence of $\mathcal{F}$-indexed i.i.d. empirical processes given by 
\begin{equation}\label{def_classical_empirical_process}
\mathbb{G}_{N}f:=\frac{1}{\sqrt{N}}\sum_{i=1}^{N}(f(Y_{i})-P_{y}f),\quad f\in\mathcal{F}.
\end{equation}

\begin{cor}[Joint weak convergence]\label{corollary_joint_weak_convergence_C}
Under the assumptions of Theorem \ref{unconditional_convergence} it follows that
\[(\mathbb{G}_{N}, \mathbb{G}_{N}')\rightsquigarrow(\mathbb{G}, \mathbb{G}')\text{ in }l^{\infty}(\mathcal{F})\times l^{\infty}(\mathcal{F}),\]
where $\mathbb{G}'$ is defined as in Theorem \ref{opCWC_case_A} (or Theorem \ref{unconditional_convergence}), $\mathbb{G}_{N}$ is the classical $\mathcal{F}$-indexed empirical process defined in (\ref{def_classical_empirical_process}), and where $\mathbb{G}$ is a Borel measurable and tight $P_{y}$-Brownian Bridge which is independent from $\mathbb{G}'$.
\end{cor}

\begin{proof}
The assumptions of Theorem \ref{unconditional_convergence} are stronger than those of Theorem \ref{opCWC_case_A} (which imply \hyperlink{opCWC}{opCWC}) and they imply that $\mathcal{F}$ is a $P_{y}$-Donsker class (see Remark \ref{rem_ass_M} and Remark \ref{condition_PM}). The proof of the corollary follows now from an application of Theorem \ref{teorema_generale_joint_weak_convergence}. 
\end{proof}


\subsection{CPS designs with first order sample inclusion probabilities proportional to some size variable which might take on arbitrarily small values}

This subsection treats the case where the first order sample inclusion probabilities are proportional to some size variable which can take on values arbitrarily close to zero. Note that this case is not covered by the theorems given in the previous subsection because assumptions \hyperlink{A0}{A0} and \hyperlink{A2_stella}{A2$^{*}$} imply that the first order sample inclusion probabilities are bounded away from zero with probability tending to $1$ or eventually almost surely (see Result \ref{Hajek_result_1} in Section \ref{rejective_sampling_review}). So, let $w:\mathcal{X}\mapsto(0,\infty)$ be a mapping such that $w(X_{i})$ can be interpreted as the "size" of the $i$th population unit. Throughout this subsection it will be assumed that the first order sample inclusion probabilities are defined as
\begin{equation}\label{first_order_Poisson_special_case}
\pi_{i,N}:=\min\left\{c_{N}(X_{1}, X_{2},\dots, X_{N}) \frac{w(X_{i})}{\sum_{j=1}^{N}w(X_{j})}; 1\right\},\quad i=1,2,\dots, N,
\end{equation}
where $c_{N}:\mathcal{X}^{N}\mapsto(0,\infty)$ is a function which makes sure that the expected sample size equals the value taken on by some other integer-valued function $n_{N}:\mathcal{X}^{N}\mapsto\{1,2,\dots, N\}$ (in many applications $\{n_{N}\}_{N=1}^{\infty}$ is simply a deterministic sequence of positive integers), i.e. $c_{N}$ makes sure that
\begin{equation}\label{expected_sample_size}
\sum_{i=1}^{N}\pi_{i,N}:=\sum_{i=1}^{N}\min\left\{c_{N}\frac{w(X_{i})}{\sum_{j=1}^{N}w(X_{j})}; 1\right\} =n_{N}.
\end{equation} 
It is not difficult to show that the function $c_{N}$ is well defined, i.e. that for every $n_{N}\in[0,N]$ there exists a unique positive constant $c_{N}$ such that equation (\ref{expected_sample_size}) holds. Moreover, under the assumptions 
\begin{itemize}
\item[\hypertarget{B0}{\textbf{B0}}) ] $n_{N}:\mathcal{X}^{N}\mapsto[0,N]$ is a measurable function and the sequence of expected sample sizes $\{n_{N}\}_{N=1}^{\infty}$ is such that
\[\frac{n_{N}}{N}\overset{P(as)}{\rightarrow}\alpha\in(0,1),\]
\item[\hypertarget{B1}{\textbf{B1}}) ] $w:\mathcal{X}\mapsto(0,\infty)$ is a measurable function such that $Ew(X_{1})<\infty$,
\end{itemize}
it can also be shown that $c_{N}$ is measurable and that $c_{N}/N\rightarrow \theta$ in probability (almost surely), where $\theta$ is the unique (positive) constant such that
\[E\min\left\{\frac{\theta w(X_{1})}{Ew(X_{2})}; 1\right\}=\alpha.\]
The details of the proof of the latter claim are left to the reader.

Now, in order obtain weak convergence theorems for the case where the first order sample inclusion probabilities are defined as in (\ref{first_order_Poisson_special_case}) it will be convenient to proceed as in Subsection 3.2 of \cite{Pasquazzi_2019} and to place restrictions on the class of functions
\[\mathcal{F}/w_{\theta}:=\{f/w_{\theta}: f\in\mathcal{F}\},\]
where $\mathcal{F}$ is the original class of interest, and where
\[w_{\theta}(X_{1}):=\min\{w(X_{1}), Ew(X_{1})/\theta\}.\]
Note that the domain of the members of the class $\mathcal{F}/w_{\theta}$ is the range of the random vectors $(Y_{i},X_{i})$ (which is assumed to be $\mathcal{Y}\times\mathcal{X}$), and that the value taken on by $f/w_{\theta}\in \mathcal{F}/w_{\theta}$ at a given realization of the random vector $(Y_{i},X_{i})$ is given by $f/w_{\theta}(Y_{i},X_{i}):=f(Y_{i})/w_{\theta}(X_{i})$.

The following lemma establishes \hyperlink{CWCM}{CWCM} for the \hyperlink{HTEP}{HTEP} sequence for the case where the first order sample inclusion probabilities are defined as in (\ref{first_order_Poisson_special_case}).

\begin{lem}[\hyperlink{CWCM}{CWCM}]\label{marginal_convergence_case_C_B}
Let $\{\mathbf{\underline{\pi}}_{N}\}_{N=1}^{\infty}$ be the sequence of vectors of first order sample inclusion probabilities for a sequence of CPS designs and let $\{\mathbf{p}_{N}\}_{N=1}^{\infty}$ be the sequence of vectors of first order sample inclusion probabilities for the corresponding sequence of canonical Poisson sampling designs. Assume that the components of each vector $\mathbf{\underline{\pi}}_{N}$ are defined as in (\ref{first_order_Poisson_special_case}) and that conditions \hyperlink{B0}{B0}, \hyperlink{B1}{B1} and condition 
\begin{itemize}
\item[\hypertarget{B2}{\textbf{B2}}) ] the members of $\mathcal{F}/w_{\theta}$ are square integrable, i.e. $E[f(Y_{1})/w_{\theta}(X_{1})]^{2}<\infty$ for every $f\in\mathcal{F}$
\end{itemize}
hold. Then it follows that conditions \hyperlink{A0}{A0}, \hyperlink{A1}{A1}, and \hyperlink{A2}{A2} of Lemma \ref{lem_marginal_convergence_Poisson_process_C} are satisfied and that the covariance function $\Sigma'$ in condition \hyperlink{A1}{A1} is given by
\begin{equation}\label{covariance_function_case_B}
\begin{split}
\Sigma'(f,g):=E&\left\{\frac{Ew(X_{2})-\theta w_{\theta}(X_{1})}{\theta w_{\theta}(X_{1})}\left[f(Y_{1})-R(f)\frac{\theta w_{\theta}(X_{1})}{Ew(X_{2})}\right]\times\right.\\
&\quad\left.\times\left[g(Y_{1})-R(g)\frac{\theta w_{\theta}(X_{1})}{Ew(X_{2})}\right]\right\},\quad f,g\in\mathcal{F},
\end{split}
\end{equation}
with
\[R(f):=\frac{Ef(Y_{1})\left(1-\frac{\theta w_{\theta}(X_{1})}{Ew(X_{2})}\right)}{E\frac{\theta w_{\theta}(X_{1})}{Ew(X_{2})}\left(1-\frac{\theta w_{\theta}(X_{1})}{Ew(X_{2})}\right)},\quad f\in\mathcal{F}.\]
\end{lem}

\begin{proof}
Define
\begin{equation}\label{definizione_w_N}
w_{N}(X_{i}):=\min\left\{w(X_{i});\frac{1}{c_{N}}\sum_{j=1}^{N}w(X_{j})\right\},\quad i=1,2,\dots, N,
\end{equation}
and note that
\begin{equation}\label{first_order_incl_w_N} 
\pi_{i,N}=\frac{c_{N}w_{N}(X_{i})}{\sum_{j=1}^{N}w(X_{j})},\quad i=1,2,\dots, N.
\end{equation}
Next, note that 
\[\left|\frac{w_{N}(X_{i})}{w_{\theta}(X_{i})}-1\right|=\frac{|w_{N}(X_{i})-w_{\theta}(X_{i})|}{w_{\theta}(X_{i})}\]
and that the right hand side is positive only if $w(X_{i})$ lies between $\sum_{j=1}^{N}w(X_{j})/c_{N}$ and $Ew(X_{1})/\theta$, in which case it is bounded by
\[\frac{1}{\min\left\{\frac{1}{c_{N}}\sum_{j=1}^{N}w(X_{j});\frac{1}{\theta}Ew(X_{1})\right\}}\left|\frac{1}{c_{N}}\sum_{j=1}^{N}w(X_{j})-\frac{1}{\theta}Ew(X_{1})\right|.\]
Since this bound does not depend on $i$, and since under assumptions \hyperlink{B0}{B0} and \hyperlink{B1}{B1} it goes to zero in probability (almost surely), it follows that
\begin{equation}\label{uniform_convergence_w_N_i}
\max_{1\leq i\leq N}\left|\frac{w_{N}(X_{i})}{w_{\theta}(X_{i})}-1\right|\overset{P(as)}{\rightarrow}0.
\end{equation}
In combination with (\ref{first_order_incl_w_N}) this yields
\begin{equation}\label{uniform_convergence_p_i} 
\max_{1\leq i\leq N}\left|\frac{\pi_{i,N}}{\theta w_{\theta}(X_{i})/Ew(X_{1})}-1\right|\overset{P(as)}{\rightarrow}0.
\end{equation}
Using this result it is easily seen that
\begin{equation}\label{limit_varianza_canonical_Poisson}
\begin{split}
\frac{1}{N}\sum_{i=1}^{N}\pi_{i,N}(1-\pi_{i,N})&=\frac{1}{N}\sum_{i=1}^{N}\pi_{i,N}-\frac{1}{N}\sum_{i=1}^{N}\pi_{i,N}^{2}\\
&\overset{P(as)}{\rightarrow}\frac{\theta Ew_{\theta}(X_{1})}{Ew(X_{1})}-\frac{\theta^{2} Ew_{\theta}^{2}(X_{1})}{[Ew(X_{1})]^{2}}.
\end{split}
\end{equation}
Since $0<w_{\theta}(X_{1})\leq Ew(X_{1})/\theta<\infty$, the limiting constant in the last line in the last display must be strictly positive unless $w_{\theta}(X_{1})=Ew(X_{1})/\theta$ with probability $1$. However, in the latter case it would follow that $\theta=\alpha=1$ which contradicts assumption \hyperlink{B0}{B0}. This proves that the limiting constant on the right side in (\ref{limit_varianza_canonical_Poisson}) is positive and hence that assumption \hyperlink{A0}{A0} holds with $\pi_{i,N}$ in place of $p_{i,N}$. From (i) in Result \ref{Hajek_result_1} in Section \ref{rejective_sampling_review} it follows that assumption \hyperlink{A0}{A0} in its original form must be satisfied as well. Actually, the previous argument shows more than that. In fact, it shows that under assumptions \hyperlink{B0}{B0} and \hyperlink{B1}{B1} 
\begin{equation}\label{positive_limit_d_N}
\frac{1}{N}\sum_{i=1}^{N}p_{i,N}(1-p_{i,N})\overset{P(as)}{\rightarrow}C\quad\text{ for some }C>0.
\end{equation}

Next, consider assumption \hyperlink{A1}{A1}. Using (\ref{uniform_convergence_p_i}), (\ref{positive_limit_d_N}) and assumption \hyperlink{B2}{B2} it is not difficult to show that assumption \hyperlink{A1}{A1} is also satisfied with the limiting covariance function $\Sigma'$ defined as in the statement of the lemma (the details of the proof are left to the reader).

Finally, it remains to show that the Lindeberg condition in assumption \hyperlink{A2}{A2} holds as well. Also this can be easily shown by using (\ref{uniform_convergence_p_i}) and assumption \hyperlink{B2}{B2} (the details are left to the reader).
\end{proof}

Having established sufficient conditions for conditional convergence of the marginal distributions it remains to deal with AEC and total boundedness. The next lemma deals with both issues. As already mentioned above, the underlying semimetric will be different from the $L_{2}(P_{y})$-semimetric $\rho$ which was used in the previous subsection. In fact, in the present setting it seems more convenient to use the semimetric
\begin{equation*}
\begin{split}
\rho_{w}(f,g)&:=\sqrt{E\left(\frac{f(Y_{1})-g(Y_{1})}{w_{\theta}(X_{1})}\right)^{2}}=\sqrt{P_{y,x}(f-g)^{2}/w_{\theta}^{2}},\quad f,g\in\mathcal{F}
\end{split}
\end{equation*}
in place of $\rho$. Note that $\rho_{w}$ can be viewed as the $L_{2}(P_{y,x})$-semimetric on the function class $\mathcal{F}/w_{\theta}$.

\begin{lem}[Total boundedness and \hyperlink{conditional_AEC}{conditional AEC}]\label{AEC_case_C_B}
Let $\{\mathbf{S}_{N}^{R}\}_{N=1}^{\infty}$ be the sequence of vectors of sample inclusion indicators for a sequence of measurable CPS designs, let $\mathcal{F}$ be a class of measurable functions $f:\mathcal{Y}\mapsto\mathbb{R}$ and let $\{\mathbb{G}_{N}'\}_{N=1}^{\infty}$ be the sequence of \hyperlink{HTEP}{HTEP}s corresponding to $\{\mathbf{S}_{N}^{R}\}_{N=1}^{\infty}$ and $\mathcal{F}$. Assume that the first order sample inclusion probabilities corresponding to each vector $\mathbf{S}_{N}^{R}$ are defined as in (\ref{first_order_Poisson_special_case}) and that conditions \hyperlink{B0}{B0} and \hyperlink{B1}{B1} hold. Moreover, assume that conditions
\begin{itemize}
\item[\hypertarget{GC_stella}{\textbf{GC$^{*}$}}) ] $(\mathcal{F}_{\infty}/w_{\theta})^{2}:=\{(f-g)^{2}/w_{\theta}^{2}:f,g\in\mathcal{F}\}$ is an outer almost sure $P_{y,x}$-Glivenko-Cantelli class;
\item[\hypertarget{F1_stella}{\textbf{F1$^{*}$}}) ] $\mathcal{F}$ has an envelope function $F$ such that $E[F^{*}(Y_{1})/w_{\theta}(X_{1})]^{2}<\infty$ and such that the uniform entropy condition
\begin{equation}\label{uniform_entropy_case_B}
\int_{0}^{\infty}\sup_{Q_{y,x}}\sqrt{\log N(\epsilon\lVert F/w_{\theta}\rVert_{L_{2}(Q_{y,x})}, \mathcal{F}/w_{\theta}, L_{2}(Q_{y,x}))}d\epsilon<\infty
\end{equation}
holds, where the supremum is taken over all finitely discrete probability measures $Q_{y,x}$ on $\mathcal{Y}\times\mathcal{X}$ such that\footnote{Note the abuse of notation: $\lVert F/w_{\theta}\rVert_{L_{2}(Q_{y,x})}$ and $\int [F(y)/w_{\theta}(x)]^{2}dQ_{y,x}(y,x)$ should actually be written as $\lVert (F\circ\phi_{y})/(w_{\theta}\circ\phi_{x)}\rVert_{L_{2}(Q_{y,x})}$ and $\int [F\circ\phi_{y}(y,x)/w_{\theta}\circ\phi_{x}(y,x)]^{2}dQ_{y,x}(y,x)$, respectively, with $\phi_{y}:\mathcal{Y}\times\mathcal{X}\mapsto\mathbb{R}$ and $\phi_{x}:\mathcal{Y}\times\mathcal{X}\mapsto\mathbb{R}$ defined as $\phi_{y}(y,x):=y$ and $\phi_{x}(y,x):=x$ for $(y,x)\in\mathcal{Y}\times\mathcal{X}$.\label{nota_fondo_pagina}} 
\[\lVert F/w_{\theta}\rVert_{L_{2}(Q_{y,x})}:=\int [F(y)/w_{\theta}(x)]^{2}dQ_{y,x}(y,x)>0;\]
\end{itemize}
hold. Then it follows that
\begin{itemize}
\item[(i) ] $\mathcal{F}$ is totally bounded w.r.t. $\rho_{w}$;
\item[(ii) ] 
\[E_{d}\lVert \mathbb{G}_{N}'\rVert_{\mathcal{F}_{\delta_{N}}^{w}}\overset{P*(as*)}{\rightarrow}0\quad\text{ for every }\delta_{N}\downarrow 0,\]
where $\mathcal{F}_{\delta}^{w}:=\{f-g:f,g\in \mathcal{F}\wedge \rho_{w}(f,g)<\delta\}$.
\end{itemize}
\end{lem}

\begin{proof}
Part (i) of the conclusion follows from condition \hyperlink{F1_stella}{F1$^{*}$} (see Problem 2.5.1 on page 133 in \cite{vdVW}).

The proof of part (ii) of the conclusion is almost the same as the proof of Lemma \ref{AEC_case_C}. The first step is to show that for arbitrary $\delta_{N}\downarrow 0$ the corresponding stochastic processes $\{\mathbb{G}_{N}'f:f\in\mathcal{F}_{\delta_{N}}^{w}\}$ are, with probability tending to $1$ (eventually almost surely), conditionally subgaussian w.r.t. to the empirical semimetric 
\begin{equation*}
\begin{split}
\rho_{N,w}(f,g)&:=\sqrt{\frac{1}{N}\sum_{i=1}^{N}\left[\frac{f(Y_{i})}{w_{\theta}(X_{i})}-\frac{g(Y_{i})}{w_{\theta}(X_{i})}\right]^{2}}\\
&:=\sqrt{\mathbb{P}_{y,x,N}((f/w_{\theta})-(g/w_{\theta}))^{2}},\quad\quad f,g\in\mathcal{F}_{\delta_{N}}^{w},
\end{split}
\end{equation*}
i.e. to show that there exists a constant $C>0$ (which does not depend on the sample points $\omega\in\Omega_{y,x}^{\infty}\times\Omega_{d}$ and neither on $N$) such that, with probability tending to $1$ (eventually almost surely),
\begin{equation}\label{subgaussian_inequality_case_B}
P_{d}\left\{|\mathbb{G}_{N}'f-\mathbb{G}_{N}'g|>x\right\}\leq 2 e^{-Cx^{2}/\rho_{N,w}^{2}(f,g)}\quad\text{for every }f,g\in\mathcal{F}_{\delta_{N}}^{w}
\end{equation}
(cfr. display (\ref{subgaussian_inequality})). To this aim, note that the difference $\mathbb{G}_{N}'f-\mathbb{G}_{N}'g$ can be written as
\begin{equation*}
\begin{split}
\mathbb{G}_{N}^{R}f-\mathbb{G}_{N}^{R}g&=\frac{w_{\theta}(X_{i})\sum_{j=1}^{N}w(X_{j})}{w_{N}(X_{i}) c_{N}\sqrt{N}}(S_{i,N}^{R}-\pi_{i,N})\left[\frac{f(Y_{i})}{w_{\theta}(X_{i})}-\frac{g(Y_{i})}{w_{\theta}(X_{i})}\right]\\
&:=\sum_{i=1}^{N}Z_{i,N},
\end{split}
\end{equation*}
where $w_{N}(X_{i})$ is defined as in (\ref{definizione_w_N}).
Then, note that $E_{d}Z_{i,N}=0$ for $i=1,2,\dots, N$, 
\[|\sqrt{N}Z_{i,N}|\leq\left(\max_{1\leq j\leq N}\frac{w_{\theta}(X_{j})}{w_{N}(X_{j})}\right)\frac{\sum_{j=1}^{N}w(X_{j})}{c_{N}}\left|\frac{f(Y_{i})}{w_{\theta}(X_{i})}-\frac{g(Y_{i})}{w_{\theta}(X_{i})}\right|,\]
and that with probability tending to $1$ (eventually almost surely) the right side in the last inequality is bounded by
\[\frac{1}{L_{\theta}}\left|\frac{f(Y_{i})}{w_{\theta}(X_{i})}-\frac{g(Y_{i})}{w_{\theta}(X_{i})}\right|,\]
where $L_{\theta}$ is a positive constant which depends only on $\theta$ (use (\ref{uniform_convergence_w_N_i}) along with the fact that assumptions \hyperlink{B0}{B0} and \hyperlink{B1}{B1} imply $c_{N}/N\overset{P(as)}{\rightarrow}\theta$). Thus, it follows by Hoeffding's lemma (see \cite{Hoeffding_1963}) that, with probability tending to $1$ (eventually almost surely),
\[E_{d}e^{\lambda Z_{i,N}(f,g)}\leq \exp\left\{\frac{\lambda^{2}}{8N L_{\theta}^{2}}\left[\frac{f(Y_{i})}{w_{\theta}(X_{i})}-\frac{g(Y_{i})}{w_{\theta}(X_{i})}\right]^{2}\right\}.\]
Now, as in the proof of Lemma \ref{AEC_case_C}, use the fact that the components of $\mathbf{S}_{N}^{R}$ are negatively associated to conclude that
\[P_{d}\left\{|\mathbb{G}_{N}'f-\mathbb{G}_{N}'g|>x\right\}\leq 2\exp\left\{-\lambda x+\frac{\lambda^{2}}{8 L_{\theta}^{2}}\rho_{N,w}^{2}(f,g)\right\}\]
with probability tending to $1$ (eventually almost surely). Optimizing the right side w.r.t. $\lambda$ yields then the subgaussian tail inequality in (\ref{subgaussian_inequality_case_B}) with $C=2L_{\theta}^{2}$. 

Next, note that Corollary 2.2.8 on page 101 in \cite{vdVW} can be applied also in the present case and conclude that, with probability tending to one (eventually almost surely),
\[E_{d} \left\lVert\mathbb{G}_{N}'\right\rVert_{\mathcal{F}_{\delta_{N}}^{w}}\leq K\int_{0}^{\infty}\sqrt{\log D(\epsilon,\mathcal{F}_{\delta_{N}}^{w}, \rho_{N,w})}d\epsilon\]
for some constant $K$ (which does not depend on the sample points $\omega\in\Omega_{y,x}^{\infty}\times\Omega_{d}$ and neither on $N$). Now note that
\begin{equation*}
\begin{split}
D(\epsilon,\mathcal{F}_{\delta_{N}}^{w}, \rho_{N,w})&\leq N(\epsilon/2,\mathcal{F}_{\delta_{N}}^{w}, \rho_{N,w})=N(\epsilon/2,\mathcal{F}_{\delta_{N}}^{w}/w_{\theta}, L_{2}(\mathbb{P}_{y,x, N}))
\end{split}
\end{equation*}
where $\mathcal{F}_{\delta_{N}}^{w}/w_{\theta}:=\{f/w_{\theta}:f\in\mathcal{F}_{\delta_{N}}^{w}\}$, so that the integral in the second last display is bounded by a constant multiple of
\[\int_{0}^{\infty}\sqrt{\log N(\epsilon,\mathcal{F}_{\delta_{N}}^{w}/w_{\theta}, L_{2}(\mathbb{P}_{y,x, N}))}d\epsilon.\]
The proof can now be completed by using assumptions \hyperlink{GC_stella}{GC$^{*}$} and \hyperlink{F1_stella}{F1$^{*}$} in order to show that this last integral goes to zero outer almost surely. Again, this can be done by the method used in the proof of Theorem 2.5.2 on page 127 in \cite{vdVW} (see the lines following display 2.5.3 on page 128 in \cite{vdVW}; see also Remark \ref{rem_ass_M_w} in order to see that assumption \hyperlink{GC_stella}{GC$^{*}$} can be replaced by a measurability condition).
\end{proof}


\begin{rem}\label{rem_ass_M_w}
Assume that $w_{\theta}:\mathcal{X}\mapsto(0,\infty)$ is any measurable and uniformly bounded function. Then condition \hyperlink{F1_stella}{F1$^{*}$} together with condition
\begin{itemize}
\item[\hypertarget{M1_primo}{\textbf{M1'}}) ] $(\mathcal{F}_{\infty}/w_{\theta})^{2}:=\{(f-g)^{2}/w_{\theta}^{2}:f,g\in\mathcal{F}\}$ is a $P_{y}$-measurable class of functions (see Definition 2.3.3 on page 110 in \cite{vdVW}), i.e. the function
\[(\mathbf{Y}_{N},\mathbf{X}_{N})\mapsto \sup_{f,g\in\mathcal{F}} \left|\sum_{i=1}^{N}e_{i} \frac{[f(Y_{i})-g(Y_{i})]^{2}}{w_{\theta}^{2}(X_{i})}\right|\]
is measurable on the completion of $((\mathcal{Y}\times\mathcal{X})^{N}, (\mathcal{A}\times\mathcal{B})^{N}, P_{y,x}^{N})$ for every $N$ and for every $(e_{1}, e_{2},\dots, e_{N})\in\mathbb{R}^{N}$
\end{itemize}
imply condition \hyperlink{GC_stella}{GC$^{*}$} (cf. Remark \ref{rem_ass_M}). Of course, condition \hyperlink{PM}{PM} implies also condition \hyperlink{M1_primo}{M1'}.

\end{rem}

\begin{rem}\label{equivalent_uniform_entropy}
Assume that $w_{\theta}:\mathcal{X}\mapsto(0,\infty)$ is any measurable and uniformly bounded function. Then, the uniform entropy conditions (\ref{uniform_entropy}) and (\ref{uniform_entropy_case_B}) are equivalent. To prove this claim, define the projections $\phi_{y}$ and $\phi_{x}$ as in footnote \ref{nota_fondo_pagina}, define $\mathcal{F}\circ\phi_{y}:=\{f\circ\phi_{y}:f\in\mathcal{F}\}$ and let $Q_{y}:=Q_{y,x}\circ\phi_{y}^{-1}$ for any probability measure $Q_{y,x}$ on $(\mathcal{Y},\mathcal{A})\times(\mathcal{X},\mathcal{B})$. Then note that $\lVert f\rVert_{L_{2}(Q_{y})}=\lVert f\circ\phi_{y}\rVert_{L_{2}(Q_{y,x})}$ for every measurable $f:\mathcal{Y}\mapsto\mathbb{R}$ and deduce that
\begin{equation}\label{equal_sup_1}
\sup_{Q_{y}}N(\epsilon\lVert F\rVert_{L_{2}(Q_{y})}, \mathcal{F}, L_{2}(Q_{y}))=\sup_{Q_{y,x}}N(\epsilon\lVert F\circ\phi_{y}\rVert_{L_{2}(Q_{y,x})}, \mathcal{F}\circ\phi_{y}, L_{2}(Q_{y,x})),
\end{equation}
where the supremum on the left side ranges over the set of all finitely discrete probability measures on $\mathcal{Y}$ such that $\lVert F\rVert_{L_{2}(Q_{y})}>0$, and where the supremum on right side ranges over the set of all finitely discrete probability measures on $\mathcal{Y}\times\mathcal{X}$ such that $\lVert F\circ\phi_{y}\rVert_{L_{2}(Q_{y,x})}>0$. 

Next, define for each finitely discrete probability measure $Q_{y,x}$ on $\mathcal{Y}\times\mathcal{X}$ a corresponding finitely discrete measure $R_{y,x}$ by setting
\[\frac{dR_{y,x}}{dQ_{y,x}}(y,x):=w_{\theta}(x)=w_{\theta}\circ\phi_{x}(y,x),\quad (y,x)\in\mathcal{Y}\times\mathcal{X}.\]
Since this density is strictly positive, it follows that the supports of $Q_{y,x}$ and $R_{y,x}$ must be the same. Moreover, it follows that 
\[\frac{dQ_{y,x}}{dR_{y,x}}(y,x)=\frac{1}{w_{\theta}\circ\phi_{x}(y,x)},\quad (y,x)\in\mathcal{Y}\times\mathcal{X}\]
which shows that the mapping $Q_{y,x}\mapsto R_{y,x}$ is a bijection between the set of all finitely discrete probability measures on $\mathcal{Y}\times \mathcal{X}$ and the set of all finitely discrete measures on $\mathcal{Y}\times \mathcal{X}$. Obviously, this bijection satisfies $Q_{y,x}(f\circ\phi_{y})=R_{y,x}[(f\circ\phi_{y})/(w_{\theta}\circ\phi_{x})]$ for every $f:\mathcal{Y}\mapsto\mathbb{R}$. Conclude that
\begin{equation}\label{equal_sup_2}
\begin{split}
&\sup_{Q_{y,x}} N(\epsilon\lVert F\circ\phi_{y}\rVert_{L_{2}(Q_{y,x})}, \mathcal{F}\circ\phi_{y}, L_{2}(Q_{y,x}))=\\
&=\sup_{R_{y,x}} N(\epsilon\lVert (F\circ\phi_{y})/(w_{\theta}\circ\phi_{x})\rVert_{L_{2}(R_{y,x})}, \mathcal{F}/w_{\theta}, L_{2}(R_{y,x}))\\
\end{split}
\end{equation}
where the supremum over $Q_{y,x}$ ranges over the set of all finitely discrete probability measures on $\mathcal{Y}\times \mathcal{X}$ such that $\lVert F\circ\phi_{y}\rVert_{L_{2}(Q_{y,x})}>0$, and where the supremum over $R_{y,x}$ ranges over the set of all finitely discrete measures on $\mathcal{Y}\times \mathcal{X}$ such that $\lVert (F\circ\phi_{y})/(w_{\theta}\circ\phi_{x})\rVert_{L_{2}(R_{y,x})}>0$. Next, note that
\begin{equation}\label{equal_sup_3}
\begin{split}
&\sup_{R_{y,x}} N(\epsilon\lVert (F\circ\phi_{y})/(w_{\theta}\circ\phi_{x})\rVert_{L_{2}(R_{y,x})}, \mathcal{F}/w_{\theta}, L_{2}(R_{y,x}))\\
&=\sup_{Q_{y,x}} N(\epsilon\lVert (F\circ\phi_{y})/(w_{\theta}\circ\phi_{x})\rVert_{L_{2}(Q_{y,x})}, \mathcal{F}/w_{\theta}, L_{2}(Q_{y,x}))
\end{split}
\end{equation}
where the supremum over $Q_{y,x}$ ranges over the set of all finitely discrete measures on $\mathcal{Y}\times \mathcal{X}$ such that $\lVert (F\circ\phi_{y})/(w_{\theta}\circ\phi_{x})\rVert_{L_{2}(Q_{y,x})}>0$ (see Problem 2.10.5 on page 204 in \citep{vdVW}). Now combine equations (\ref{equal_sup_1}), (\ref{equal_sup_2}) and (\ref{equal_sup_3}) to obtain
\[\sup_{Q_{y}}N(\epsilon\lVert F\rVert_{L_{2}(Q_{y})}, \mathcal{F}, L_{2}(Q_{y}))=\sup_{Q_{y,x}} N(\epsilon\lVert (F\circ\phi_{y})/(w_{\theta}\circ\phi_{x})\rVert_{L_{2}(Q_{y,x})}, \mathcal{F}/w_{\theta}, L_{2}(Q_{y,x}))\]
where, by an abuse of notation, the right side can be written as the integrand on the left side in (\ref{uniform_entropy_case_B}). This shows that the uniform entropy integrals in  (\ref{uniform_entropy}) and (\ref{uniform_entropy_case_B}) are actually the same.
\end{rem}

\begin{rem}\label{rem_condizione_F2}
Assume that $w_{\theta}:\mathcal{X}\mapsto(0,\infty)$ is any measurable and uniformly bounded function. Then condition \hyperlink{F1_stella}{F1$^{*}$} does obviously imply condition \hyperlink{B2}{B2}. Moreover, from Remark \ref{equivalent_uniform_entropy} it follows that condition \hyperlink{F1_stella}{F1$^{*}$} implies also condition \hyperlink{F1}{F1}. Since condition \hyperlink{GC_stella}{GC$^{*}$} is stronger than condition \hyperlink{GC}{GC}, it follows further that conditions \hyperlink{F1_stella}{F1$^{*}$}, \hyperlink{GC_stella}{GC$^{*}$} and \hyperlink{M2}{M2} imply that $\mathcal{F}$ is a $P_{y}$-Donsker class (cf. Remark \ref{rem_ass_M}).
\end{rem}


\begin{thm}[conditional weak convergence]\label{opCWC_case_B_C}
Let $\{\mathbf{S}_{N}^{R}\}_{N=1}^{\infty}$ be the sequence of vectors of sample inclusion indicators corresponding to a sequence of measurable Poisson sampling designs, let $\mathcal{F}$ be a class of measurable functions $f:\mathcal{Y}\mapsto\mathbb{R}$ and let $\{\mathbb{G}_{N}'\}_{N=1}^{\infty}$ be the sequence of \hyperlink{HTEP}{HTEP}s corresponding to $\{\mathbf{S}_{N}^{R}\}_{N=1}^{\infty}$ and $\mathcal{F}$. Assume that the first order sample inclusion probabilities corresponding to each vector $\mathbf{S}_{N}^{R}$ are defined as in (\ref{first_order_Poisson_special_case}) and assume that conditions \hyperlink{B0}{B0}, \hyperlink{B1}{B1}, \hyperlink{F1_stella}{F1$^{*}$} and \hyperlink{GC_stella}{GC$^{*}$} are satisfied. Then it follows that
\begin{itemize}
\item[(i) ] there exists zero-mean Gaussian process $\{\mathbb{G}'f:f\in\mathcal{F}\}$ with covariance function given by $\Sigma'$ as defined in (\ref{covariance_function_case_B}) (or in assumption \hyperlink{A1}{A1}) which is a Borel measurable and tight mapping from some probability space into $l^{\infty}(\mathcal{F})$ such that
\[\sup_{h\in BL_{1}(l^{\infty}(\mathcal{F}))}\left|E_{d}h(\mathbb{G}_{N}')-Eh(\mathbb{G}')\right|\overset{P*(as*)}{\rightarrow}0,\]
\item[(ii) ] the sample paths $f\mapsto\mathbb{G}'f$ are uniformly continuous w.r.t. the semimetric $\rho_{w}(f,g):=[P_{y}(f-g)^{2}/w_{\theta}^{2}]^{1/2}$ with probability $1$.
\end{itemize}
\end{thm}

\begin{proof}
Assumptions \hyperlink{B0}{B0} and \hyperlink{B1}{B1} imply that the function $w_{\theta}$ is well defined and that it is measurable and uniformly bounded, and together with assumption \hyperlink{F1_stella}{F1$^{*}$} they imply also condition \hyperlink{B2}{B2}. From Lemma \ref{marginal_convergence_case_C_B} and Lemma \ref{lem_marginal_convergence_Poisson_process_C} it follows therefore that $\{\mathbb{G}_{N}'\}_{N=1}^{\infty}$ satisfies \hyperlink{CWCM}{CWCM} for some zero-mean Gaussian limit process with covariance function given by $\Sigma'$ as defined in (\ref{covariance_function_case_B}) (or in assumption \hyperlink{A1}{A1}). Moreover, Lemma \ref{AEC_case_C_B} shows that $\mathcal{F}$ is totally bounded w.r.t. $\rho_{w}$ and that $\{\mathbb{G}_{N}'\}_{N=1}^{\infty}$ is \hyperlink{conditional_AEC}{conditionally AEC} w.r.t. $\rho_{w}$. The two conclusions of the theorem follow now by Corollary \ref{cor_uniformly_continuous_limit_process}.
\end{proof}

\begin{thm}[Unconditional weak convergence]\label{unconditional_convergence_case_B_C}
Let $\{\mathbf{S}_{N}^{R}\}_{N=1}^{\infty}$, $\mathcal{F}$ and $\{\mathbb{G}_{N}'\}_{N=1}^{\infty}$ be defined as in Theorem \ref{opCWC_case_B_C}. Assume that the first order sample inclusion probabilities corresponding to each vector $\mathbf{S}_{N}^{R}$ are defined as in (\ref{first_order_Poisson_special_case}) and assume that conditions \hyperlink{B0}{B0}, \hyperlink{B1}{B1}, \hyperlink{F1_stella}{F1$^{*}$} and \hyperlink{PM}{PM} are satisfied. Then it follows that 
\begin{itemize}
\item[(i) ] there exists zero-mean Gaussian process $\{\mathbb{G}'f:f\in\mathcal{F}\}$ with covariance function given by $\Sigma'$ as defined in (\ref{covariance_function_case_B}) (or in assumption \hyperlink{A1}{A1}) which is a Borel measurable and tight random element of $l^{\infty}(\mathcal{F})$ such that
\[\mathbb{G}_{N}'\rightsquigarrow\mathbb{G}'\quad\text{ in }l^{\infty}(\mathcal{F});\]
\item[(ii) ] the sample paths $f\mapsto\mathbb{G}'f$ are uniformly continuous w.r.t. the semimetric $\rho_{w}(f,g):=[P_{y}(f/w_{\theta}-g/w_{\theta})^{2}]^{1/2}$ with probability $1$.
\end{itemize}
\end{thm}

\begin{proof}
Remark \ref{rem_ass_M_w} shows that conditions \hyperlink{F1_stella}{F1$^{*}$} and \hyperlink{PM}{PM} imply assumption \hyperlink{GC_stella}{GC$^{*}$}. The conditions of the present theorem are therefore stronger than those of Theorem \ref{opCWC_case_B_C}, and the conclusions of the present theorem follows therefore from Theorem \ref{opCWC_case_B_C} and Remark \ref{rem_measurable_suprema} (note that assumption \hyperlink{PM}{PM} implies that the suprema in Remark \ref{rem_measurable_suprema} are measurable).
\end{proof}

\begin{cor}[Joint weak convergence]\label{corollary_joint_weak_convergence_C_B}
Under the assumptions of Theorem \ref{unconditional_convergence_case_B_C} it follows that
\[(\mathbb{G}_{N}, \mathbb{G}_{N}')\rightsquigarrow(\mathbb{G}, \mathbb{G}')\text{ in }l^{\infty}(\mathcal{F})\times l^{\infty}(\mathcal{F}),\]
where $\mathbb{G}'$ is defined as in Theorem \ref{unconditional_convergence_case_B_C}, $\mathbb{G}_{N}$ is the classical $\mathcal{F}$-indexed empirical process defined in (\ref{def_classical_empirical_process}), and where $\mathbb{G}$ is a Borel measurable and tight $P_{y}$-Brownian Bridge which is independent from $\mathbb{G}'$.
\end{cor}

\begin{proof}
In the proof of Theorem \ref{unconditional_convergence_case_B_C} it has already been shown that the assumptions of Theorem \ref{unconditional_convergence_case_B_C} are stronger than those of Theorem \ref{opCWC_case_B_C} which imply \hyperlink{opCWC}{opCWC}. Moreover, from Remark \ref{rem_ass_M_w}, Remark \ref{rem_condizione_F2} and Remark \ref{condition_PM} it follows that $\mathcal{F}$ is a $P_{y}$-Donsker class. The proof of the corollary follows now from an application of Theorem \ref{teorema_generale_joint_weak_convergence}. 
\end{proof}

\section{Extensions for H{\'a}jek empirical processes}\label{Hajek_empirical_process_theory_POISSON}

This section is very similar to Section 4 in \citep{Pasquazzi_2019}. It extends the weak convergence results for \hyperlink{HTEP}{HTEP} sequences to the corresponding H{\'a}jek empirical processes (henceforth \hypertarget{HEP}{HEP}). Given a class $\mathcal{F}$ of functions $f:\mathcal{Y}\mapsto\mathbb{R}$, the \hyperlink{HEP}{HEP} is defined as
\begin{equation}\label{Def_HEP}
\mathbb{G}_{N}''f:=\sqrt{N}\left(\frac{1}{\widehat{N}}\sum_{i=1}^{N}\frac{S_{i,N}}{\pi_{i,N}}f(Y_{i})-\frac{1}{N}\sum_{i=1}^{N}f(Y_{i})\right),\quad f\in\mathcal{F},
\end{equation}
with $\widehat{N}:=\sum_{i=1}^{N}(S_{i,N}/\pi_{i,N})$ the Horvitz-Thompson estimator of the population size $N$. Note that the value taken on by $\mathbb{G}_{N}''f$ is undefined when $\widehat{N}=0$. However, this will not be problem here since the assumptions in the forthcoming theory will always imply that
\begin{equation}\label{condizione_consistenza_N_hat}
P_{d}\left\{\left|\frac{\widehat{N}}{N}-1\right|>\epsilon\right\}\overset{P*(as*)}{\rightarrow} 0\quad\text{ for every }\epsilon>0.
\end{equation}
In fact, this condition allows to consider in place of the \hyperlink{HEP}{HEP} as defined in (\ref{Def_HEP}) the closely related empirical process given by
\begin{equation}\label{HEP_approximation}
\widetilde{\mathbb{G}}_{N}''f:=\frac{1}{\sqrt{N}}\sum_{i=1}^{N}\left(\frac{S_{i,N}}{\pi_{i,N}}-1\right)[f(Y_{i})-\mathbb{P}_{y,N}f],\quad f\in\mathcal{F},
\end{equation}
where $\mathbb{P}_{y,N}:=\sum_{i=1}^{N}\delta_{Y_{i}}/N$ is the empirical measure on $\mathcal{Y}$. In order to see why under condition (\ref{condizione_consistenza_N_hat}) we can consider $\widetilde{\mathbb{G}}_{N}''$ in place of the \hyperlink{HEP}{HEP} it is sufficient to observe that
\begin{equation}\label{HEP_equvalenza_asintotica}
\mathbb{G}_{N}''f-\widetilde{\mathbb{G}}_{N}''f=\left(\frac{N}{\widehat{N}}-1\right) \widetilde{\mathbb{G}}_{N}''f,\quad f\in\mathcal{F},
\end{equation}
and that this together with condition (\ref{condizione_consistenza_N_hat}) implies that any one of the three weak convergence results in $l^{\infty}(\mathcal{F})$ for the sequence $\{\widetilde{\mathbb{G}}_{N}''\}_{N=1}^{\infty}$ carries over immediately to the corresponding sequence of \hyperlink{HEP}{HEP}s, and viceversa.

The following lemma establishes conditional convergence of the marginal distributions for the sequence $\{\widetilde{\mathbb{G}}_{N}''\}_{N=1}^{\infty}$ and hence for the corresponding sequence of \hyperlink{HEP}{HEP}s as well. 

\begin{lem}[\hyperlink{CWCM}{CWCM}]\label{marginal_convergence_HAJEK}
Let $\{\mathbf{S}_{N}^{R}\}_{N=1}^{\infty}$, $\{\mathbf{p}_{N}\}_{N=1}^{\infty}$ and $\mathcal{F}$  be defined as in Lemma \ref{lem_marginal_convergence_Poisson_process_C}, let $\{\mathbf{\underline{\pi}}_{N}\}_{N=1}^{\infty}$ be the sequence of vectors of first order sample inclusion probabilities corresponding to $\{\mathbf{S}_{N}^{R}\}_{N=1}^{\infty}$, and let $\{\widetilde{\mathbb{G}}''\}_{N=1}^{\infty}$ be the sequence of empirical processes defined by (\ref{HEP_approximation}). Assume that conditions
\begin{itemize}
\item[\hypertarget{C1}{\textbf{C1}}) ] $\mathcal{F}$ contains a constant function which is not identically equal to zero, i.e. a function $f:\mathcal{Y}\mapsto\mathbb{R}$ such that $f\equiv C$ $P_{y}$-almost surely for some constant $C\neq 0$;
\item[\hypertarget{C2}{\textbf{C2}}) ] $P_{y}|f|<\infty$ for every $f\in\mathcal{F}$
\end{itemize}
and conditions \hyperlink{A0}{A0}, \hyperlink{A1}{A1} and \hyperlink{A2}{A2} are satisfied. Then the function
\begin{equation}\label{funzione_covarianza_Sigma_primo_primo}
\Sigma''(f,g):=\Sigma'(f,g)-\frac{P_{y}f}{C}\Sigma'(C,g)-\frac{P_{y}g}{C}\Sigma'(f, C)+\frac{P_{y}f P_{y}g}{C^{2}}\Sigma'(C, C),\quad f,g\in\mathcal{F},
\end{equation}
with $\Sigma'$ defined as in assumption \hyperlink{A1}{A1}, is a positive semidefinite covariance function, and for every finite-dimensional $\mathbf{f}\in\mathcal{F}^{r}$ and for every $\mathbf{t}\in\mathbb{R}^{r}$
\begin{equation*}\label{limit_chf_Hajek}
E_{d}\exp(i\mathbf{t}^{\intercal}\widetilde{\mathbb{G}}''_{N}\mathbf{f})\overset{P(as)}{\rightarrow} \exp\left(-\frac{1}{2}\mathbf{t}^{\intercal}\Sigma''(\mathbf{f})\mathbf{t}\right),
\end{equation*}
where $\Sigma''(\mathbf{f})$ is the covariance matrix whose elements are given by $\Sigma''_{(ij)}(\mathbf{f}):=\Sigma''(f_{i},f_{j})$. 
\end{lem}

\begin{proof}
The proof is almost the same as the proof of Lemma \ref{lem_marginal_convergence_Poisson_process_C}. Define the sequences of sample inclusion indicators $\{\mathbf{S}_{N}^{P}\}_{N=1}^{\infty}$ and $\{\mathbf{S}_{N}^{P_{0}}\}_{N=1}^{\infty}$ as in the proof of Lemma \ref{lem_marginal_convergence_Poisson_process_C}. Then, define the sequence of stochastic processes $\{\mathbb{\widetilde{T}}_{N}^{P}\}_{N=1}^{\infty}:=\{\mathbb{\widetilde{T}}_{N}^{P}f:f\in\mathcal{F}\}\}_{N=1}^{\infty}$ by 
\begin{equation*}
\begin{split}
\mathbb{\widetilde{T}}_{N}^{P}f&:=\sqrt{N}T_{N}(f-\mathbb{P}_{y,N}f;\mathbf{S}_{N}^{P})\\
&=\frac{1}{\sqrt{N}}\sum_{i=1}^{N}\left(\frac{S_{i,N}^{P}}{p_{i,N}}-1\right)\left[(f(Y_{i})-\mathbb{P}_{y,N}f)-R_{N}(f-\mathbb{P}_{y,N}f)p_{i,N}\right]\\
&=\frac{1}{\sqrt{N}}\sum_{i=1}^{N}\left(\frac{S_{i,N}^{P}}{p_{i,N}}-1\right)(f(Y_{i})-\mathbb{P}_{y,N}f)+\\
&\quad\quad\quad-\frac{R_{N}(f-\mathbb{P}_{y,N}f)}{\sqrt{N}}\sum_{i=1}^{N}(S_{i,N}^{P}-p_{i,N})\\
&:=\mathbb{\widetilde{Y}}_{N}^{P}f-\mathbb{\widetilde{R}}_{N}^{P}f.
\end{split}
\end{equation*}
Note that $E_{d}\mathbb{\widetilde{T}}_{N}^{P}f=0$ for every $f\in\mathcal{F}$, and that 
\begin{equation*}
\begin{split}
\Sigma_{N}''(f,g)&:=E_{d}\mathbb{\widetilde{T}}_{N}^{P}f\mathbb{\widetilde{T}}_{N}^{P}g\\
&=\Sigma_{N}'(f,g)-\Sigma_{N}'(\mathbb{P}_{y,N}f,g)-\Sigma_{N}'(f,\mathbb{P}_{y,N}g)+\Sigma_{N}'(\mathbb{P}_{y,N}f,\mathbb{P}_{y,N}g)
\end{split}
\end{equation*}
where $\Sigma_{N}'(f,g):=E_{d}\mathbb{T}_{N}^{P}f\mathbb{T}_{N}^{P}g$ is defined as in the proof of Lemma \ref{lem_marginal_convergence_Poisson_process_C}. Now, it follows from assumptions \hyperlink{C1}{C1}, \hyperlink{C2}{C2} and \hyperlink{A1}{A1} that
\begin{equation}\label{limit_covariance_function_HAJEK}
\Sigma_{N}''(f,g)\overset{P(as)}{\rightarrow}\Sigma''(f,g),\quad f,g\in\mathcal{F},
\end{equation}
where $\Sigma''(f,g)$ is defined as in (\ref{funzione_covarianza_Sigma_primo_primo}). This implies that $\Sigma'':\mathcal{F}^{2}\mapsto\mathbb{R}$ must be positive semidefinite and proves the first part of the conclusion of the lemma. 

In order to prove the second part of the conclusion, consider for some given $\mathbf{f}\in\mathcal{F}^{r}$ the triangular array of rowwise conditionally independent random vectors
\[Z_{i,N}\mathbf{f}:=\left(\frac{S_{i,N}^{P}}{p_{i,N}}-1\right)\left[(\mathbf{f}(Y_{i})-\mathbb{P}_{y,N}\mathbf{f})-R_{N}(\mathbf{f}-\mathbb{P}_{y,N}\mathbf{f})p_{i,N}\right],\quad \mathbf{f}\in\mathcal{F}^{r},\]
\[i=1,2,\dots, N,\quad N=1,2,\dots,\]
where $\mathbb{P}_{y,N}\mathbf{f}:=(\mathbb{P}_{y,N}f_{1}, \mathbb{P}_{y,N}f_{2}, \dots, \mathbb{P}_{y,N}f_{r})^{\intercal}$.
Observe that the random vector $\mathbb{\widetilde{T}}_{N}^{P}\mathbf{f}:=(\mathbb{\widetilde{T}}_{N}^{P}f_{1}, \mathbb{\widetilde{T}}_{N}^{P}f_{2}, \dots, \mathbb{\widetilde{T}}_{N}^{P}f_{r})^{\intercal}$ can be written as
\[\mathbb{\widetilde{T}}_{N}^{P}\mathbf{f}=\frac{1}{\sqrt{N}}\sum_{i=1}^{N}Z_{i,N}\mathbf{f}.\]
Using the fact that $\Sigma_{N}''(f,g):=E_{d}\mathbb{\widetilde{T}}_{N}^{P}f\mathbb{\widetilde{T}}_{N}^{P}g\overset{P(as)}{\rightarrow}\Sigma''(f,g)$ along with condition \hyperlink{A2}{A2} it is not difficult to show that the Lindeberg condition
\[\frac{1}{N\mathbf{t}^{\intercal}\Sigma_{N}''(\mathbf{f})\mathbf{t}}\sum_{i=1}^{N}E_{d}\left(\mathbf{t}^{\intercal}Z_{i,N}\mathbf{f}\right)^{2}I\left(|\mathbf{t}^{\intercal}Z_{i,N}\mathbf{f}|>\epsilon\sqrt{\mathbf{t}^{\intercal}\Sigma_{N}''(\mathbf{f})\mathbf{t}}\right)\overset{P(as)}{\rightarrow}0,\quad\epsilon>0,\]
must be satisfied whenever $\mathbf{f}\in\mathcal{F}$ and $\mathbf{t}\in\mathbb{R}^{r}$ such that $\mathbf{t}^{\intercal}\Sigma''(\mathbf{f})\mathbf{t}>0$. Therefore it follows that
\[E_{d}\exp\left(i\mathbf{t}^{\intercal}\mathbb{\widetilde{T}}_{N}^{P}\mathbf{f})\right)\overset{P(as)}{\rightarrow} \exp\left(-\frac{1}{2}\mathbf{t}^{\intercal}\Sigma'(\mathbf{f})\mathbf{t}\right).\]

Next, consider the sequence of stochastic processes $\{\mathbb{\widetilde{T}}_{N}^{P_{0}}\}_{N=1}^{\infty}:=\{\{\mathbb{\widetilde{T}}_{N}^{P_{0}}f: f\in\mathcal{F}\}\}_{N=1}^{\infty}$ with $\mathbb{\widetilde{T}}_{N}^{P_{0}}f$ defined in the same way as $\mathbb{\widetilde{T}}_{N}^{P}f$ but with $\mathbf{S}_{N}^{P_{0}}$ in place of $\mathbf{S}_{N}^{P}$. Use assumption \hyperlink{A0}{A0} along with Result \ref{Hajek_result_2} in Section \ref{rejective_sampling_review} to show that
\[\left|E_{d}\exp\left(i\mathbf{t}^{\intercal}\mathbb{\widetilde{T}}_{N}^{P}\mathbf{f})\right)-E_{d}\exp\left(i\mathbf{t}^{\intercal}\mathbb{\widetilde{T}}_{N}^{P_{0}}\mathbf{f})\right)\right|\overset{P(as)}{\rightarrow}0, \quad \mathbf{f}\in\mathcal{F}^{r}, \mathbf{t}\in\mathbb{R}^{r}.\]
Note that this does not require to know the joint distributions of the vectors $\mathbf{S}_{N}^{P_{0}}$ and $\mathbf{S}_{N}^{P}$. 

Third, consider the sequence of stochastic processes $\{\mathbb{\widetilde{Y}}_{N}^{R}\}_{N=1}^{\infty}:=\{\{\mathbb{\widetilde{Y}}_{N}^{R}f: f\in\mathcal{F}\}\}_{N=1}^{\infty}$ with $\mathbb{\widetilde{Y}}_{N}^{R}f$ defined in the same way as $\mathbb{\widetilde{Y}}_{N}^{P}f$ but with $\mathbf{S}_{N}^{R}$ in place of $\mathbf{S}_{N}^{P}$. Use Result \ref{Hajek_result_3} in Section \ref{rejective_sampling_review} to conclude that
\[\left|E_{d}\exp\left(i\mathbf{t}^{\intercal}\mathbb{\widetilde{T}}_{N}^{P_{0}}\mathbf{f})\right)-E_{d}\exp\left(i\mathbf{t}^{\intercal}\mathbb{\widetilde{Y}}_{N}^{R}\mathbf{f})\right)\right|\overset{P(as)}{\rightarrow}0, \quad \mathbf{f}\in\mathcal{F}^{r}, \mathbf{t}\in\mathbb{R}^{r}\]
as well. 

Finally, note that the definition of $\mathbb{\widetilde{Y}}_{N}^{R}$ coincides with the one of $\mathbb{\widetilde{G}}_{N}''$ except for the fact that the former contains the first order sample inclusion probabilities corresponding to $\mathbf{S}_{N}^{P}$ in place of those corresponding to $\mathbf{S}_{N}^{R}$, i.e. $\mathbb{\widetilde{Y}}_{N}^{R}$ contains $p_{i,N}$ in place of $\pi_{i,N}:=E_{d}S_{i,N}^{R}$. However, this problem can be easily fixed by using Result \ref{Hajek_result_1}.
\end{proof}

\begin{rem}\label{rem_equivalenza}
Assumption \hyperlink{A0}{A0} implies that condition (\ref{condizione_consistenza_N_hat}) holds. By (\ref{HEP_equvalenza_asintotica}) it follows therefore that the conditions of Lemma \ref{marginal_convergence_HAJEK} imply also that
\begin{equation*}\label{limit_chf_Hajek_HEP}
E_{d}\exp(i\mathbf{t}^{\intercal}\mathbb{G}''_{N}\mathbf{f})\overset{P(as)}{\rightarrow} \exp\left(-\frac{1}{2}\mathbf{t}^{\intercal}\Sigma''(\mathbf{f})\mathbf{t}\right).
\end{equation*}
\end{rem}

\begin{rem}\label{condizione_C2}
Assumption \hyperlink{C2}{C2} is certainly satisfied if assumption \hyperlink{F1}{F1} holds or if $w_{\theta}:\mathcal{X}\mapsto(0,\infty)$ is measurable and uniformly bounded and assumption \hyperlink{F1_stella}{F1$^{*}$} holds.
\end{rem}

The next two lemmas establish \hyperlink{conditional_AEC}{conditional AEC} of the $\{\widetilde{\mathbb{G}}''\}_{N=1}^{\infty}$ sequence for the case where there is a positive lower bound for the $\pi_{i,N}$'s and for the case where the $\pi_{i,N}$'s are proportional to some size variable which can take on arbitrarily small values, respectively.

\begin{lem}[\hyperlink{conditional_AEC}{conditional AEC}]\label{HEP_AEC_case_A}
Let $\mathcal{F}$ be a class of functions $f:\mathcal{Y}\mapsto\mathbb{R}$ which satisfies assumption \hyperlink{M2}{M2}. Then, under the assumptions of Lemma \ref{AEC_case_C}, it follows that
\[E_{d}\lVert\mathbb{\widetilde{G}}_{N}''\rVert_{\mathcal{F}_{\delta_{N}}}\overset{P*(as*)}{\rightarrow}0\quad\text{ for every }\delta_{N}\downarrow 0.\]
\end{lem}

\begin{proof}
First, note that
\[\mathbb{\widetilde{G}}_{N}''f=\mathbb{G}'_{N}f -\frac{\mathbb{P}_{y,N}f}{\sqrt{N}}\sum_{i=1}^{N}\left(\frac{S_{i,N}^{R}}{\pi_{i,n}}-1\right),\quad f\in\mathcal{F},\]
where $\mathbb{G}'_{N}$ is the \hyperlink{HTEP}{HTEP}. From this it follows that
\[E_{d}\lVert\mathbb{\widetilde{G}}_{N}''\rVert_{\mathcal{F}_{\delta_{N}}}\leq E_{d}\lVert\mathbb{G}_{N}'\rVert_{\mathcal{F}_{\delta_{N}}}+\lVert\mathbb{P}_{y,N}\rVert_{\mathcal{F}_{\delta_{N}}}E_{d}\left|\frac{1}{\sqrt{N}}\sum_{i=1}^{N}\left(\frac{S_{i,N}}{\pi_{i,n}}-1\right)\right|.\]
Since by Theorem 2.8 in \cite{Joag-Dev_1983} the $S_{i,N}^{R}$'s are negatively associated, it follows that
\begin{equation}\label{variance_sample_size_estimator}
E_{d}\left|\frac{1}{\sqrt{N}}\sum_{i=1}^{N}\left(\frac{S_{i,N}}{\pi_{i,n}}-1\right)\right|^{2}\leq \frac{1}{N}\sum_{i=1}^{N}\frac{1-\pi_{i,N}}{\pi_{i,N}},
\end{equation}
and assumption \hyperlink{A2_stella}{A2$^{*}$} implies that the right side in the latter inequality is bounded in probability (eventually almost surely). To complete the proof of the lemma it remains therefore to show that
\[\lVert\mathbb{P}_{y,N}\rVert_{\mathcal{F}_{\delta_{N}}}\overset{as*}{\rightarrow}0\quad\text{ for every }\delta_{N}\downarrow 0.\]
To this aim note that
\[\lVert\mathbb{P}_{y,N}\rVert_{\mathcal{F}_{\delta_{N}}}\leq\lVert\mathbb{P}_{y,N}-P_{y}\rVert_{\mathcal{F}_{\delta_{N}}}+\lVert P_{y}\rVert_{\mathcal{F}_{\delta_{N}}}\leq \lVert\mathbb{P}_{y,N}-P_{y}\rVert_{\mathcal{F}}+\delta_{N},\]
and that $\lVert\mathbb{P}_{y,N}-P_{y}\rVert_{\mathcal{F}}\overset{as*}{\rightarrow}0$ because assumptions \hyperlink{F1}{F1}, \hyperlink{GC}{GC} and \hyperlink{M2}{M2} imply that $\mathcal{F}$ is a $P_{y}$-Donsker class (see Remark \ref{rem_ass_M}) and hence an outer almost sure $P_{y}$-Glivenko-Cantelli class.
\end{proof}

\begin{lem}[\hyperlink{conditional_AEC}{conditional AEC}]\label{HEP_AEC_case_B}
Let $\mathcal{F}$ be a class of functions $f:\mathcal{Y}\mapsto\mathbb{R}$ which satisfies assumptions \hyperlink{C1}{C1} and \hyperlink{M2}{M2}. Then, under the assumptions of Lemma \ref{AEC_case_C_B}, it follows that
\[E_{d}\lVert\mathbb{\widetilde{G}}_{N}''\rVert_{\mathcal{F}_{\delta_{N}}^{w}}\overset{P*(as*)}{\rightarrow}0\quad\text{ for every }\delta_{N}\downarrow 0.\]
\end{lem}

\begin{proof}
Follow the steps in the proof of Lemma \ref{HEP_AEC_case_A} up to inequality (\ref{variance_sample_size_estimator}) (with $\mathcal{F}_{\delta_{N}}^{w}$ in place of $\mathcal{F}_{\delta_{N}}$) and note that the right side of that inequality is bounded in probability (eventually almost surely) because assumptions \hyperlink{B0}{B0} and \hyperlink{B1}{B1} imply (\ref{uniform_convergence_p_i}) (see the proof of Lemma \ref{marginal_convergence_case_C_B}), and assumptions \hyperlink{C1}{C1} and \hyperlink{F1_stella}{F1$^{*}$} imply $Ew_{\theta}^{-2}(X_{1})<\infty$. To complete the proof of the lemma it remains therefore to show that
\[\lVert\mathbb{P}_{y,N}\rVert_{\mathcal{F}_{\delta_{N}}^{w}}\overset{as*}{\rightarrow}0\quad\text{ for every }\delta_{N}\downarrow 0.\]
To this aim note that
\[\lVert\mathbb{P}_{y,N}\rVert_{\mathcal{F}_{\delta_{N}}^{w}}\leq\lVert\mathbb{P}_{y,N}-P_{y}\rVert_{\mathcal{F}_{\delta_{N}}^{w}}+\lVert P_{y}\rVert_{\mathcal{F}_{\delta_{N}}^{w}}\leq \lVert\mathbb{P}_{y,N}-P_{y}\rVert_{\mathcal{F}}+\frac{Ew(X_{1})}{\theta}\delta_{N},\]
and that $\lVert\mathbb{P}_{y,N}-P_{y}\rVert_{\mathcal{F}}\overset{as*}{\rightarrow}0$ because assumptions \hyperlink{F1_stella}{F1$^{*}$}, \hyperlink{GC_stella}{GC$^{*}$} and \hyperlink{M2}{M2} imply that $\mathcal{F}$ is a $P_{y}$-Donsker class (see Remark \ref{rem_condizione_F2}) and hence an outer almost sure $P_{y}$-Glivenko-Cantelli class.
\end{proof}

Having found sufficient conditions for \hyperlink{CWCM}{CWCM} and for \hyperlink{conditional_AEC}{conditional AEC} w.r.t to suitable semimetrics, we are now ready to prove the three desired weak convergence results. Since the sufficient conditions under consideration imply condition (\ref{condizione_consistenza_N_hat}) (see Remark \ref{rem_equivalenza}), the weak convergence results for $\{\widetilde{\mathbb{G}}''\}_{N=1}^{\infty}$ and for the \hyperlink{HEP}{HEP} sequence $\{\mathbb{G}_{N}''\}_{N=1}^{\infty}$ are equivalent. Since only the \hyperlink{HEP}{HEP} sequence is of interest in applications, the weak convergence results will be stated only in terms of the latter.

\begin{thm}[Conditional and unconditional weak convergence]\label{opCWC_HEP}
Let $\mathcal{F}$ be a class of functions $f:\mathcal{Y}\mapsto\mathbb{R}$ which satisfies assumption \hyperlink{C1}{C1} and let $\{\mathbb{G}_{N}''\}_{N=1}^{\infty}$ be the corresponding sequence of \hyperlink{HEP}{HEP}s as defined in (\ref{Def_HEP}). Then, under the assumptions of Theorem \ref{unconditional_convergence} or the assumptions of Theorem \ref {unconditional_convergence_case_B_C} it follows that
\begin{itemize}
\item[(i) ] there exists zero-mean Gaussian process $\{\mathbb{G}''f:f\in\mathcal{F}\}$ with covariance function $\Sigma'':\mathcal{F}^{2}\mapsto\mathbb{R}$ defined as in (\ref{funzione_covarianza_Sigma_primo_primo}) which is a Borel measurable and tight random element of $l^{\infty}(\mathcal{F})$;
\item[(ii) ] (conditional weak convergence)
\[\sup_{h\in BL_{1}(l^{\infty}(\mathcal{F}))}\left|E_{d}h(\mathbb{G}_{N}'')-Eh(\mathbb{G}'')\right|\overset{P*(as*)}{\rightarrow}0;\]
\item[(iii) ] (unconditional weak convergence)
\[\mathbb{G}_{N}''\rightsquigarrow\mathbb{G}''\quad\text{ in }l^{\infty}(\mathcal{F}).\]
\end{itemize}
Moreover,
\begin{itemize}
\item[(iv) ] under the assumptions of Theorem \ref{unconditional_convergence} it follows that the sample paths $f\mapsto\mathbb{G}''f$ are uniformly $\rho$-continuous with probability $1$;
\item[(v) ] under the assumptions of Theorem \ref{unconditional_convergence_case_B_C} it follows that the sample paths $f\mapsto\mathbb{G}''f$ are uniformly $\rho_{w}$-continuous with probability $1$.
\end{itemize}
\end{thm}

\begin{proof}
Consider first the assumptions of Theorem \ref{unconditional_convergence}. Remark \ref{remark_condizione_A2_stella} and Remark \ref{condizione_C2} show that the assumptions of Theorem \ref{unconditional_convergence} together with assumption \hyperlink{C1}{C1} imply the assumptions of Lemma \ref{marginal_convergence_HAJEK}, and the conclusion of that lemma says that the sequence of auxiliary processes $\{\widetilde{\mathbb{G}}''\}_{N=1}^{\infty}$ satisfies \hyperlink{CWCM}{CWCM} for some zero-mean Gaussian limit process with covariance function given by $\Sigma''$ as defined in (\ref{funzione_covarianza_Sigma_primo_primo}). Next, in the proof of Theorem \ref{unconditional_convergence} it has already been shown that the assumptions of that theorem are stronger than those of Lemma \ref{AEC_case_C}. Hence, the assumptions of Theorem \ref{unconditional_convergence} along with assumption \hyperlink{C1}{C1} imply the assumptions of Lemma \ref{HEP_AEC_case_A} (use the fact that assumption \hyperlink{PM}{PM} is stronger that assumption \hyperlink{M2}{M2}; see Remark \ref{condition_PM}) whose conclusion implies that $\{\widetilde{\mathbb{G}}''\}_{N=1}^{\infty}$ is \hyperlink{conditional_AEC}{conditionally AEC} w.r.t. $\rho$. Since the first part of the conclusion of Lemma \ref{AEC_case_C} says that $\mathcal{F}$ is totally bounded w.r.t. $\rho$, it follows by 
Corollary \ref{cor_uniformly_continuous_limit_process} that $\{\widetilde{\mathbb{G}}''\}_{N=1}^{\infty}$ (and hence also the corresponding sequence of \hyperlink{HEP}{HEP}s) satisfies part (ii) of the conclusion of the present theorem for some $\mathbb{H}'$ which satisfies the conditions given in parts (i) and (iv). Part (iii) of the conclusion of the theorem follows now from Remark \ref{rem_measurable_suprema} (recall that condition \hyperlink{PM}{PM} implies that the suprema in Remark \ref{rem_measurable_suprema} are measurable).

Now, consider the assumptions of Theorem \ref{unconditional_convergence_case_B_C}. In the proof of Theorem \ref{unconditional_convergence_case_B_C} it has already been shown that its assumptions are stronger than those of Theorem \ref{opCWC_case_B_C}, and in the proof of the latter theorem it has been shown that its assumptions imply the conditions of Lemma \ref{marginal_convergence_case_C_B} whose conclusion says that conditions \hyperlink{A0}{A0}, \hyperlink{A1}{A1} and \hyperlink{A2}{A2} are satisfied. Use Remark \ref{condizione_C2} to conclude that the assumptions of Theorem \ref{unconditional_convergence_case_B_C} along with assumption \hyperlink{C1}{C1} imply the assumptions of Lemma \ref{marginal_convergence_HAJEK} whose conclusion says that the sequence of auxiliary processes $\{\widetilde{\mathbb{G}}''\}_{N=1}^{\infty}$ satisfies \hyperlink{CWCM}{CWCM} for some zero-mean Gaussian limit process with covariance function given by $\Sigma''$ as defined in (\ref{funzione_covarianza_Sigma_primo_primo}). Next, recall that in the proof of Theorem \ref{opCWC_case_B_C} it has been shown that its assumptions imply those of Lemma \ref{AEC_case_C_B}, and conclude that the assumptions of Theorem \ref{unconditional_convergence_case_B_C} along with condition \hyperlink{C1}{C1} must therefore imply the assumptions of Lemma \ref{HEP_AEC_case_B} (use the fact that assumption \hyperlink{PM}{PM} is stronger that assumption \hyperlink{M2}{M2}; see Remark \ref{condition_PM}) whose conclusion implies that $\{\widetilde{\mathbb{G}}''\}_{N=1}^{\infty}$ is \hyperlink{conditional_AEC}{conditionally AEC} w.r.t. $\rho_{w}$. Since the first part of the conclusion of Lemma \ref{AEC_case_C_B} says that $\mathcal{F}$ is totally bounded w.r.t. $\rho_{w}$, it follows by Corollary \ref{cor_uniformly_continuous_limit_process} that the sequence of auxiliary processes $\{\widetilde{\mathbb{G}}''\}_{N=1}^{\infty}$ (and hence also the corresponding sequence of \hyperlink{HEP}{HEP}s) satisfies part (ii) of the conclusion of the present theorem for some $\mathbb{H}'$ which satisfies the conditions given in parts (i) and (v). Again, part (iii) of the conclusion of the theorem follows now from Remark \ref{rem_measurable_suprema} (recall that condition \hyperlink{PM}{PM} implies that the suprema in Remark \ref{rem_measurable_suprema} are measurable).
\end{proof}

\begin{cor}[Joint weak convergence]\label{corollary_joint_weak_convergence_HEP}
Under the assumptions of Theorem \ref{opCWC_HEP} it follows that
\begin{itemize}
\item[(i) ]
\[(\mathbb{G}_{N}, \mathbb{G}_{N}'')\rightsquigarrow(\mathbb{G}, \mathbb{G}'')\text{ in }l^{\infty}(\mathcal{F})\times l^{\infty}(\mathcal{F}),\]
where $\mathbb{G}''$ is defined as in the conclusion of Theorem \ref{opCWC_HEP}, $\mathbb{G}_{N}$ is the classical $\mathcal{F}$-indexed empirical process defined in (\ref{def_classical_empirical_process}), and where $\mathbb{G}$ is a Borel measurable and tight $P_{y}$-Brownian Bridge which is independent from $\mathbb{G}''$.
\end{itemize}
\end{cor}

\begin{proof}
As already shown in the proof of Corollary \ref{corollary_joint_weak_convergence_C} (Corollary \ref{corollary_joint_weak_convergence_C_B}), the assumptions of Theorem \ref{unconditional_convergence} (Theorem \ref{unconditional_convergence_case_B_C}) imply that $\mathcal{F}$ is a $P_{y}$-Donsker class. The conclusion of the present corollary follows therefore from Theorem \ref{opCWC_HEP} and Theorem \ref{teorema_generale_joint_weak_convergence}.
\end{proof}

\section{Simulation results}\label{Simulation_results}

This section about simulation results is analogous to Section 5 in \citep{Pasquazzi_2019} (see also Appendix S4 in \citep{Bertail_2017}). The numerical results given in this section have been obtain by using the R Statistical Software \cite{citazioneR} in order to repeat $B=1000$ times the following steps:
\begin{itemize}
\item[1) ] Generate a population of $N$ independent observations $(Y_{i},X_{i})$ from the linear model $Y_{i}=X_{i}+U_{i}$, where the $X_{i}$'s are i.i.d. lognormal with $E(\ln X_{i})=0$ and $Var(\ln X_{i})=1$, and where the $U_{i}$'s are independent zero mean Gaussian random variables with $Var(U_{i})=X_{i}^{2}$, $i=1,2,\dots, N$.
\item[2) ] Select a sample $\mathbf{s}_{N}:=(s_{1,N}, s_{2,N}, \dots, s_{N,N})$ according to the CPS design with sample size $n$ (specified below) and with first order sample inclusion probabilities $\pi_{i,N}$ proportional to the $X_{i}$ values (this step was performed by using the function "UPmaxentropy" from the R package "sampling" \citep{citazione_sampling_package}).
\item[3) ] Compute the Horvitz-Thompson and the H{\'a}jek estimator for the population cdf $F_{Y,N}(t):=\sum_{i=1}^{N}I(Y_{i}\leq t)/N$, $t\in\mathbb{R}$, and compute the uniform distance between each of those estimators and $F_{Y,N}$, i.e. compute $\lVert \mathbb{G}'_{N}\rVert_{\mathcal{F}}$ and $\lVert \mathbb{G}''_{N}\rVert_{\mathcal{F}}$ for the case where $\mathcal{F}:=\{I(y\leq t):t\in\mathbb{R}\}$.
\item[4) ] Estimate the $\gamma$-quantiles $q_{\gamma}'$ and $q_{\gamma}''$ of the limiting distributions of $\lVert \mathbb{G}'_{N}\rVert_{\mathcal{F}}$ and $\lVert \mathbb{G}''_{N}\rVert_{\mathcal{F}}$, i.e. the $\gamma$-quantiles of the distributions of $\lVert \mathbb{G}'\rVert_{\mathcal{F}}$ and $\lVert \mathbb{G}''\rVert_{\mathcal{F}}$. This was done by using Algorithm 5.1 in \citep{Kroese} which was also used in the simulation study in \cite{Bertail_2017}. The details for the implementation of this algorithm are described below.
\item[5) ] Compute the asymptotic uniform $\gamma$-confidence bands for the population cdf $F_{Y,N}$ based on the Horvitz-Thompson and the H{\'a}jek estimators and verify whether $F_{Y,N}$ lies within these confidence bands, i.e. verify whether $\lVert \mathbb{G}'_{N}\rVert_{\mathcal{F}}\leq \widehat{q}_{\gamma}'$ and whether $\lVert \mathbb{G}''_{N}\rVert_{\mathcal{F}}\leq \widehat{q}_{\gamma}''$, where $\widehat{q}_{\gamma}'$ and $\widehat{q}_{\gamma}''$ are the estimates of $q_{\gamma}'$ and $q_{\gamma}''$ obtained from step 4. Note that the widths of the two asymptotic uniform $2\gamma$-confidence bands for $F_{Y,N}$ are given by $2\widehat{q}_{\gamma}'N^{-1/2}$ and $2\widehat{q}_{\gamma}''N^{-1/2}$, respectively.  
\end{itemize}

The $\gamma$-quantiles of the distributions of $\lVert \mathbb{G}'\rVert_{\mathcal{F}}$ and $\lVert \mathbb{G}''\rVert_{\mathcal{F}}$ were estimated according to the following procedure (see Algorithm 5.1 in \citep{Kroese}):
\begin{itemize}
\item[i) ] Estimate the covariance matrices $\Sigma'(\mathbf{f})$ and $\Sigma''(\mathbf{f})$ for $\mathbf{f}:=(I(y\leq Y_{i_{1}}), I(y\leq Y_{i_{2}}), \dots, I(y\leq Y_{i_{n}}))^{\intercal}$ where $(i_{1}, i_{2}, \dots, i_{n})$ correspond to the sampled population units, i.e. $(i_{1}, i_{2}, \dots, i_{n})$ are the values of the subscript $i$ for which $s_{i,N}=1$, $i=1,2,\dots, N$. The components $\Sigma'_{r,c}(\mathbf{f})$ and $\Sigma''_{r,c}(\mathbf{f})$, $r,c=1,2,\dots, n$, of the two covariance matrices were estimated as follows (cf. Lemma \ref{marginal_convergence_case_C_B} and Lemma \ref{marginal_convergence_HAJEK}):
\[\widehat{\Sigma}'_{r,c}(\mathbf{f}):=\frac{1}{N}\sum_{k=1}^{N}s_{k,N}\frac{1-\pi_{k,N}}{\pi_{k,N}^{2}}Z_{k,i_{r}}Z_{k,i_{c}}\]
with
\[Z_{k,i_{r}}:=I(Y_{k}\leq Y_{i_{r}})-\pi_{k,N}\frac{\sum_{k=1}^{N}\frac{s_{k,N}}{\pi_{k,N}}I(Y_{k}\leq Y_{i_{r}})(1-\pi_{k,N})}{\sum_{k=1}^{N}\pi_{k,N}(1-\pi_{k,N})},\quad k=1,2,\dots, N;\]
and
\[\widehat{\Sigma}''_{r,c}(\mathbf{f}):=\frac{1}{\sum_{k=1}^{N}\frac{s_{k,N}}{\pi_{k,N}}}\sum_{k=1}^{N}s_{k,N}\frac{1-\pi_{k,N}}{\pi_{k,N}^{2}}[Z_{k,i_{r}}-\overline{Z}_{i_{r}}][Z_{k,i_{c}}-\overline{Z}_{i_{c}}]\]
with
\[\overline{Z}_{i_{r}}:=\frac{1}{\sum_{k=1}^{N}\frac{s_{k,N}}{\pi_{k,N}}}\sum_{k=1}^{N}\frac{s_{k,N}}{\pi_{k,N}}Z_{k,i_{r}}.\]
\item[ii) ] Compute the Cholesky decompositions of the estimated covariance matrices, i.e. compute two lower triangular matrices $L$ and $H$ such that $\widehat{\Sigma}'(\mathbf{f})=LL^{\intercal}$ and $\widehat{\Sigma}''_{i,j}(\mathbf{f})=HH^{\intercal}$.
\item[iii) ] Generate independently $1000$ random vectors $\mathbf{Z}_{b}:=(Z_{1,b}, Z_{2,b}, \dots, Z_{n,b})^{\intercal}$, $b=1,2,\dots, 1000$, whose components $Z_{k,b}$ are i.i.d. standard normal random variables and compute the vectors $\mathbf{G}_{b}':=L\mathbf{Z}_{b}$ and $\mathbf{G}_{b}'':=H\mathbf{Z}_{b}$ which can be considered as realizations of the limit processes $\mathbb{G}'$ and $\mathbb{G}''$, respectively.
\item[iv) ] for each $b=1,2,\dots, 1000$ compute the maximum norms $\lVert \mathbf{G}_{b}'\rVert_{\infty}$ and $\lVert \mathbf{G}_{b}''\rVert_{\infty}$ (i.e. the two maxima of the absolute values of the components of $\mathbf{G}_{b}'$ and $\mathbf{G}_{b}''$), put the two vectors $(\lVert \mathbf{G}_{1}'\rVert_{\infty}, \lVert \mathbf{G}_{2}'\rVert_{\infty}, \dots, \lVert \mathbf{G}_{1000}'\rVert_{\infty})$ and $(\lVert \mathbf{G}_{1}''\rVert_{\infty}, \lVert \mathbf{G}_{2}''\rVert_{\infty}, \dots, \lVert \mathbf{G}_{1000}''\rVert_{\infty})$ in ascending order and put $\widehat{q}_{\gamma}'$ equal to the $\gamma$-quantile of the first vector, and put $\widehat{q}_{\gamma}''$ equal to the $\gamma$-quantile of the second vector.
\end{itemize}

\begin{table}[h]\label{tabella_simulazioni}
\caption{Simulation results for the Horvitz-Thompson empirical process.}
\label{tabella_simulazioni_HTEP}
\begin{tabular}{rccc}
\hline
 & $\gamma=0.90$ & $\gamma=0.95$ & $\gamma=0.99$\\
\hline
$\mathbf{N=500}$ \\
\multirow{ 2}{*}{$\alpha=0.05$} & 0.841 &0.916 &0.983\\
&(0.4233; 0.4541)&(0.4755; 0.5121)&(0.5779; 0.6481)\\
\multirow{ 2}{*}{$\alpha=0.10$} & 0.866 &0.913 &0.990\\
&(0.3022; 0.3185)&(0.3378; 0.3610)&(0.4079; 0.4662)\\
$\mathbf{N=1000}$ \\
\multirow{ 2}{*}{$\alpha=0.05$} & 0.877 &0.937 &0.991\\
&(0.3102; 0.3264)&(0.3470; 0.3676)&(0.4198; 0.4656)\\
\multirow{ 2}{*}{$\alpha=0.10$} & 0.875& 0.928& 0.981\\
&(0.2190; 0.2310)&(0.2444; 0.2602)&(0.2942; 0.3296)\\
$\mathbf{N=2000}$ \\
\multirow{ 2}{*}{$\alpha=0.05$} & 0.873& 0.938& 0.993\\
&(0.2247; 0.2362)&(0.2509; 0.2669)&(0.3019; 0.3374)\\
\multirow{ 2}{*}{$\alpha=0.10$} & 0.885& 0.949& 0.991\\
&(0.1574; 0.1666)&(0.1755; 0.1903)&(0.2110; 0.2384)\\
\hline
\end{tabular}
\end{table}

\begin{table}[h]\label{tabella_simulazioni}
\caption{Simulation results for the H{\'a}jek empirical process.}
\label{tabella_simulazioni_HEP}
\begin{tabular}{rccc}
\hline
 & $\gamma=0.90$ & $\gamma=0.95$ & $\gamma=0.99$\\
\hline
$\mathbf{N=500}$ \\
\multirow{ 2}{*}{$\alpha=0.05$} & 0.860& 0.928& 0.991\\
&(0.4319; 0.4562)&(0.4838; 0.5100)&(0.5864; 0.6480)\\
\multirow{ 2}{*}{$\alpha=0.10$} & 0.875& 0.925& 0.989\\
&(0.3062; 0.3203)&(0.3418; 0.3652)&(0.4127; 0.4523)\\
$\mathbf{N=1000}$ \\
\multirow{ 2}{*}{$\alpha=0.05$} & 0.887& 0.940& 0.993\\
&(0.3145; 0.3344)&(0.3514; 0.3785)&(0.4243; 0.4737)\\
\multirow{ 2}{*}{$\alpha=0.10$} & 0.880& 0.935& 0.982\\
&(0.2212; 0.2332)&(0.2465; 0.2623)&(0.2962; 0.3313)\\
$\mathbf{N=2000}$ \\
\multirow{ 2}{*}{$\alpha=0.05$} & 0.878& 0.944& 0.994\\
&(0.2269; 0.2416)&(0.2529; 0.2686)&(0.3049; 0.3387)\\
\multirow{ 2}{*}{$\alpha=0.10$} & 0.886& 0.948& 0.994\\
&(0.1585; 0.1684)&(0.1765; 0.1899)&(0.2116; 0.2370)\\
\hline
\end{tabular}
\end{table}

Table \ref{tabella_simulazioni_HTEP} (for the \hyperlink{HTEP}{HTEP}) and Table \ref{tabella_simulazioni_HEP} (for the \hyperlink{HEP}{HEP}) summarize the simulation results. For each considered population size $N=500, 1000, 2000$, for each considered sampling fraction $\alpha:=n/N=0.05, 0.10$ and for each considered confidence level $\gamma=0.90, 0.95, 0.99$ the two tables report the estimate of the coverage probability of the corresponding confidence band for $F_{Y,N}$ as well as the average width (the first figure within each bracket) and the maximum width (the second figure within each bracket) of the $B=1000$ simulated confidence bands. The simulation results suggest that the confidence bands based on the \hyperlink{HTEP}{HTEP} and on the \hyperlink{HEP}{HEP} are very similar. Their coverage accuracy is quite precise for the populations of size $N\geq 1000$ and seems not to depend on the sampling fraction $\alpha$. As expected, the with of the confidence bands is roughly proportional to $n^{-1/2}$ and it appears to be quite stable from sample to sample (the differences between the maximum widths and the average widths are rather small). However, for many applications the widths of the confidence bands might be too large. This problem can be probably overcome through alternative estimators which use the information provided by the auxiliary variable $X$ more efficiently (see e.g. \citep{Rueda_2010} or \citep{Pasquazzi_DeCapitani_2016} and references therein).

\bibliographystyle{abbrvnat}

\bibliography{arXiv_FCLTs_for_CPS_designs}   

\end{document}